\documentclass[11pt]{article}

\usepackage{amsthm}
\usepackage{amsmath}
\usepackage{amsfonts, epsfig, amsmath, amssymb, color, amscd}
\usepackage{pb-diagram, pb-xy}
\usepackage{amssymb,epsfig}
\usepackage{color}
\usepackage{pb-diagram, pb-xy}
\usepackage{amssymb,epsfig,amsfonts}

\usepackage{tikz}
\usetikzlibrary{matrix,arrows,calc,snakes,patterns,decorations.markings}

\usepackage[all, arc]{xy}

\usepackage[linktocpage=true,plainpages=false,pdfpagelabels=false]{hyperref}

\newfont{\sdbl}{msbm9}
\newfont{\dbl}{msbm10 at 12pt}
\theoremstyle{definition}
\newcommand{\da}{{\mbox{\dbl A}}}

\newcommand{\dpp}{{\mbox{\dbl P}}}
\newcommand{\dz}{{\mbox{\dbl Z}}}

\newcommand{\dn}{{\mbox{\dbl N}}}

\newcommand{\sdz}{{\mbox{\sdbl Z}}}
\newcommand{\sdn}{{\mbox{\sdbl N}}}
\newcommand{\sdp}{{\mbox{\sdbl P}}}
\newcommand{\dc}{{\mbox{\dbl C}}}
\newcommand{\dq}{{\mbox{\dbl Q}}}

\newcommand{\ord}{\mathop{\rm ord}\nolimits}

\newcommand{\gr}{\mathop {\rm gr}}

\newcommand{\End}{\mathop {\rm End}}
\newcommand{\Adm}{\mathop {\rm Adm}}

\newcommand{\LT}{\mathop {\rm LT}}

\newcommand{\Sup}{\mathop {\rm Supp}}
\newcommand{\Quot}{\mathop {\rm Quot}}

\newcommand{\rk}{\mathop {\rm rk}}

\newcommand{\trdeg}{\mathop {\rm trdeg}}

\newcommand{\Spec}{\mathop {\rm Spec}}

\newcommand{\Proj}{\mathop {\rm Proj}}

\newcommand\limind{\mathop{\underrightarrow{\lim}}}

\newcommand{\Hom}{\mathop{\rm Hom}}
\newcommand{\idm}{\mathfrak{m}}
\newcommand{\cee}{{{\cal E}}}

\makeatother
\newtheorem{defin}{Definition}[section]
\newtheorem{rem}{Remark}[section]

\newtheorem{ex}{Example}[section]

\theoremstyle{plain}
\newtheorem{prop}{Proposition}[section]
\newtheorem{thm}{Theorem}[section]
\newtheorem{lem}{Lemma}[section]
\newtheorem{cor}{Corollary}[section]

\numberwithin{equation}{section}

\def\Hom{{{\rm Hom }}}

\def\Spec{{{\rm Spec \,}}}

\def\Proj{{{\rm Proj \,}}}

\def\Quot{{{\rm Quot \,}}}

\def\dim{{{\rm dim \,}}}

\def\ker{{{\rm ker \,}}}

\newcommand{\Ord}{\mathop {\rm \bf ord}}

\newcommand{\co}{{{\cal O}}}
\newcommand{\cf}{{{\cal F}}}

\newcommand{\cm}{{{\cal M}}}

\newcommand{\cd}{{{\cal D}}}

\title{Schur-Sato theory for quasi-elliptic rings}
\author{Alexander Zheglov\footnote{This work was supported by RSF grant no. 22-11-00272}}

\date{}

\begin{document}

\maketitle

%\quad \qquad \qquad \qquad
\centerline{\em Dedicated to the memory of A.N. Parshin}

\begin{abstract}
The notion of quasi-elliptic rings appeared as a result of an attempt to classify a wide class of commutative rings of operators found in the theory of integrable systems, such as rings of commuting differential, difference, differential-difference, etc. operators. They are contained in a certain non-commutative "universal" ring - a purely algebraic analogue of the ring of pseudodifferential operators on a manifold, and admit (under certain mild restrictions) a convenient algebraic-geometric description. An important algebraic part of this description is the Schur-Sato theory - a generalisation of the well known theory for ordinary differential operators. Some parts of this theory were developed earlier in a series of papers, mostly for dimension two.

In this paper we present this theory in arbitrary dimension. We apply this theory to prove two classification theorems of quasi-elliptic rings in terms of certain pairs of subspaces (Schur pairs). They are necessary for the algebraic-geometric description of quasi-elliptic rings mentioned above. 

The theory is effective and has several other applications, among them is a new proof of the Abhyankar inversion formula.

\end{abstract}

\tableofcontents

%{\bf Keywords:} 

\section{Introduction}

This work appeared as an attempt to classify a wide class of commutative
rings of operators found in the theory of (quantum) integrable systems, such as
rings of commuting differential, difference,  differential-difference, etc. operators. Typical examples are quantum integrable systems (\cite{BEG}, \cite{Krichever77}, \cite{ChalykhVeselov90}, \cite{Ch}, \cite{ChalykhVeselov93}, \cite{ChalykhVeselov98}, \cite{BerestEtigofGinzburg03}, etc.) and their isospectral deformations (\cite{BurbanZheglov2017}), commutative rings of difference or differential-difference operators (\cite{Mumford}, \cite{MoerbMumford}, \cite{KricheverNovikov2003}, \cite{MironovMaul}, \cite{MironovMaul1}, \cite{Heckman91}, etc.).

As a result a notion of {\it quasi-elliptic rings} appeared (see definition \ref{D:quasi-elliptic}). 
All such rings are contained in a certain non-commutative  "universe" ring $\hat{D}_n^{sym}$ (see section \ref{S:algebra_hat_D_n}) - a purely
algebraic analogue of the ring of pseudodifferential operators on a manifold, and admit
(under certain mild restrictions) a convenient algebraic-geometric description. An important algebraic part of this description is the {\it Schur-Sato theory} - a generalisation of the well known theory for ordinary differential operators, see the works \cite{Mul}, \cite{Mulase2} (our work was largely inspired by these articles), see also \cite[Ch. 9]{Zheglov_book} for a detailed exposition of this theory adopted to our generalization. The Schur-Sato theory is also one of steps toward the {\it higher dimensional Krichever correspondence} -- a higher-dimensional analogue of the well-known and fruitful interrelation between KdV- or, more general, KP-equations, and algebraic curves with additional geometric data on the other side. This theory describes a one-to-one correspondence between equivalence classes of objects in the following diagram: 
$$
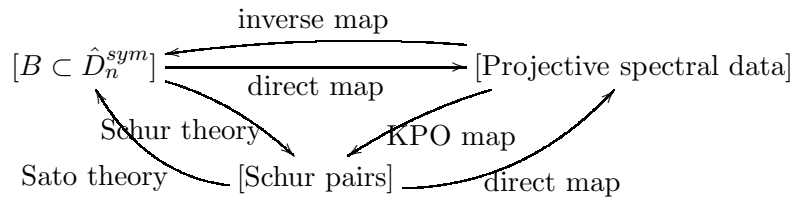
\begin{array}{c}
\xymatrix{
[B\subset \hat{D}_n^{sym}]  \ar[rr]_-{\mbox{direct map}}  \ar@/^10pt/[rd]_-{\mbox{Schur theory}} & & \mbox{[Projective spectral data]} \ar@/_10pt/[ll]_-{\mbox{inverse map}} \ar@/_10pt/[ld]^-{\mbox{KPO map}}\\
& \mbox{[Schur pairs]} \ar@/_20pt/[ru]_-{\mbox{direct map}} \ar@/^20pt/[lu]^-{\mbox{Sato theory}} & \\
}
\end{array}
$$
Here $B$ are quasi-elliptic rings, and other notions like Schur pairs, Schur theory and Sato theory are explained in the relevant sections of this paper. The correspondence between the equivalence classes of quasi-elliptic rings $B$ and projective spectral data was announced in the paper \cite{Zheglov_belovezha}, though without details. These details as well as other parts of this diagram will be explained in a subsequent paper \cite{Zheglov2020}. The present paper contains necessary preliminary results for \cite{Zheglov2020}. Besides, we expect the results from this paper will be also useful for various applications, e.g. for explicit calculations of examples. 

The brief history of the Krichever correspondence and its higher-dimensional analogues is follows. In case $n=1$ (the classical Krichever correspondence) the theory is well known and rich, in particular there is a well known extensive classification theory of commutative rings of ordinary differential operators, see \cite{Wall}, \cite{Schur}, \cite{BC1}-\cite{BC3}, \cite{Baker}, \cite{Krichever77}, \cite{Krichever78}, \cite{Dr}, \cite{Mumford}, \cite{SegalWilson}, \cite{Verdier}, \cite{Sato}, \cite{SN}, \cite{SS}, \cite{Man}, \cite{Mul}, \cite{Previato}, etc., see also the book \cite{Zheglov_book}, where a detailed exposition of this classification is given, for more references. 

In case $n=2$ (and partially for general $n$, e.g. the KPO map (Krichever- Parshin-Osipov map) for certain geometric data in \cite{Osipov}) a higher-dimensional analogue of the Krichever correspondence was initiated in the works \cite{Parshin2000}, \cite{Parshin2001}, \cite{Pa} and then developed in the works \cite{Osipov}, \cite{ZheglovKP}, \cite{OsipovZheglov}, \cite{KOZ2009}, \cite{Ku1}, \cite{Zheglov2013}, \cite{KOZ2014}, \cite{KurkeZheglov}, \cite{Zheglov2018} and related works \cite{BurbanZheglov2017}, \cite{Zheglov_belovezha}, \cite{Kulikov}. In particular, one-to one correspondences in the diagram above are established  and some examples are calculated for $n=2$ in a series of works \cite{Zheglov2013}, \cite{KOZ2014}, \cite{KurkeZheglov}, \cite{Zheglov2018}, \cite{BurbanZheglov2017}, \cite{Kulikov}. 

In this paper we improve and extend the corresponding results from the $n=2$ case to the general case. 

Starting with the ring of partial differential operators $D_n=K[[x_1, \ldots ,x_n]][\partial_1, \ldots , \partial_n]$ (PDO for short), where $K$ is a field of characteristic zero, one can define its completion. There are several versions of completion. In \cite{Zheglov2013}, \cite{Zheglov2018} a non-symmetric version $\hat{D}_2$ (for $n=2$) was used. It was shown there that commutative subrings $B\subset \hat{D}_2$ satisfying certain mild conditions are classified in terms of certain geometric spectral data, and this classification is, in a sense, a natural generalisation of the Krichever classification of commuting ordinary differential operators (ODO for short). Recall that the classification of commuting ODOs is especially simple for  subrings of {\it rank one}, because any rank one commutative subalgebra of ODOs is essentially determined (up to a linear change of variables) by purely geometric spectral data consisting of a projective curve, a regular point on the curve, and a coherent torsion free sheaf of rank one over the curve, see e.g. \cite[Th. 10.26]{Zheglov_book} for detailed explanations of this case. The classification of commutative subrings in $\hat{D}_2$ of {\it rank one} (for an appropriately defined notion of rank) possess the analogous property, see \cite{Zheglov2018}, and the  refined classification of rank one quasi-elliptic rings in $\hat{D}_n$ was announced in the paper \cite{Zheglov_belovezha}. One of the main aims of this paper is to give a detailed proof of preliminary classification theorems of quasi-elliptic rings of rank one and even of rank $r$ in terms of Schur pairs (see the explanation below). They are necessary for the proof of theorem from \cite{Zheglov_belovezha}.

The advantage of the non-symmetric version $\hat{D}_n$  of the completion (see section \ref{S:algebra_D_n}) is the existence of the Schur theory for the related ring of formal pseudo-differential operators $\hat{E}_n$ (see sections \ref{S:algebra_D_n}, \ref{S:Schur_theory}  below) --  an important tool of the classification theory. On the other hand, there is a "symmetric" version of completion $\hat{D}_n^{sym}$ introduced first in \cite[Def. 5.1]{BurbanZheglov2017}, which is more convenient in some cases (in particular, important explicit examples from \cite{BurbanZheglov2017} were calculated in this ring), and which contains the non-symmetric one. We recall its definition and the basic properties of this ring (with slight improvements) in section \ref{S:algebra_hat_D_n}. The ring $\hat{D}_n^{sym}$ can be thought of as a simple purely algebraic analogue of the algebra of (analytic) pseudodifferential operators on a manifold, cf. \cite{Shubin}.\footnote{Although congenial constructions of rings have been encountered in various areas of mathematics, from analysis on manifolds to deformation quantization, quantum groups and number theory (cf. \cite{Fontaine}, cf. also remark \ref{R:Birkhoff} below), we could not find a more or less exact analogue of our construction in the literature. } 
At the end of section \ref{S:algebra_hat_D_n} we give one unexpected application of this construction - a new proof of  the Abhyankar inversion formula from \cite{Ab}, \cite[III.2]{BCR}. 

The extension of the classical Schur theory (cf. \cite{Schur}, \cite{Mul}) to the rings $\hat{D}_n$, $\hat{E}_n$ is given in section \ref{S:Schur_theory}. Some elements of this theory appeared before in \cite{Parshin2000} for the rings of pseudo-differential operators in $n$ variables and in \cite{Zheglov2013} for the completed ring $\hat{D}_2$. It is an important algebraic tool in studying basic algebraic properties of commuting operators, we deduce them in section \ref{S:algebraic_properties}.

The extension of the classical Sato theory (cf. \cite{SN}, \cite{Mul}) is given in section \ref{S:Sato_theory}. Some elements of this theory appeared before in \cite{Zheglov2013} for the case $n=2$, and in \cite[\S 5]{BurbanZheglov2017} for general $n$, but only for special subspaces. In particular, the Sato operators considered in loc. sit. belonged to the ring  $\hat{D}_n^{sym}$. On the other hand, the subspaces and Sato operators from \cite{Zheglov2013} were of different nature (in particular, they never belong to $\hat{D}_n^{sym}$). To join formally different previous versions of Sato theories we develop a unified Sato theory in general case, following the exposition of \cite{BurbanZheglov2017}. Now subspaces and Sato operators belong to a wider space (more precisely, a bimodule) $\hat{\Pi}_n$ introduced in section \ref{S:algebra_D_n}. As a corollary of this new unified Sato theory we get a notion of Schur pairs from the diagram above and derive two constructions connecting them with the quasi-elliptic rings of operators $B\subset \hat{D}_n^{sym}$, see section \ref{S:Schur_pairs}. Another application is a new short proof (though not effective) of the generalised Birkhoff  decomposition theorem by M. Mulase \cite{Mulase_inv}, see example \ref{R:Birkhoff}. 

In section \ref{S:general_Schur_theory} we develop another analogue of the Schur theory, now for the ring  $\hat{D}_n^{sym}$. The difference between this Schur theory and the Schur theory from section \ref{S:Schur_theory} is essential: the Schur operators from section \ref{S:Schur_theory} belong to the ring of formal pseudo-differential operators $\hat{E}_n$ (which contains a formal inverse to a derivation), and the Schur operators from section  \ref{S:general_Schur_theory} belong to the ring of "pseudodifferential" operators $\hat{D}_n^{sym}$ (which contains no formal inverses to the derivations). Besides, the difference between rings $\hat{E}_n$ and $\hat{D}_n^{sym}$ is  essential  already in the case $n=1$: in this case the ring $\hat{E}_1$ has no zero divisors, unlike the ring $\hat{D}_1^{sym}$. The Schur theory from section \ref{S:general_Schur_theory} is used for the proof of the refined classification of rank one quasi-elliptic rings in section \ref{S:classification} and for a description (in the same section) of admissible operators from section \ref{S:admissible_op}. Besides, it is more convenient for explicit calculations than the other Schur theory and is necessary for a number of results in the subsequent paper \cite{Zheglov2020}, cf. remark \ref{R:purity}. 

\smallskip 

The paper is organized as follows. 

In section \ref{S:algebra_hat_D_n} we introduce the definition and the basic properties of the ring $\hat{D}_n^{sym}$ and its order function $\Ord$, which is used extensively in the whole paper. At the end of section we give a new proof of the the Abhyankar inversion formula. 

In section \ref{S:algebra_D_n} we introduce the definition and  basic properties of the rings $\hat{D}_n$, $\hat{E}_n$, $E_n$, $\Pi_n$ and of the bimodule $\hat{\Pi}_n$, which are used later in the generalized Schur and Sato theories. Besides, we introduce the important technical notion of {\it slice decomposition} of an operator. 

In section \ref{S:spectral} we recall and extend the notion of the {\it spectral module}, and in section \ref{S:quasi-elliptic} we define another important order function $\ord_{\Gamma}$ and then define the notion of quasi-elliptic rings. The equivalence relations between these rings are introduced later, in section \ref{S:classification}. 

In section \ref{S:Schur_theory} we develop the extension of the classical Schur theory  to the rings $\hat{D}_n$, $\hat{E}_n$ as it was mentioned above. We follow standard steps of the classical theory from \cite{Mul}, introducing the notions of Schur operators, roots, normalized operators and admissible operators for our extended and completed rings. As an application of the Schur theory we derive basic algebraic properties of quasi-elliptic rings in section \ref{S:algebraic_properties}.

In section \ref{S:Sato_theory} we develop the extension of the classical Sato theory for the bimodule $\hat{\Pi}_n$. In this section are defined such important notions as regular operators (section \ref{S:regular_op}), Sato action and Sato operators (section \ref{S:Sato_theorem}), Schur pairs and (analytical) rank of quasi-elliptic rings and of Schur pairs (section \ref{S:Schur_pairs}). As an application of the theory we get a description of units (section \ref{S:Sato_theorem}), and derive two constructions connecting Schur pairs with the quasi-elliptic rings of operators (section \ref{S:Schur_pairs}). In section \ref{S:admissible_change} we introduce the notion of admissible linear changes of variables, which is used in the refined classification of rank one quasi-elliptic rings, and prove important technical statements necessary for the proof of this classification in section \ref{S:classification}. 

In section \ref{S:general_Schur_theory} we develop another analogue of the Schur theory for the ring  $\hat{D}_n^{sym}$ as it was mentioned above. This section is divided into 3 parts, where first we give a description of centralizers of operators with constant coefficients (section \ref{S:centralizers}), then we describe this theory for the case $n=1$ (section \ref{S:Schur,n=1}; in this case the results are simpler and serve as an induction step for the next section), and at last describe it in general case (section \ref{S:Schur,n}). 

In section \ref{S:classification} we combine all technique together to prove two classification theorems: one is for quasi-elliptic rings of any rank, and another is the refined classification for quasi-elliptic rings of rank one. We would like to emphasize that the definition of rank in this paper is weaker than in \cite{Zheglov2018}, \cite{Zheglov_belovezha}: it is just the analytical rank from these works, without extra conditions like finite generation of the spectral module, strongly admissibility of  quasi-elliptic ring, etc. 

\smallskip

In this paper all rings are assumed to be {\it rings over $K$}, i.e. associative $K$-algebras with the multiplicative unity $1$, where $K$ is a field of characteristic zero. Most results are proved even for a more general relative case, i.e. for rings over $K_y=K[[y_1, \ldots ,y_m]]$, $m\ge 0$. For reader's convenience we tried to keep the exposition of this paper self-contained. 

{\bf Acknowledgements.} I am grateful to  S. Gorchinskiy and D. Osipov for their interest and stimulating discussions. I'm also grateful to participants of the seminar on non-commutative geometry at MSU, especially for A. Arutyunov, A.S. Mischenko,  F. Popelenskiy, G. Sharygin for their attention and stimulating questions. I'd also like to thank the organizers of the conferences "Dynamics in Siberia-21" from Novosibirsk and "Frontier of Differential Geometry - 2021" from the Sino-Russian Mathematical Center in Beijing, where first results of this paper were presented, for warm and stimulating atmosphere.

\subsection{List of notations}

For reader's convenience we make an account of the most important notations used in this paper. 

1. $K_y:=K[[y_1, \ldots y_m]]$, $m\ge 0$, $\hat{R}_y:=K_y [[x_1,\ldots ,x_n]]$, the $K$-vector space 
$$
\cm_n := \hat{R}_y [[\partial_1, \dots, \partial_n]] = \left\{
\sum\limits_{\underline{k} \ge \underline{0}} a_{\underline{k}} \underline{\partial}^{\underline{k}} \; \left|\;  a_{\underline{k}} \in \hat{R}_y \right. \;\mbox{for all}\;  \underline{k} \in \dn_0^n
\right\},
$$
$\upsilon:\hat{R}_y\rightarrow \dn_0\cup \infty$ --  the discrete valuation defined by the unique maximal ideal $\idm = (y_1,\ldots ,y_m; x_1, \dots, x_n)$ of $\hat{R}_y$,  \\
for any element
$
0\neq P := \sum\limits_{\underline{k} \ge \underline{0}} a_{\underline{k}} \underline{\partial}^{\underline{k}} \in \cm_n
$
$$
\Ord (P) := \sup\bigl\{|\underline{k}| - \upsilon(a_{\underline{k}}) \; \big|\; \underline{k} \in \dn_0^n \bigr\} \in \dz \cup \{\infty \},
$$
$$
\hat{D}_n^{sym}:=\bigl\{Q \in \cm \,\big|\, \Ord (Q) < \infty \bigr\};
$$
$
P_m:= \sum\limits_{ |\underline{i}| - |\underline{k}| = m} \alpha_{\underline{k}, \underline{i}} \,  \underline{x}^{\underline{i}} \underline{\partial}^{\underline{k}}
$ -- the $m$-th \emph{homogeneous component} of $P$,\\
$\sigma (P):=P_{\Ord (P)}=P_{-d}$ -- the highest symbol.

2. $\hat{D}_n=\hat{D}_{n}^{n}[\partial_n]$, where 
$\hat{D}_{n}^{n}$ denotes the subring in $\hat{D}_n^{sym}$ consisting of operators {\it not depending on $\partial_n$},\\
$\hat{E}_n=\hat{D}_{n}^{n}((\partial_n^{-1}))$, $E_n=D_{n}^n((\partial_n^{-1}))$, $\Pi_n =\{P\in \hat{E}_n| \quad \Ord (P)<\infty\}$, $\hat{\cm}_n:=\{\sum_{i\in \sdz}a_{i}\partial_n^i, \quad a_i\in \hat{D}_n^n\},$
$$
\hat{\Pi}_n:=\{ P\in \hat{\cm}_n| \quad \Ord (P)<\infty \}
$$
-- the left $\hat{D}_n^{sym}$-module and right $\Pi_n$-module;\\
the  {\it order function} $\ord_n$  on $\hat{E}_n$ is defined as 
$
\ord_n(P)=l
$
if $\hat{E}_n\ni P=\sum_{s=-\infty}^lp_s\partial_n^s$.

3. The first {\it slice decomposition} of elements from $\hat{\Pi}_n$: 
$$
P = \sum\limits_{\underline{i} \ge \underline{0}} \frac{\underline{x}^{\underline{i}}}{\underline{i}!} \, P_{(\underline{i})}, \quad \mbox{\rm where} \quad \underline{i}\in \dn_0^n\times \dn_0^m, \quad \underline{x}^{\underline{i}}=x_1^{i_1}\ldots x_n^{i_n}y_1^{j_1}\ldots y_m^{j_m},\quad \underline{i}! =i_1!\ldots i_n!,
$$
$$
P_{(\underline{i})} = \underline{i}! \sum_{\substack{\underline{k} \in \sdn_0^{n-1}\times\sdz \\ |\underline{k}| - |\underline{i}| \le d = \Ord (P)}} \alpha_{\underline{k}, \underline{i}} \, \underline{\partial}^{\underline{k}}, \quad \alpha_{\underline{k}, \underline{i}}\in K \quad \mbox{-- slice}.
$$

4. $F:=\hat{D}_n^{sym}/(x_1,\ldots ,x_n) \hat{D}_n^{sym}\simeq K_y[[\partial_1,\ldots ,\partial_n]]\cap \Pi_n$ -- the {\it spectral module} of a commutative ring $B\subset \hat{D}_n^{sym}$.

5. The $\Gamma$-order and quasi-elliptic rings are defined in definitions \ref{D:G-order} and \ref{D:quasi-elliptic}. 

6. $V_n:= K_y\{\{\partial_1, \ldots , \partial_{n-1}\}\}((\partial_n^{-1}))$, where $K_y\{\{\partial_1, \ldots , \partial_{n-1}\}\}=K_y[[\partial_1, \ldots , \partial_{n-1}]]\cap \hat{D}_n^n$. 
It has a structure of a right $\hat{E}_n$-module via the isomorphism of vector spaces $V_n\simeq \hat{E}_n/(x_1, \ldots ,x_n) \hat{E}_n$.

The Schur pairs $A,W\subset V_n$, the notion of rank for them and for quasi-elliptic rings are defined in definitions \ref{D:sch} and \ref{D:an-alg-rank}. 

7. The second {\it slice decomposition} of elements from $\hat{\Pi}_n$:
$$
P = \sum\limits_{\underline{i} \ge \underline{0}} \frac{\underline{x}^{\underline{i}}}{\underline{i}!} \, P_{(\underline{i})}, \quad \mbox{\rm where} \quad \underline{x}^{\underline{i}}=x_1^{i_1}\ldots x_n^{i_n}, \quad  i!=i_1!\ldots i_n!
$$
$$
P_{(\underline{i})} = \underline{i}! \sum_{\substack{\underline{k} \in \sdn_0^{n-1}\times\sdz \\ |\underline{k}| - |\underline{i}| \le d = \Ord (P)}} \alpha_{\underline{k}, \underline{i}} \, \underline{\partial}^{\underline{k}} \in V_n\cap \Pi_n, \quad \alpha_{\underline{k}, \underline{i}}\in K_y
\quad \mbox{-- slice}.
$$

The partial slice decomposition: $P=\sum_{q\ge 0}P_{[q]}$, where 
$$
P_{[q]}:= \sum_{\underline{k}\in \sdn_0^n, |\underline{k}|=q} 
\frac{x^{\underline{k}}}{\underline{k}!}P_{(\underline{k})}.
$$ 
For an operator $P\in \hat{\Pi}_n$ we denote by $\bar{P}:=P|_{y=0}$. $\tilde{K}$ denotes an extension of $K$ by some roots of unity.

\section{The ring $\hat{D}_n^{sym}$ and its order function $\Ord$}
\label{S:algebra_hat_D_n}

Denote by $K_y:=K[[y_1, \ldots y_m]]$, $m\ge 0$, the ring of formal power series over $K$ (if $m=0$, we set $K_y=K$).  
In this subsection we define a symmetric version of completion of the algebra of PDOs $D_n=K_y[[x_1,\ldots ,x_n]][\partial_1,\ldots ,\partial_n]$. It can be thought of as a simple purely algebraic analogue of the algebra of (analytic) pseudodifferential operators on a manifold. Such a completion appeared first in the paper \cite{BurbanZheglov2017}, and we give here an improved exposition of its properties. 
We'll use here a slightly different notation than in \cite[\S 5]{BurbanZheglov2017}.

Denote $\hat{R}_y:=K_y [[x_1,\ldots ,x_n]]$. Consider the $K$-vector space 
$$
\cm_n := \hat{R}_y [[\partial_1, \dots, \partial_n]] = \left\{
\sum\limits_{\underline{k} \ge \underline{0}} a_{\underline{k}} \underline{\partial}^{\underline{k}} \; \left|\;  a_{\underline{k}} \in \hat{R}_y \right. \;\mbox{for all}\;  \underline{k} \in \dn_0^n
\right\},
$$
where $\underline{k}$ is the multi-index, $\underline{\partial}^{\underline{k}}=\partial_1^{k_1}\ldots \partial_n^{k_n}$, and $\underline{k}\ge \underline{0}$ means that $k_i\ge 0$ for all $1\le i\le n$. 

Let $\upsilon:\hat{R}_y\rightarrow \dn_0\cup \infty$ be the discrete valuation defined by the unique maximal ideal $\idm = (y_1,\ldots ,y_m; x_1, \dots, x_n)$ of $\hat{R}_y$. Denote by $|\underline{k}|=k_1+\ldots +k_n$. 
\begin{defin}
\label{D:AlgebraPi} For any element
$
0\neq P := \sum\limits_{\underline{k} \ge \underline{0}} a_{\underline{k}} \underline{\partial}^{\underline{k}} \in \cm_n
$
we define its \emph{order} to be
\begin{equation}\label{E:LastOrder}
\Ord (P) := \sup\bigl\{|\underline{k}| - \upsilon(a_{\underline{k}}) \; \big|\; \underline{k} \in \dn_0^n \bigr\} \in \dz \cup \{\infty \},
\end{equation}
and define $\Ord (0):=-\infty$. 
Define 
$$
\hat{D}_n^{sym}:=\bigl\{Q \in \cm \,\big|\, \Ord (Q) < \infty \bigr\}.
$$
Let $P\in \hat{D}_n^{sym}$. Then we have uniquely determined $\alpha_{\underline{k},\underline{i}}\in K$ such that 
\begin{equation}\label{E:expOperatorP}
P = \sum\limits_{\underline{k}, \underline{i} \, \ge \, \underline{0}} \alpha_{\underline{k}, \underline{i}} \,  \underline{x}^{\underline{i}} \underline{\partial}^{\underline{k}},
\end{equation}
where $\underline{i}\in \dn_0^{n+m}$, $\underline{x}^{\underline{i}}=x_1^{i_1}\ldots x_n^{i_n}y_1^{i_{n+1}}\ldots y_m^{i_{n+m}}$. 

For any $m \ge -d:= \Ord (P)$ we put:
$$
P_m:= \sum\limits_{ |\underline{i}| - |\underline{k}| = m} \alpha_{\underline{k}, \underline{i}} \,  \underline{x}^{\underline{i}} \underline{\partial}^{\underline{k}}
$$
to be the $m$-th \emph{homogeneous component} of $P$. Note that $\Ord (P_m) = -m$ and we have a {\it homogeneous decomposition} 
$P = \sum\limits_{m=-d}^\infty P_m.$
\end{defin}

\begin{rem}
Note that for a partial differential operator $P$ with {\it constant} highest symbol the order $\Ord (P)$ and the usual order  coincide. 
\end{rem}

\begin{defin}
\label{D:symbol}
Define the {\it highest symbol} of $P\in \hat{D}_n^{sym}$ as $\sigma (P):=P_{\Ord (P)}=P_{-d}$. We say that $P\in \hat{D}_n^{sym}$ is {\it homogeneous} if $P=\sigma (P)$. 
\end{defin}

\begin{thm}
\label{T:D_n_properties}
There are the following properties of $\hat{D}_n^{sym}$: 
\begin{enumerate}
\item
$\hat{D}_n^{sym}$ is a ring (with natural operations $\cdot$, $+$ coming from $D_n$); $\hat{D}_n^{sym}\supset D_n$. 
\item
$\hat{R}_y$ has a natural structure of a left $\hat{D}_n^{sym}$-module, which extends its natural structure of a left $D_n$-module. 
\item
We have a natural isomorphism of $K$-vector spaces 
$$
\hat{D}_n^{sym}/\idm \hat{D}_n^{sym}\rightarrow K[\partial_1,\ldots ,\partial_n].
$$
\item
Operators from $\hat{D}_n^{sym}$ can realise arbitrary endomorphisms of the $K$-algebra $\hat{R}$ which are continuous in the $\idm$-adic topology. 
\item
There are Dirac delta functions, operators of integration, difference operators.
\end{enumerate}
\end{thm}

\begin{proof}
The proof of items 1-3 is essentially contained in \cite[Th. 5.3]{BurbanZheglov2017} (our extension of the coefficient ring is not essential for the proof).

4. Let $\alpha \in \End_{K-alg}^c\hat{R}$ be a continuous algebra endomorphism. Then it is defined by the images $\alpha (x_i)\in \idm $. Put $u_i=\alpha (x_i)-x_i$, and define $\hat{D}_n^{sym}\ni P_{\alpha}=\sum_{\underline{i}\ge 0} (\underline{i}!)^{-1}\underline{u}^{\underline{i}}\underline{\partial}^{\underline{i}}$. Then for any $f\in \hat{R}$ we have 
$$
P_{\alpha}\circ f(x_1,\ldots ,x_n)=f(x_1+u_1,\ldots ,x_n+u_n);
$$
in particular, $P_{\alpha}$ realize the endomorphism $\alpha$. Note that this operator can be produced by the following receipt. Consider a commutative power series ring \\
$K[[y_1, \ldots , y_m; x_1, \ldots ,x_n; \partial_1, \ldots ,\partial _n]]$ with a product denoted by $*$. If we take the formal exponent $\exp (u_1*\partial_1+\ldots +u_n*\partial_n)$, write all coefficients at the left hand side and replace $*$ by $\cdot$, we'll obtain the operator  $P_{\alpha}$.  

5. Note that $\delta_i:=\exp ((-x_i)*\partial_i)$ (a partial case of the shift operator  $P_{\alpha}$) realize the $i$-th delta-function (cf. \cite[Example 5.4]{BurbanZheglov2017}):
$$
\delta_i\circ f(x_1,\ldots ,x_{i-1},x_i,x_{i+1},\ldots ,x_n)=f(x_1,\ldots ,x_{i-1},0,x_{i+1},\ldots ,x_n);
$$
the operator
$$
\int_i:=(1-\exp ((-x_i)*\partial_i))\cdot \partial_i^{-1}=\sum_{k=0}^{\infty}\frac{x^{k+1}}{(k+1)!}(-\partial)^k
$$
realise the operators of integration (cf. \cite[Example 5.5]{BurbanZheglov2017})
$$
\int_i\circ x_i^m=\frac{x_i^{m+1}}{m+1};
$$
the ordinary difference operators $\sum_{i=0}^M f_i(n)T^i$ with $T\circ f(n)=f(n+1)$ can be embedded e.g. as follows 
$$
\sum_{i=0}^Mf_i(n)T^i\hookrightarrow \hat{D}_n^{sym} \mbox{\quad via\quad} T\mapsto x, n\mapsto -x\partial. 
$$
\end{proof}

\begin{rem}
\label{R:ord_properties}
Unlike the usual ring of PDOs the ring $\hat{D}_n^{sym}$ contains zero divisors. There are the following obvious properties of the order function (cf. the proof of \cite[Th. 5.3]{BurbanZheglov2017}):
\begin{enumerate}
\item
$\Ord (P\cdot Q)\le \Ord (P)+\Ord (Q)$, and the equality holds if $\sigma (P)\cdot \sigma (Q)\neq 0$,
\item
$\sigma (P\cdot Q)=\sigma (P)\cdot \sigma (Q)$, provided $\sigma (P)\cdot \sigma (Q)\neq 0$,
\item
$\Ord (P+Q)\le \max\{\Ord (P),\Ord (Q)\}$.
\end{enumerate}
In particular, the function $-\Ord$ determines  a discrete pseudo-valuation on the ring $\hat{D}_n^{sym}$. 
\end{rem}

\begin{rem}
\label{R:other_definitions}
There are other possible ways to define a "symmetric" completion of the ring $D_n$ (see \cite[\S 2.1.5]{Zheglov2013}). E.g. we can define for each sequence in $\idm D_n$, $\{(P_n)_{n\in \sdn}\}$, such that $P_n(R)$ converges uniformly in $\hat{R}$ (i.e. for any $k>0$ there is $N>0$ such that $P_n(\hat{R})\subseteq \idm^k$ for $n\ge N$) a $k$-linear operator $P: \hat{R}\rightarrow \hat{R}$ by
$$
P(f)=\limind_{n\rightarrow \infty}\sum_{v=0}^nP_v(f), \mbox{\quad} P:=\sum_n P_n,
$$
and define a completion to be the ring consisting of such operators. This completion is bigger, but $\hat{D}_n^{sym}$ has finer properties sufficient for many aims. In particular, the classification theory from \cite{Zheglov2013} deals with commutative subrings belonging to the more narrow ring $\hat{D}_n^{sym}$. 
\end{rem}

Now we'll deduce one unexpected corollary of our first theorem -- a new proof of the Abhyankar formula, see \cite[III.2]{BCR}, \cite{Ab}, \cite{Good}, \cite{Joni}. Following the exposition of \cite{BCR}, we introduce notation first.

Put 
$$
MA_n^0((K))=\{ F=(F_1, \ldots ,F_n)| \quad F_i\in K[[x_1, \ldots ,x_n]], F_i(0)=0,\quad i=1, \ldots , n\}.
$$
If $U\in K[[x_1, \ldots ,x_n]]$ and $F\in MA_n^0((K))$ we put $U(F)=U(F_1, \ldots , F_n)$. The Jacobian matrix of $F\in MA_n^0((K))$ is $J(F)=(\partial_i(F_j))$ and its Jacobian determinant is $j(F)=\det J(F)$.   There is a monoid homomorphism 
$$
J_0: MA_n^0((K))\rightarrow M_n(K), \quad J_0(F)=J(F)(0),
$$
and denote by $GA_n^0((K))$ the group of invertible elements in $MA_n^0((K))$ -- it consists of those $F$ for which $J_0(F)$ is invertible. We have a split exact sequence 
$$
1\rightarrow GA_n^1((K))\rightarrow GA_n^0((K))\stackrel{J_0}{\rightarrow} GL_n(K)\rightarrow 1,
$$
where, if $a=(a_{ij})\in GL_n(K)$, then the splitting is given via $L(a)=(L_1, \ldots ,L_n)$ with $L_i=\sum_j a_{ij}x_j$, $i=1, \ldots ,n$. The kernel $GA_n^1((K))$ of $J_0$ consists of those $F$ of the form $X+H$, where $H$ involves only monomials of degree $\ge 2$ in $x$.  

For a given $F$ define derivations $\bigtriangleup_i: K[[\underline{x}]]\rightarrow K[[\underline{x}]]$ by 
$$
\bigtriangleup_i(G):= \det J(F_1,\ldots ,F_{i-1}, G, F_{i+1}, \ldots , F_n), \quad G\in K[[\underline{x}]].
$$
If $j(F)=1$, we have $\bigtriangleup_i(F_j)=\delta_{ij}$, $i,j =1, \ldots ,n$. Thus the 
$\bigtriangleup_i$'s commute in $K[[F]]$ and hence also in $K[[\underline{x}]]$. Clearly, $\bigtriangleup_i=\sum_{j=1}^n\bigtriangleup_i(x_j)\partial_j$, and the map $\phi (x_i):=F_i$, $\phi (\partial_i):= \bigtriangleup_i$ determines an endomorphism of $D_n$. If $F\in K[x_1, \ldots ,x_n]$, the open Jacobian Conjecture asks whether the inverse map $G: K[x_1, \ldots ,x_n]\rightarrow K[x_1, \ldots ,x_n]$ is polynomial. Clearly, for this it suffices to show that the endomorphism $\phi$ of the $n$-th Weyl algebra is an automorphism, which is the subject of another open Dixmier conjecture, cf. \cite{Ts1}, \cite{Ts2}, \cite{BK}.

\begin{cor}[the Abhyankar formula]
\label{C:Abhyankar}
Let $F\in GA_n^1((K))$ with $j(F)=1$. Put $H_i:=x_i -F_i(\underline{x})$. Then the inverse map is given by the formula
\begin{equation}
\label{E:Abhyankar}
G_i=\sum_{\underline{p}\in \sdn_0^n}\frac{\underline{\partial}^{\underline{p}}}{\underline{p}!}(x_iH^{\underline{p}}).
\end{equation}
Analogously, for any $U\in K[[x_1, \ldots ,x_n]]$ we have 
\begin{equation}
\label{E:Abhyankar1}
U=\sum_{\underline{p}\in \sdn_0^n}\frac{\underline{\partial}^{\underline{p}}}{\underline{p}!}(U(F)H^{\underline{p}}).
\end{equation}
\end{cor}

\begin{proof}
By theorem \ref{T:D_n_properties} the operator $S:=\exp (-H_1*\partial_1-\ldots - H_n*\partial_n)$ acts on  $f(x_1, \ldots ,x_n)$ as $S\circ f(\underline{x})=f(\underline{x}+H)$, hence $Sf(\underline{x})S^{-1}=f(\underline{x}+H)$ in the ring $\hat{D}_n^{sym}$. Since $\Ord (\exp (-H_1*\partial_1-\ldots - H_n*\partial_n)-1)<0$ by the assumptions on $H$, the inverse operator is given by the series $S^{-1}=1+S_-+S_-^2+\ldots $, where $S_-=1-S$. 

Denote by $*:D_n\rightarrow D_n$, $*(x_i):=x_i$, $*(\partial_i):=-\partial_i$  the anti-isomorphism on the ring of differential operators. It can be extended on some operators from $\hat{D}_n^{sym}$, in particular on $S, S^{-1}$ (since $S_-$ is a sum of differential operators of strictly decreasing order). Since we have 
$$
*(Sx_iS^{-1})=*(S^{-1})x_i*(S)=*(F_i)=F_i, \quad i=1, \ldots ,n,
$$
we get $[x_i,*(S)S]=0$ for all $i=1, \ldots ,n$, hence $f:=*(S)S\in K[[x_1, \ldots ,x_n]]$. On the other hand, we have
$$
S\partial_iS^{-1}=\partial_i +S\partial_i(S^{-1})=\sum_{j=1}^n\bigtriangleup_i(x_j)\partial_j, \quad i=1, \ldots ,n.
$$
Indeed, $S^{-1}=1-S_-^{-1}$, where all monomials of $S_-^{-1}$ depend on some $\partial_i$, and therefore all monomials of $S\partial_iS^{-1}$ depend on some $\partial_i$ as well. On the other hand, $[F_i,S\partial_jS^{-1}]= [x_i, \partial_j]=\delta_{ij}$, i.e. the operator $S\partial_jS^{-1}- \sum_{j=1}^n\bigtriangleup_i(x_j)\partial_j$ commutes with all $F_i$. Since all its monomials depend on some $\partial_i$, it must be zero. So, we have 
\begin{multline*}
*(S\partial_iS^{-1})=-*(S^{-1})\partial_i*(S)=*(\bigtriangleup_i)=-\sum_{j=1}^n\partial_j\bigtriangleup_i(x_j)= \\
- \sum_{j=1}^n\bigtriangleup_i(x_j)\partial_j - \sum_{j=1}^n\partial_j(\bigtriangleup_i(x_j)), \quad i=1, \ldots ,n.
\end{multline*}
Note that $\sum_{j=1}^n\partial_j(\bigtriangleup_i(x_j))=0$. For, the matrix $(\bigtriangleup_i(x_j))$, $i,j=1, \ldots ,n$ is the inverse matrix to the matrix $J(F)$, i.e. it is equal to $J(G)$.  It is not difficult to see that $J(G)=E+N$, where $N$ is a nilpotent matrix. Then
$$
\sum_{j=1}^n\partial_j(\bigtriangleup_i(x_j))= \sum_{j=1}^n \partial_j(\partial_i (G_j))=\partial_i (Tr(J(G))=0.
$$
Hence $*(S^{-1})\partial_i*(S)=S\partial_iS^{-1}$ and therefore $[\partial_i,*(S)S]=0$ for all $i=1, \ldots ,n$, i.e. $f\in K$. Since $f=\sum_{\underline{p}\in \sdn_0^n}\frac{\underline{\partial}^{\underline{p}}}{\underline{p}!}(H^{\underline{p}})$, it must be equal to $1$, i.e. $*(S)=S^{-1}$. 

Now we have $G_-=*(S)x_iS$. Since $G_i\in K[[x_1, \ldots ,x_n]]$ and $S=1-S_-$, where all monomials of $S_-$ depend on some $\partial_i$, it suffices to find only the free term of the operator $*(S)x_i$. Since 
$*(S)=\exp (\partial_1*H_1+\ldots +\partial_n*H_n)$, we immediately get the formula \eqref{E:Abhyankar}. The second formula \eqref{E:Abhyankar1} is now obvious.
\end{proof}

\section{The rings $\hat{D}_n$, $\hat{E}_n$, $E_n$ and their order function $\ord_n$}
\label{S:algebra_D_n}

From technical point of view it is more convenient to deal with a more narrow non-symmetric version of completion $\hat{D}_n$ (it is well adapted for the classification of commutative subrings). In particular, there is an analogue of the Schur theory which will be important for us (see below). 

\begin{defin}
\label{D:nonsymmetric}
We define $\hat{D}_1=D_1$ and define $\hat{D}_n=\hat{D}_{n}^{n}[\partial_n]$, where 
$\hat{D}_{n}^{n}$ denotes the subring in $\hat{D}_n^{sym}$ consisting of operators {\it not depending on $\partial_n$}. Obviously, $\hat{D}_n\subset \hat{D}_n^{sym}$.

We define associated pseudo-differential rings $\hat{E}_n=\hat{D}_{n}^{n}((\partial_n^{-1}))$, \\
$E_n=D_{n}^n((\partial_n^{-1}))$ (cf. \cite[\S 2.1.4]{Zheglov2013}, \cite[Def. 2.1]{Zheglov2018}).  

The subset
\begin{equation}
\label{F:Pi}
\Pi_n =\{P\in \hat{E}_n| \quad \Ord (P)<\infty\} 
\end{equation}
is, clearly, an associative subring. Consider the vector space 
$$
\hat{\cm}_n:=\{\sum_{i\in \sdz}a_{i}\partial_n^i, \quad a_i\in \hat{D}_n^n\}.
$$
The function $\Ord$ can be obviously extended to this space.

We define the vector subspace 
$$
\hat{\Pi}_n:=\{ P\in \hat{\cm}_n| \quad \Ord (P)<\infty \}.
$$
\end{defin}

\begin{defin}
\label{D:ord_n}
We define the  {\it order function} $\ord_n$  on $\hat{E}_n$ as
$
\ord_n(P)=l
$
if $\hat{E}_n\ni P=\sum_{s=-\infty}^lp_s\partial_n^s$, and set $\ord_n(0):=- \infty$. 

The coefficient $p_l$ is called {\it the highest term} and will be denoted by $HT_n(P)$ (as the term naturally associated with the function $\ord_n$).

The symbol $\sigma$ on $\hat{\Pi}_n$ can be defined in the same way as in $\hat{D}_n^{sym}$. 
\end{defin}

The order function $\ord_n$ and the highest term $HT_n$ behave like the $\Ord$-function and highest symbol. Namely, the following properties obviously hold:
\begin{enumerate}
\item 
$HT_n(P\cdot Q)=HT_n(P)\cdot HT_n(Q)$ provided $HT_n(P)\cdot HT_n(Q)\neq 0$;
\item
$\ord_n (P\cdot Q)\le \ord_n (P)+\ord_n (Q)$, and the equality holds if $HT_n (P)\cdot HT_n (Q)\neq 0$,
\item
$\ord_n (P+Q)\le \max\{\ord_n (P),\ord_n (Q)\}$.
\end{enumerate}
In particular, the function $-\ord_n$ determines  a discrete pseudo-valuation on the ring $\hat{D}_n$.

The space $\hat{\Pi}_n$ has a natural left $\hat{D}_n^{sym}$-module structure. To prove it we'll need one more definition:

\begin{defin}
\label{D:slices}
Let $P \in \hat{D}_n^{sym}$ be an operator of order $d$ given by the expansion (\ref{E:expOperatorP}). Then we have another form of the formal power series expansion of $P$ called \emph{slice decomposition}:
\begin{equation}
P = \sum\limits_{\underline{i} \ge \underline{0}} \frac{\underline{x}^{\underline{i}}}{\underline{i}!} \, P_{(\underline{i})}, \quad \mbox{\rm where} \quad \underline{i}\in \dn_0^n\times \dn_0^m, \quad \underline{x}^{\underline{i}}=x_1^{i_1}\ldots x_n^{i_n}y_1^{j_1}\ldots y_m^{j_m},\quad \underline{i}! =i_1!\ldots i_n!,
$$
$$
P_{(\underline{i})} = \underline{i}! \sum_{\substack{\underline{k} \ge \underline{0} \\ |\underline{k}| - |\underline{i}| \le d }} \alpha_{\underline{k}, \underline{i}} \, \underline{\partial}^{\underline{k}}, \quad \alpha_{\underline{k}, \underline{i}}\in K.
\end{equation}
For any multi-index $\underline{i}$, the partial differential operator with constant coefficients  $P_{(\underline{i})} \in K [\partial_1, \dots, \partial_n]$ is called $\underline{i}$-th slice of $P$.

Analogously, for any element $P\in \hat{\cm}_n$ the slice decomposition is defined as:
\begin{equation}
P = \sum\limits_{\underline{i} \ge \underline{0}} \frac{\underline{x}^{\underline{i}}}{\underline{i}!} \, P_{(\underline{i})}, \quad \mbox{\rm where} \quad
P_{(\underline{i})} = \underline{i}! \sum_{\substack{\underline{k} \in \sdn_0^{n-1}\times\sdz }} \alpha_{\underline{k}, \underline{i}} \, \underline{\partial}^{\underline{k}}, \quad \alpha_{\underline{k}, \underline{i}}\in K.
\end{equation}

\end{defin}

\begin{lem}
\label{L:module_structure}
There is a natural left $\hat{D}_n^{sym}$-module and right $\Pi_n$-module structure on $\hat{\Pi}_n$:
$(P\in \hat{D}_n^{sym}, A\in \hat{\Pi}_n, Q\in \Pi_n) \rightarrow P\circ A\circ Q\in \hat{\Pi}_n$. 

Besides, the properties 1-3 from remark \ref{R:ord_properties} hold also for the module $\hat{\Pi}_n$. 
\end{lem}

\begin{proof}
Any $A\in \hat{\Pi}_n$ can be uniquely represented as a sum $A=A_++A_-$, where $A_+\in \hat{D}_n^{sym}$ and $A_-=\sum_{i\ge 0}A_i\partial_n^{-i}$. We define a left $\hat{D}_n^{sym}$-module structure by the rule 
$P\circ A:= P\cdot A_+ + P\circ A_-$, where 
\begin{equation}
\label{E:umn}
P\circ A_-= \sum_{i\ge 0}(P\cdot A_i) \partial_n^{-i},
\end{equation}
and this sum is well defined by the following reasons. By the properties of the function $\Ord$ we have $\Ord (A_i)\le \Ord (A)+i$ for any $i\ge 0$, and therefore $\Ord (P\cdot A_i)\le \Ord (P)+\Ord (A_i)+i$ (in particular, $\Ord (P\circ A_-)< \infty$ is this sum is well defined). Let $P\cdot A_i=\sum_{j\ge 0}B_{i,j}\partial_n^j$, then 
$$
B_{i,j}=\sum_{l\ge 0}C_{j+l}^lP_{j+l}\partial_n^l(A_i),
$$
where $P=\sum_{l\ge 0}P_l\partial_n^l$, and 
$$
(P\cdot A_i)\partial_n^{-i}=\sum_{k=-i}^{\infty}C_{i,k}\partial_n^k,
$$
where $C_{i,k}=B_{i,k+i}$. Then the sum \eqref{E:umn} is well defined if the sums $\sum_{i\ge 0}C_{i,k}$ are well defined for any $k$. Note that $\Ord (C_{i,k})\le \Ord (P)+\Ord (A)-k$, and therefore we need to look at the slice decompositions of $C_{i,k}$. As 
$$
C_{i,k}=B_{i,k+i}=\sum_{l\ge 0}C_{k+i+l}^lP_{k+i+l}\partial_n^l(A_i),
$$
and $\Ord (P_l)\le \Ord (P)-l$, in the case when $\alpha :=\Ord (P)-k-i <0$ the slice decomposition is
$$
C_{i,k}=\sum_{|\underline{p}|\ge |\alpha |} \frac{x^{\underline{p}}}{\underline{p}!}(C_{i,k})_{(\underline{p})}.
$$
Thus, the sum $\sum_{i\ge 0} C_{i,k}$ is well defined, as there are only finitely many terms of the form $\frac{x^{\underline{p}}}{\underline{p}!}(C_{i,k})_{(\underline{p})}$ for any given $|\underline{p}|$ in the slice decompositions. 

By similar arguments, $P\circ A = \sum_{l=0}^{\infty} (P_l\partial_n^l)\circ A$ (as there are only finitely many terms of the form $\frac{x^{\underline{p}}}{\underline{p}!}((P_l\partial_n^l)\circ A)_{(\underline{p})}$ for any given $|\underline{p}|$ in the whole sum), and 
$$
(P_l\partial_n^l)\circ A=\sum_{i\in\sdz}(P_l\partial_n^l)\cdot (A_i\partial_n^{-i})
$$ 
(the product is taken in the ring $\Pi_n$), since there are only finitely many summands of given $\ord_n$-order in this sum. Therefore,
$$
P\circ A= \sum_{l\ge 0}\sum_{i\in \sdz}(P_l\partial_n^l)\cdot (A_i\partial_n^{-i}),
$$
in particular, all axioms of the left module are satisfied.

Analogously, we define a right $\Pi_n$-module structure by the rule $A\circ Q:= A\circ Q_++A\circ Q_-$, where $A\circ Q_+=A_+\cdot Q_+ + A_-\cdot Q_+$ (the first product is taken in $\hat{D}_n^{sym}$, and the second in $\Pi_n$), and $A\circ Q_-=A_+\circ Q_- + A_-\cdot Q_+$. By the same arguments as above 
$$
A\circ Q= \sum_{i\in \sdz}\sum_{j\ge m}(A_i\partial_n^{-i})\cdot (Q_j\partial_n^{-j}).
$$

Therefore, as $\Pi_n$ is an associative ring,
\begin{multline*}
(P\circ A)\circ Q=\sum_{j\ge m}\sum_{l\ge 0}\sum_{i\in\sdz} ((P_l\partial_n^l)\cdot (A_i\partial_n^{-i}))\cdot (Q_j\partial_n^{-j})= \\
\sum_{l\ge 0}\sum_{i\in\sdz}\sum_{j\ge m} (P_l\partial_n^l)\cdot ((A_i\partial_n^{-i})\cdot (Q_j\partial_n^{-j}))=P\circ (A\circ Q)
\end{multline*}
and $A$ is a bimodule. 

The properties 1-3 are now obvious.
\end{proof}

\section{Commutative subrings in $\hat{D}_n^{sym}$ and their spectral modules}
\label{S:spectral}

\subsection{Spectral modules}

Let $B\subset \hat{D}_n^{sym}$ be a commutative subring. 

\begin{defin}
\label{D:spectral_module}
The $B$-module $F:=\hat{D}_n^{sym}/(x_1,\ldots ,x_n) \hat{D}_n^{sym}\simeq K_y[[\partial_1,\ldots ,\partial_n]]\cap \Pi_n$ is called {\it spectral module} of the ring $B$. 
\end{defin}

Note that $F$ is actually a right $\hat{D}_n^{sym}$ module. However, since the ring $B$ is commutative, we will view $F$ as a left $B$-module, having the natural right action in mind. The following proposition explains the term "spectral", cf. \cite[Prop. 4.2]{Zheglov_belovezha}.

\begin{prop}
\label{P:spectral_module}
Let $K_y=K$. Let $B\subset \hat{D}_n^{sym}$ be a finitely generated commutative subring over $K_y$ such that the spectral module $F$ is finitely generated. 

For any character $\chi_q :B\rightarrow K_q$, where $q\subset B$ is a maximal ideal and $K_q=B/q$ is the residue field, consider the vector space 
$$
Sol(B,\chi_q )=\{f\in K_q[[x_1, \ldots , x_n]] \,\big|\,\quad Q (f)=\chi_q (Q)f\quad \forall Q\in B \}.
$$
Observe that $Sol(B,\chi_q )$ has a natural $B$-module structure: $f\in Sol(B,\chi_q ) \Rightarrow \forall Q\in B$ $Q(f)\in Sol(B,\chi_q )$. 

The following K--linear map
\begin{equation}
F \stackrel{\eta_{\chi_q}}\rightarrow \mathsf{Sol}\bigl(B, \chi_q\bigr)^\ast, \quad \underline{\partial}^{\underline{i}} \mapsto \Bigl(f \mapsto \left.\dfrac{1}{\underline{p}!} \dfrac{\underline{\partial}^{|p|}f}{\partial x_1^{p_1} \ldots \partial x_n^{p_n}}\right|_{(0, 0)}\Bigr)
\end{equation}
is also $B$--linear, where $\mathsf{Sol}\bigl(B, \chi_q\bigr)^\ast = \Hom_{K_q}\bigl(\mathsf{Sol}\bigl(B, \chi_q\bigr), K_q\bigr)$ is the vector space dual of the solution space. Moreover, the induced map
\begin{equation}
\label{E:solspaceisom}
F|_{\chi_q}:=(B/\ker \chi_q )\otimes_BF\stackrel{\bar{\eta}_{\chi_q}}\rightarrow Sol(B,\chi_q )^*
\end{equation}
is an isomorphism of $B$-modules.  In particular, $\dim_{K}\Bigl(\mathsf{Sol}\bigl(B, \chi_q\bigr)\Bigr)<\infty$  for any  $\chi_q$.
\end{prop}

\begin{proof}
In the one-dimensional case, for the ring $D_1$, the isomorphism (\ref{E:solspaceisom}) is due to Mumford \cite[Section 2]{Mumford}. In our case the proof is practically the same as for partial differential operators \cite[Th. 4.5 item 2]{BurbanZheglov2017}, cf. \cite[Theorem 1.14]{BurbanZheglov2016}, \cite[Rem. 2.3]{KOZ2014}. For convenience of the reader we give it here. 

First note that the following map
\begin{equation}
\Hom_{K_q}\bigl(F, K_q\bigr) \stackrel{\Phi}\rightarrow K_q[[x_1, \ldots , x_n]], \quad \lambda \mapsto \sum\limits_{\underline{p}\ge 0} \frac{1}{\underline{p}!} \lambda(\underline{\partial}^{\underline{p}}) \underline{x}^{\underline{p}}
\end{equation}
is  an  isomorphism of left $\hat{D}_n^{sym}$--modules, where the $\hat{D}_n^{sym}$-module structure on $\Hom_{K_q}$ is given, as usual, by the rule $(P\cdot \lambda )(-):= \lambda (-\cdot P)$ and $\hat{D}_n^{sym}$ acts on $K_q[[x_1, \ldots ,x_n]]$ by the usual application of an  operator on a function.

Let $B \stackrel{\chi_q}\rightarrow K_q$ be a character, then $K_q = B/\ker(\chi_q)$ is a left $B$--module. We obtain a $B$--linear map
\begin{equation}
\Psi: \Hom_{B}(F, K_q) \stackrel{I}\rightarrow \Hom_{K_q}(F, K_q) \stackrel{\Phi}\rightarrow K_q[[x_1, \ldots , x_n]],
\end{equation}
where $I$ is the forgetful map. The image of $I$ consists of those $K_q$-linear functionals, which are also $B$--linear, i.e.
\begin{equation*}
\mathsf{Im}(I) = \bigl\{\lambda \in \Hom_{K_q}(F, K_q) \; \big| \;  \lambda(\,-\, \cdot P) = \chi_q(P)\cdot \lambda(\,-\,)\; \mbox{for all} \; P \in B\bigr\}.
\end{equation*}
This implies that $\mathsf{Im}(\Psi) = \mathsf{Sol}(B, \chi_q)$. Next, we have  a canonical  isomorphism of $B$--modules:
$
\Hom_{B}(F, K_q) \cong \Hom_{K_q}\bigl(B/\ker(\chi_q) \otimes_{B} F, K_q\bigr).
$
 Dualizing again, we get an isomorphism of vector spaces
$$
\Psi^\ast: \mathsf{Sol}(B, \chi_q)^\ast \rightarrow \bigl(B/\ker(\chi_q) \otimes_{B} F\bigr)^{\ast\ast} \cong B/\ker(\chi_q) \otimes_{B} F.
$$
It remains to observe that $\Psi^\ast$ is also $B$--linear and  $\bigl(\Psi^\ast\bigr)^{-1} = \bar{\eta}_{\chi_q}$.
\end{proof}

Typical examples of rings satisfying the conditions from proposition are rings with special condition on their highest symbols, e.g.:

\begin{thm}{\cite[T.2.1]{KOZ2014}}
\label{T:techn5.2}
 Let $P_1,\ldots ,P_n\in D_n=K[[x_1,\ldots ,x_n]][\partial_1,\ldots ,\partial_n]\subset \hat{D}_n^{sym}$ be any commuting operators of positive order.  Let $B$ be any commutative $K$-subalgebra in $D_n$ which contains the operators $P_1,\ldots ,P_n$.
Assume that the intersection of the characteristic divisors of $P_1,\ldots ,P_n$ is empty.

Then the map from $\gr(D_n)$ to $\gr(D_n)/\idm \gr(D_n)  = K[ \xi_1, ..., \xi_n]$, where $\gr$ denotes the associated graded ring with respect to the filtration defined by the {\it usual order} on $D_n$, induces an embedding on $\gr(B)$ and we also have the following properties.
\begin{enumerate}
\item \label{I:fg}
 $ K[\xi_1,\ldots ,\xi_n]$ is finitely generated as $\gr (B)$-module.
\item
The rings $B$ and $\gr B$ are finitely generated integral $K$-algebras of Krull dimension $n$.
\item \label{I:uni}
The affine variety $U = \Spec \, B$ over $K$ can be naturally completed to an $n$-dimensional irreducible projective variety $X$ with boundary $C$ which is an integral Weil divisor not contained in the singular locus of $X$. Moreover, $C$ is an unirational and ample $\dq$-Cartier   divisor.
\item The $B$-module $F = D_n/  \idm D_n$, which defines a coherent sheaf on $U$, can be naturally extended to a torsion free coherent
sheaf $\cf$ on $X$. Moreover, the self-intersection index $(C^n)$ on $X$  is equal to $\delta^n/\rk (\cf)$, where 
\begin{equation}
\label{E:delta}
\delta= {\rm gcd} \ \{n   \mid  B_n/B_{n-1} \ne 0 , \, n\ge 1\} \mbox{.}
\end{equation}

\end{enumerate}
\end{thm}

%Other examples, including certain criterions, see in section \ref{S:coherence}. 

\subsection{$\Gamma$-order and quasi elliptic rings}
\label{S:quasi-elliptic}

In this subsection we describe commutative subrings in $\hat{D}_n^{sym}$ mentioned in Introduction -- the quasi-elliptic rings -- that admit effective description in terms of algebro-geometric spectral data. 

First we introduce the notion of $\Gamma$-order. This order is defined on {\it some elements} of the ring $\hat{D}_n^{sym}$ or of the ring $\hat{E}_n$.  

\begin{defin}
\label{D:G-order}
Let's denote by $\hat{D}_n^{i_1,\ldots ,i_q}$ the subring in $\hat{D}_n^{sym}$ consisting of operators {\it not depending} on $\partial_{i_1},\ldots ,\partial_{i_q}$. The $\Gamma$-order is defined recursively. 

We say that a nonzero operator $P\in \hat{D}_n^{2,3,\ldots n}$  has $\Gamma$-order $k_1$ if $P=\sum_{s=0}^{k_1}p_s\partial_1^s$, where $0\neq p_{k_1}\in \hat{R}$. 

We say that a nonzero operator $P\in \hat{D}_n^{i+1,i+2,\ldots ,n}$  has $\Gamma$-order $(k_1, \ldots ,k_i)$ if $P=\sum_{s=0}^{k_i}p_s\partial_i^s$, where $p_s\in \hat{D}_n^{i, i+1, \ldots ,n}$, and the $\Gamma$-order of $p_{k_i}$ is $(k_1,\ldots ,k_{i-1})$. 

We say that a nonzero operator $P\in \hat{D}_n^{sym}$ {\it has $\Gamma$-order} 
$$\ord_{\Gamma} (P)=(k_1,\ldots ,k_n)$$ 
if $P=\sum_{s=0}^{k_n}p_s\partial_n^s$, where $p_s\in \hat{D}_n^{n}$,
and the $\Gamma$-order of $p_{k_n}$ is $(k_1,\ldots ,k_{n-1})$. 

In this situation we say that the operator $P$ is {\it monic} if the highest coefficient (defined recursively in analogous way)  $p_{k_1,\ldots ,k_n}$ is $1$.

The same definition works for the rings $\hat{E}_n,\hat{\Pi}_n$ with minimal modification at the last step: $P\in \hat{\Pi}_n$ {\it has $\Gamma$-order} 
$$
\ord_{\Gamma} (P)=(k_1,\ldots ,k_n)
$$ 
if $P=\sum_{s=-\infty}^{k_n}p_s\partial_n^s$, where $p_s\in \hat{D}_n^{n}$,
and the $\Gamma$-order of $p_{k_n}$ is $(k_1,\ldots ,k_{n-1})$.

The $\Gamma$-orders are ordered with respect to the anti-lexicographical order on $\dz^n$, i.e. $(k_1,\ldots ,k_n)< (k_1',\ldots ,k_n')$ iff $k_n<k_n'$ or $k_n=k_n'$ and $k_{n-1}<k_{n-1}'$, or $k_n=k_n'$, $k_{n-1}=k_{n-1}'$ and $k_{n-2}<k_{n-2}'$ etc. 
\end{defin}

\begin{defin}
\label{D:growth_cond}
We say that an operator $P\in \hat{E}_n$ {\it satisfies  condition $A_{1}$} if $\Ord (P)\le |\ord_{\Gamma}(P)|$.
\end{defin}

It is easy to see that any {\it monic} operator $P$ satisfies condition $A_1$  iff  $\Ord (P)=|\ord_{\Gamma}(P)|$.

The $\Gamma$-order, like the function $\Ord$, has similar properties:

\begin{lem}{(cf. \cite[Lemma 2.7]{Zheglov2013})}
\label{L:lemma2.7}
Assume $P_1,P_2\in \hat{E}_n$ satisfy  the condition $A_{1}$. Then $P_1P_2$ satisfies the condition $A_{1}$ 
and $\ord_{\Gamma}(P_1P_2)= \ord_{\Gamma}(P_1)+\ord_{\Gamma}(P_2)$. 
\end{lem}

\begin{proof}
The first property follows from the second and from the properties of function $\Ord$. The second property follows by induction on $n$ because of the properties of $HT_n$. For $n=1$ this is true, because $\hat{R}_y$ is a domain.  
\end{proof}

\begin{cor}
\label{C:corol2.1}
If the operator $S=1-S_-$, where $S_-\in \hat{D}_n^n[[\partial_n^{-1}]]\partial_n^{-1}$, satisfies the condition $A_{1}$ (or equivalently $\Ord (S)=|\ord_{\Gamma}(S)|=0$), then the operator $S^{-1}=1+\sum_{q=1}^{\infty}(S_-)^q$ also satisfies it (and $\Ord (S^{-1})=0$, $|\ord_{\Gamma}(S^{-1})|=0$). 
\end{cor}

\begin{proof}
By the properties of $\Ord$ we have $\Ord (S_-)\le 0$. So, the corollary will be proved if we show the series $1+\sum_{q=1}^{\infty}(S_-)^q$ is well defined. By the properties of $\ord_n$ we have $\ord_n (S_-)^k \le -k$. So, the series is well defined and determines the inverse operator. Clearly, it satisfies $A_1$, and $\ord_{\Gamma}(S^{-1})=(0,\ldots ,0)$. 
\end{proof}

\begin{defin}
\label{D:quasi-elliptic}
The subring $B\subset \hat{D}_n\subset \hat{D}_n^{sym}$ of commuting operators is called {\it 1-quasi elliptic} (or just quasi-elliptic for brevity) if there are $n$ operators $P_1,\ldots ,P_n$ such that 
\begin{enumerate}
\item
$\ord_{\Gamma}(P_i)=(0,\ldots 0,1,0\ldots 0, l_i)$  for $1\le i<n$,  
where $1$ stands at the $i$-th place and $l_i\in \dz_+$; 
\item
$\ord_{\Gamma}(P_n)=(0,\ldots , 0, l_n),$ where $l_n>0$;
\item
For $1\le i\le n$ $\Ord (P_i)=|\ord_{\Gamma}(P_i)|$;
\item
$P_i$ are monic.
\end{enumerate}

In general, we call operators $P_1,\ldots ,P_n\in \hat{E}_n$ {\it formally quasi elliptic} if they satisfy the conditions 1-3 above, and we call them monic  formally quasi elliptic, if they satisfy the conditions 1-4 above. 
\end{defin}

\begin{rem}
\label{R:purity}
According to definition, the quasi-elliptic rings are contained in $\hat{D}_n$. However we'll consider them in $\hat{D}_n^{sym}$ because of a weak equivalence relation used in one of classification theorems \ref{T:schurpair1} below. Besides, if we consider a quasi-elliptic ring of rank one (see definition \ref{D:an-alg-rank} below) satisfying mild properties, then its centralizer in $\hat{D}_n^{sym}$ will be still commutative quasi-elliptic ring contained in $\hat{D}_n$. The proof of this result appears to be quite non-trivial and requires not only the technique developed in this paper, but several extra results we'll present in a subsequent paper \cite{Zheglov2020}. 
\end{rem}

\begin{ex}
Immediately from definition \ref{D:quasi-elliptic} it follows that  1-quasi-elliptic subalgebras in $\hat{D}_1$ are exactly {\it elliptic} subalgebras of ordinary differential operators.  

Recall that an ordinary differential operator $P = a_n \partial^n + a_{n-1} \partial^{n-1} + \dots + a_0 \in D_1$ of positive order $n$ is called \emph{(formally) elliptic} if $a_n \in K^*$. A ring $B \subset D_1$ containing an  elliptic element is called \emph{elliptic}.

There is a well known extensive classification theory of such rings, see Introduction. 
\end{ex}

\section{Schur theory for $\hat{D}_n$}
\label{S:Schur_theory}

Many properties of 1-quasi elliptic rings are based on the technique developed in a higher-dimensional  analogue of the Schur theory. In this section we collect technical statements of this theory for the ring $\hat{D}_n$. A review of the Schur theory for the ring $D_1$ see in \cite{Mul} and in \cite[\S 5]{Quandt} for the  relative case. For the ring $\hat{D}_2$ this theory is contained in \cite{Zheglov2013}, \cite{KurkeZheglov}. 

\subsection{Roots of formally quasi-elliptic operators}
\label{S:roots}

\begin{lem}
\label{L:lemma2.9}
Let $P_1,\ldots , P_n\in \hat{D}_n^{sym}$ be monic formally quasi-elliptic operators with $\ord_{\Gamma}(P_i)=(0,\ldots ,0,1,0,\ldots ,l_i)$, $i<n$, where $1$ stands at the $i$-th place, and $\ord_{\Gamma}(P_n)=(0,\ldots ,0, l_n)$. Then 
\begin{enumerate}
\item\label{a}
There exist unique monic commuting operators $L_1, \ldots , L_n\in \hat{E}_n$ with $\ord_{\Gamma} (L_i)=(0,\ldots ,0,1,0,\ldots ,0)$ ($1$ is on the $i$-th place) such that $L_n^{l_n}=P_n$, $L_iL_n^{l_i}=P_i$, $i<n$. 
\item\label{b}
If $P_1, \ldots , P_n$ satisfy the condition $A_{1}$ then $L_1,\ldots ,L_n$ satisfy the condition $A_{1}$.  In particular, they are monic formally quasi-elliptic as well.
\item\label{c}
If $P_1, \ldots , P_n\in D_n$ then $L_1,\ldots ,L_n\in  E_n$.
\end{enumerate} 
\end{lem}

\begin{proof}
The proof is an easy generalisation of analogous claim in \cite[Lemma 2.9]{Zheglov2013}. 

\ref{a}. We can find each coefficient of the operator $L_n=\partial_n+u_0+u_{-1}\partial_n^{-1}+\ldots $, where $u_i\in \hat{D}_n^n$ step by step, by solving the system of equations, which can be obtained by comparing the coefficients of $P_n$ and $L_n^{l_n}$: if $P_n=\sum_{i=1}^{l_n}p_i\partial_n^i$ with $p_{l_n}=1$, then
\begin{equation}
\label{vs5}
l_n u_0=p_{l_n-1}, \mbox{\quad} l_n u_{-i}+F(u_0,\ldots ,u_{-i+1})=p_{l_n-1-i},
\end{equation}
where $F$ is a polynomial in $u_0,\ldots ,u_{-i+1}$ and their derivatives. 
Clearly, this system is uniquely solvable. So, the operator $L_n$ is uniquely defined. Note that $L_n$ is invertible in $\hat{E}_n$ and $\ord_{\Gamma} (L_n^{-1})=(0,\ldots ,0,-1)$. Therefore, the operators $L_i=P_iL_n^{-l_i}$ for $i<n$ are also uniquely defined. 

The same arguments show that item \ref{c} is true.  

The proof of item \ref{b} is by induction. If $P_1, \ldots , P_n$ satisfy the condition $A_1$, then from equations \eqref{vs5} we get $\Ord (u_0)\le 1$ and $\Ord (u_{-i})\le 1+i$, because $F(u_0,\ldots ,u_{-i+1})$ is the coefficient at $\partial_n^{-i}$ of the expression $(\partial_n+u_0+\ldots + u_{-i+1}\partial_n^{-i+1})^{l_n})$ whose order is $\ge 1+i$ by induction (and clearly $\Ord (p_{l_n-1-i})\le 1+i$). Then $\Ord (L_n) =1$ (since $\Ord (\partial_n)=1$). Obviously, $\ord_{\Gamma}(L_n)=(0,\ldots ,0,1)$ and therefore $L_n$ satisfies the condition $A_1$.
\end{proof}

\subsection{Normalized quasi-elliptic operators}
\label{S:normalized_op}

\begin{defin}
\label{D:defin2.19}
We say that monic formally quasi-elliptic operators $P_1, \ldots , P_n \in \hat{E}_n$  are {\it almost normalized} if 
$$P_n=\partial_n^{l_n}+ \sum_{s=-\infty}^{l_n-1}p_{n,s}\partial_n^{s} \mbox{\quad} P_i=\partial_i\partial_n^{l_i}+ \sum_{s=-\infty}^{l_i-1}p_{i,s}\partial_n^{s}, \quad i<n$$ where $p_{j,s}\in \hat{D}_n^n$.  

We say that $P_1, \ldots , P_n$ are {\it normalized} if 
$$P_n=\partial_n^{l_n}+ \sum_{s=-\infty}^{l_n-2}p_{n,s}\partial_n^{s} \mbox{\quad} P_i=\partial_i\partial_n^{l_i}+ \sum_{s=-\infty}^{l_i-1}p_{i,s}\partial_n^{s},\quad i<n$$ where $p_{j,s}\in \hat{D}_n^n$.  

We call a commutative subring $B\subset \hat{E}_n$ {\it normalized} if it contains a set of normalized quasi-elliptic operators.
\end{defin}

\begin{cor}
\label{C:normalised_roots}
Let $P_1,\ldots , P_n\in \hat{D}_n^{sym}$ be monic formally quasi-elliptic operators from lemma \ref{L:lemma2.9}. 

They are normalised (almost normalised) iff the corresponding operators $L_1, \ldots , L_n$ are normalised (almost normalised). 
\end{cor}

\begin{proof}
This claim follows directly from the proof of lemma \ref{L:lemma2.9}.
\end{proof}

\begin{lem}
\label{L:lemma7}
For any monic formally quasi-elliptic operators $P_1, \ldots , P_n \in \hat{E}_n$ with $\ord_{\Gamma}(P_i)=(0,\ldots ,0,1,0,\ldots ,l_i)$, $i<n$, where $1$ stands at the $i$-th place, and $\ord_{\Gamma}(P_n)=(0,\ldots ,0, l_n)$, we have

\begin{enumerate}
\item\label{1a}
There exists an invertible function $f\in \hat{R}_y$ such that the operators \\
$f^{-1}P_1f, \ldots f^{-1}P_nf$ are almost normalized.
\item\label{1b}
If  $P_1, \ldots , P_n$ are almost normalised, then there exists an invertible operator $S\in \hat{D}_n^n$  of the form $S=\exp (\int pdx_n)$ with $\Ord (S)=0$, where  $\partial_i(p)=0$ for $i<n$,  such that the operators $S^{-1}P_1S, \ldots , S^{-1}P_nS$ are normalized. 

Combining with previous item, for any monic formally quasi-elliptic operators \\
$P_1, \ldots , P_n \in \hat{E}_n$ there exists an invertible operator $S\in \hat{D}_n^n$ with $\Ord (S)=0$ such that the operators $S^{-1}P_1S, \ldots , S^{-1}P_nS$ are normalized. 
\item\label{2*}
If $S_1$ is another operator with such a property, then all coefficients of the operator $S^{-1}S_1$ are from $K_y$. In particular, if $K_y=K$, then $S^{-1}S_1\in K$.
\end{enumerate}

\end{lem}

\begin{proof} The proof is an easy generalisation of analogous claim in \cite[Lemma 2.10]{Zheglov2013}. 

\ref{1a}) Let's show that there exists an invertible function $f\in \hat{R}_y$ such that 
\begin{equation}
\label{vs2}
f^{-1}P_nf=\partial_n^{l_n}+ \sum_{s=-\infty}^{l_n-1}p'_{n,s}\partial_n^{s}, \mbox{\quad} f^{-1}P_if=\partial_i\partial_n^{l_i}+ \sum_{s=-\infty}^{l_i-1}p'_{i,s}\partial_n^{s}.
\end{equation}
Let $P_i=\sum_{s=-\infty}^{l_i}p_{i,s}\partial_n^{s}$ and $p_{i,l_i}=\partial_i\partial_n^{l_i}+g_i$, $g_i\in \hat{R}_y$. Since the operators $P_i$ commute, $g_i$ don't depend on $x_n$ for any $i<n$. 

Easy direct computations show that for any invertible function $f\in\hat{R}_y$ we have 
$$
f^{-1}P_nf=\partial_n^{l_n}+ \sum_{s=-\infty}^{l_n-1}p'_{n,s}\partial_n^{s}, \mbox{\quad} f^{-1}P_if=(\partial_i+ f^{-1}\partial_i(f)+g_i)\partial_n^{l_i}+ \sum_{s=-\infty}^{l_i-1}p'_{i,s}\partial_n^{s}, \quad i<n
$$
with some coefficients $p'_{j,s}\in \hat{D}_n^n$. Now we can construct the needed function by induction. 
First, we can find a function $f_1$ in the form $f_1=\exp (-\int g_1dx_1)$. After conjugating all operators by $f_1$ we get new operators $P_1, \ldots, P_n$, where $P_1$ has the from from equation \eqref{vs2}. Since the operators $P_i$ commute, the new functions $g_i$, $i>1$ do not depend on $x_n,x_1$ (as easy direct calculations with the highest terms show). Then we can find a function $f_2$ in the form $f_2=\exp (-\int g_2dx_2)$. Since it does not depend on $x_n, x_1$ too, after conjugating all operators by $f_2$ we get new operators $P_1, \ldots, P_n$, where $P_1,P_2$ has the from from equation \eqref{vs2}. Continuing this procedure, we can find the needed function as the product $f:= f_1\ldots f_{n-1}$. 

\ref{1b}) By item \ref{1a}, conjugating by an appropriate function $f$ we can reduce the claim to the operators $P_1, \ldots , P_n$ that look like the right hand side in (\ref{vs2}) (obviously, $\Ord (f)=0$). 

Then direct calculations show that for any $i<n$
$$
0=[P_i,P_n]= \partial_i(p_{n,l_n-1})\partial_n^{l_n+l_i-1}+A,
$$
where $A$ consists of monomials with $\ord_n$-order less that $l_n+l_i-1$, that is the expression $p_{n,l_n-1}$ don't depend on $x_i$, $i<n$. Therefore, we may look for an operator $S$ such that $\partial_i(S)=0$, $i<n$. Direct computations (note that such an $S$ commutes with 
$p_{n,l_n-1}\in \hat{D}_n^n$) show that 
\begin{equation}
\label{E:bla}
S^{-1}P_nS=\partial_n^{l_n}+ (p_{n,l_n-1} + l_nS^{-1}\partial_n(S))\partial_n^{l_n-1}+ \sum_{s=-\infty}^{l_n-2}p'_{n,s}\partial_n^{s}, \mbox{\quad} S^{-1}P_iS=\partial_i\partial_n^{l_i}+ \sum_{s=-\infty}^{l_i-1}p'_{i,s}\partial_n^{s}. 
\end{equation}
Hence, we can find a needed operator in the form $S=\exp (-\int p_{n,l_n-1}/l_n dx_n)$, where we choose the integral in such a way that $(-\int p_{n,l_n-1}/l_n dx_n)|_{\underline{x}=\underline{0}}=0$. Since $\Ord (P_n)=l_n$, we have $\Ord (p_{n,l_n-1})\le 1$ and therefore $\Ord (-\int p_{n,l_n-1}/l_n dx_n)\le 0$.
Since $p_{n,l_n-1}\in \hat{D}_n^n$ and the coefficients of the expression $(-\int p_{n,l_n-1}/l_n dx_n)$ don't depend on $x_1, \ldots , x_{n-1}$, the exponent is well defined and its order is zero. clearly, $S$ is invertible. Since the conjugation by $f$ and $S$ preserves the $\ord_{\Gamma}$-orders and $\Ord$-orders of operators $P_i$, the normalised operators are also formally quasi-elliptic. 

\ref{2*}) Let $S':=S^{-1}S_1$. Then conjugation by $S'$ preserves the normalised form of operators $P_1, \ldots ,P_n$. Then from \eqref{E:bla} it follows that $\partial_n(S')=0$. Clearly, we must also have  $(S')^{-1}\partial_i S'=\partial_i$, where from $\partial_i (S')=0$ for all $i\le n$. So, the coefficients of the operator $S'$ are from $K_y$. If $K_y=K$, then, since $\Ord (S')=0$, it follows $S'\in K$. 
\end{proof}

\subsection{Higher-dimensional analogue of the Schur theorem}
\label{SS:Schur_theorem}

\begin{thm}
\label{T:lemma8}
Let $L_1, \ldots , L_n\in \hat{E}_n$ be  monic almost normalized formally quasi-elliptic operators with $\ord_{\Gamma} (L_i)=(0,\ldots ,0, 1, 0,\ldots ,0)$, where $1$ stands at the $i$-th place, $1\le i\le n$:
$$ L_n=\partial_n + v_{n,0} + \sum_{q=1}^{\infty}v_{n,q}\partial_n^{-q} \quad  L_i=\partial_i + \sum_{q=1}^{\infty}v_{i,q}\partial_n^{-q}, \quad i<n.$$
 Then 
\begin{enumerate}
\item\label{1}
There exists a monic invertible operator $S=1+S^-$, where $S^-\in \hat{D}_n^n[[\partial_n^{-1}]]\partial_n^{-1}$, with $\ord_{\Gamma}(S)=\ord_{\Gamma}(S^{-1})=(0,0)$ such that $S^{-1}\partial_iS=L_i$,  $1\le i< n$, $S^{-1}L_{n,0}S=L_n$, where $L_{n,0}=\partial_n +v_{n,0}$. 

Moreover, $S$ satisfies the condition $A_{1}$, i.e. in particular $\Ord (S)=|\ord_{\Gamma}(S)|=\Ord (S^{-1})=|\ord_{\Gamma}(S^{-1})|=0$
\item\label{2}
If $S_1$ is another operator with such a property, then 
$$SS_1^{-1}\in K_y [[\partial_1, \ldots ,\partial_{n-1}]]((L_{n,0}^{-1}))\cap \Pi_n$$ 
and satisfies $A_{1}$. In particular, if $K_y=K$ or $SS_1^{-1}\in E_n$, then $SS_1^{-1}\in K_y [\partial_1, \ldots ,\partial_{n-1}]((L_{n,0}^{-1}))$. 

\item
\label{3}
If $L_1, \ldots ,L_n\in  E_n$, then $S\in  E_n$.
\end{enumerate}

\end{thm}

\begin{proof}  The proof is an easy generalisation of analogous claim in \cite[Lemma 2.11]{Zheglov2013}, cf. also analogous theorem in \cite{Parshin2000}.

\ref{1}) It suffices to prove the following fact: if 
$$
L_n=\partial_n + v_{n,0} + \sum_{q=k\ge 1}^{\infty}v_{n,q}\partial_n^{-q} \quad 
L_i=\partial_i + \sum_{q=k\ge 1}^{\infty}v_{i,q}\partial_n^{-q}, \quad 1\le i,j< n, \mbox{\quad } [L_i,L_j]=0, $$ 
then there exists an operator $S_k=1+s_k\partial_n^{-k}$ with $s_k\in\hat{D}_n^n$ such that 
$$S_k^{-1}L_iS_k=\partial_i + \sum_{q=k+1}^{\infty}v'_{i,q}\partial_n^{-q}, \quad \quad 1\le i< n, \quad S_k^{-1}L_nS_k=\partial_n +v_{n,0} + \sum_{q=k+1}^{\infty}v'_{n,q}\partial_n^{-q}.$$

Indeed, if this fact is proved, then $S^{-1}=\prod_{q=1}^{\infty}S_q$, where $S_1$ is taken for given $L_1, \ldots , L_n$, $S_2$ is taken for $S_1^{-1}L_iS_1$, and so on. It is not difficult to see that the sequence $\{\prod_{q=1}^{m}S_q\}$ is convergent in $\hat{E}_n$ (with respect to the topology given by pseudo-valuation $-\ord_n$), i.e. the infinite product is well defined. 

To prove the fact let's note first that, since $[L_i,L_j]=0$, it follows $\partial_i(v_{j,k})- \partial_j(v_{i,k})=0$ for $i,j \neq n$ and $\partial_n(v_{j,k})- \partial_j(v_{n,k})+
[v_{n,0},v_{j,k}]=0$, $\partial_j(v_{n,0})=0$ for $j\neq n$. Besides,
$$
S_k^{-1}\partial_iS_k=\partial_i +S_k^{-1}\partial_i(S_k )=\partial_i + \partial_i(s_k)\partial_n^{-k}+\ldots , \quad i<n 
$$
$$
S_k^{-1}L_{n,0}S_k=\partial_n +S_k^{-1}\partial_n(S_k )+S_k^{-1}v_{n,0}S_k=\partial_n + v_{n,0} + (\partial_n(s_k)+[v_{n,0},s_k])\partial_n^{-k}+\ldots ,
$$
whence $s_k$ can be found from the following system:
\begin{equation}
\label{sv2}
\partial_i(s_k)=-v_{i,k} \mbox{\quad} \partial_n(s_k)+[v_{n,0},s_k]=-v_{n,k}, \quad 1\le i< n. 
\end{equation}
This system is solvable, because it is compatible and all coefficients of $v_{i,k}$ belong to $\hat{R}_y$ (so that we can always integrate them). 

Since $\Ord (L_i)=|\ord_{\Gamma}(L_i)|=1$ for all $1\le i\le n$, we have $\Ord (v_{i,k})\le 1+k$ and therefore $\Ord (s_k)\le k$. Hence $\Ord S_k=0$  (since $\Ord (s_k\partial_n^{-k})\le 0$ and $\Ord (1)=0$). Therefore, $\Ord (\prod_{q=1}^{\infty}S_q)=0$. Clearly, the operator $(\prod_{q=1}^{\infty}S_q)$ is invertible, because it is of the form $1+S^-$, where $S^-\in \hat{D}_n^n[[\partial_n^{-1}]]\partial_n^{-1}$. By the same reason its $\Gamma$-order is $(0,\ldots ,0)$, i.e. the operator $S:=(\prod_{q=1}^{\infty}S_q)^{-1}$ satisfies the condition $A_1$. The rest of the claim follows from corollary \ref{C:corol2.1}.

\ref{2}) Clearly, $[S S_1^{-1}, \partial_i]=0$ for any $1\le i< n$ and $[S S_1^{-1}, L_{n,0}]=0$. Note that any series in $\hat{E}_n$ can be written as a series in $L_{n,0}^{-1}$ with coefficients in $\hat{D}_n^n$ and vice versa (note that $L_{n,0}^{-1}=\partial_n^{-1}(1+v_{n,0}\partial_n^{-1})^{-1}$ and $\partial_n^{-1}=L_{n,0}^{-1}(1-v_{n,0}L_{n,0}^{-1})^{-1}$). Therefore, the operator $S S_1^{-1}$ written as a series in $L_{n,0}^{-1}$ has constant coefficients, for all these coefficients commute with $\partial_i$, $i<n$ and with $v_{n,0}$ (as they don't depend on $\partial_n$ and on $x_i$ with $i<n$) and therefore must commute also with $\partial_n$, i.e. it has coefficients from $K_y$. 

By lemma \ref{L:lemma2.7} and corollary \ref{C:corol2.1} it satisfies the condition $A_1$. So, if $K_y=K$, then $SS_1^{-1}\in K[\partial_1, \ldots ,\partial_{n-1}]((\partial_n^{-1}))$. 

\ref{3}) The proof is the same as in item \ref{1}). Note that in this case $L_{n,0}, L_{n,0}^{-1}\in E_n$.
\end{proof}

An immediate corollary of all these technical statements is 

\begin{cor}
\label{C:corol3.1}
Let $P_1,\ldots , P_n\in \hat{E}_n$ be monic formally quasi-elliptic operators. 

The set of commuting with $P_1,\ldots , P_n$ operators in $\hat{E}_n$ is a commutative ring, which can be embedded into the ring $K_y [[\partial_1, \ldots ,\partial_{n-1}]]((\partial_n^{-1}))$. 

Analogously, the set of commuting with $P_1,\ldots , P_n$ operators in $\Pi_n$ is a commutative ring, which can be embedded into the ring $K_y [[\partial_1, \ldots ,\partial_{n-1}]]((\partial_n^{-1}))\cap \Pi_n$. 

If $K_y=K$, then it can be embedded into the ring $K[\partial_1, \ldots ,\partial_{n-1}]((\partial_n^{-1}))$ \\
($K [\partial_1, \ldots ,\partial_{n-1}]((\partial_n^{-1}))\cap \Pi_n$ respectively). 

\end{cor}

\begin{proof}
By lemma \ref{L:lemma7} there exists an invertible operator $S\in \hat{D}_n^n$ such that the operators $P_i':=S^{-1}P_iS$ are normalised. By lemma \ref{L:lemma2.9} and corollary \ref{C:normalised_roots} the corresponding operators $L_1, \ldots , L_n$ are normalised as well. By theorem \ref{T:lemma8} there is an invertible operator $S_1$ such that $S_1L_iS_1^{-1}=\partial_i$. In particular,  $S_1P_n'S_1^{-1}=\partial_n^{l_n}$ and $S_1P_i'S_1^{-1}= \partial_i\partial_n^{l_i}$ for $i<n$. 

If $[Q,P_i]=0$, then $[S_1S^{-1} Q SS_1^{-1}, \partial_n^{l_n}]=0$ and $[S_1S^{-1}QSS_1^{-1}, \partial_i\partial_n^{l_i}]=0$ for $i<n$. Put $Q':= S_1S^{-1} Q SS_1^{-1}$, and assume $Q'=\sum_{j=N}^{\infty} q_j\partial_n^{-j}$ with $q_j\in \hat{D}_n^n$. Assume that $q_l$ is the first coefficient not belonging to the ring $K_y [[\partial_1, \ldots ,\partial_{n-1}]]$. Then
$$
0=[Q',\partial_n^{l_n}]= -l_i \partial_n (q_l)\partial_n^{l_i-1-l} + \mbox{terms of lower $\ord_n$-order,}
$$
where from $\partial_n (q_l)=0$. Analogously,
$$
0=[Q',\partial_i\partial_n^{l_i}]= -\partial_i(q_l)\partial_n^{l_i-l}+ \mbox{terms of lower $\ord_n$-order,}
$$
where from $\partial_i (q_l)=0$ for all $i<n$. Hence, $q_l$ has constant coefficients, $q_l\in K_y[[\partial_1, \ldots ,\partial_{n-1}]]$ -- a contradiction. If $K_y=K$, then, since $\Ord (q_l)< \infty$, we must have $q_l\in K [\partial_1, \ldots ,\partial_{n-1}]$. 

So, $Q'\in K_y [\partial_1, \ldots ,\partial_{n-1}]((\partial_n^{-1}))$ and therefore the set of commuting with $P_1,\ldots , P_n$ operators in $\hat{E}_n$ form a commutative subring, which can be embedded into the ring $K_y [\partial_1, \ldots ,\partial_{n-1}]((\partial_n^{-1}))$.

If $Q\in \Pi_n$, then $Q'\in \Pi_n$ too, since $S, S_1$ are of finite order. So, \\
$Q'\in K_y [[\partial_1, \ldots ,\partial_{n-1}]]((\partial_n^{-1}))\cap \Pi_n$.
\end{proof}

\begin{rem}
\label{R:Schur_oper}
The embedding from corollary is made with the help of conjugation by an operator $S$. We will call it a {\it Schur operator}. Note that it has an invertible symbol, thus we get  $\sigma (S^{-1}QS)= \sigma (S)^{-1}\sigma (Q)\sigma (S)\in K_y [[\partial_1, \ldots ,\partial_{n-1}]]((\partial_n^{-1}))\cap \Pi_n$ for any operator $Q$ commuting with $P_1, \ldots , P_n$. Denote by $B$  the ring of operators commuting with $P_1, \ldots , P_n$. Then we get  an embedding of the ring $\gr (B)$ into the ring of pseudo-differential operators with  constant coefficients, where  $\gr (B)$ denotes the associated graded ring with respect to the filtration defined by the function $\Ord$. 
\end{rem}

Another corollary is the result on "purity" of rings of differential operators. 

\begin{cor}
\label{C:purity}
Assume $K_y=K$. Let $B\subset D_n\subset \hat{D}_n$ be a quasi-elliptic ring of commuting partial differential operators. Then the centraliser $C(B)\subset \hat{D}_n$  is a ring of partial differential operators, i.e. $C(B)\subset D_n$.
\end{cor} 

\begin{proof}
The proof is an easy generalisation of \cite[Prop.3.1]{Zheglov2013}. 

By lemma \ref{L:lemma7} there exists an invertible function $f$ such that the operators $P_i':=f^{-1}P_if$ are almost normalised, where $P_i$ are quasi-elliptic operators from $B$. By lemma \ref{L:lemma2.9} and corollary \ref{C:normalised_roots} the corresponding operators $L_1, \ldots , L_n\in E_n$ are almost normalised as well.

Then the same arguments as in corollary \ref{C:corol3.1} combined with theorem \ref{T:lemma8}, item 3 show that there is an operator $S\in E_n$ such that $SBS^{-1}\subset K_y[\partial_1, \ldots ,\partial_{n-1}]((L_{n,0}^{-1}))$. Namely, If $[Q,P_i']=0$, then $[S Q S^{-1}, L_{n,0}^{l_n}]=0$ and $[SQS^{-1}, \partial_i L_{n,0}^{l_i}]=0$ for $i<n$. Put $Q':= S Q S^{-1}$, and assume $Q'=\sum_{j=N}^{\infty} q_jL_{n,0}^{-j}$ with $q_j\in \hat{D}_n^n$. Assume that $q_l$ is the first coefficient not belonging to the ring $K [\partial_1, \ldots ,\partial_{n-1}]$. Then (note that $[q_l,L_{n,0}]\in \hat{D}_n^n$)
$$
0=[Q',L_{n,0}^{l_n}]= -l_i [q_l,L_{n,0}]\partial_n^{l_i-1-l} + \mbox{terms of lower $\ord_n$-order,}
$$
where from $[q_l,L_{n,0}]=0$. Analogously,
$$
0=[Q',\partial_iL_{n,0}^{l_i}]= -\partial_i(q_l)\partial_n^{l_i-l}+ \mbox{terms of lower $\ord_n$-order,}
$$
where from $\partial_i (q_l)=0$ for all $i<n$. Hence, $q_l$ has constant coefficients. Since $\Ord (q_l)< \infty$, we must have $q_l\in K [\partial_1, \ldots ,\partial_{n-1}]$ -- a contradiction.

Then we have also 
$$SC(B)S^{-1}\subset K[\partial_1, \ldots ,\partial_{n-1}]((L_{n,0}^{-1}))\subset E_n.$$
Therefore, $C(B)\subset \hat{D}_n\cap E_n=D_n$.
\end{proof}

\subsection{Admissible operators}
\label{S:admissible_op}

Yet another corollary is the description of the so called admissible operators. They appear naturally as compositions of invertible operators $S, S^{-1}$ from lemma \ref{L:lemma7} and theorem \ref{T:lemma8} for different choices of formally quasi-elliptic operators $P_1, \ldots ,P_n$. We will need this description  to prove the classification theorems in section \ref{S:classification}.
In this section we give only preparatory result, the final result will be given in proposition \ref{P:admissible} below.

\begin{defin}{(cf. \cite[Def.3.3]{Zheglov2013})}
\label{D:admissible}
An operator $T\in \Pi_n$ is said to be 1-admissible if it is an invertible operator, $\ord_n(T)=0$,  such that  \\
$T^{-1}\partial_iT\in K_y[[\partial_1, \ldots , \partial_{n-1}]]((\partial_n^{-1}))\cap \Pi_n$ for $1\le i\le n$. The set of all admissible operators in $\hat{E}_n$ is denoted by $\Adm_1 (\hat{E}_n)$. 
\end{defin}

\begin{lem}
\label{L:conjugation_preserves_order}
Let $T\in \Pi_n$ be a 1-admissible  operator.

Then it holds
$$
\ord_n (T^{-1}qT)= \ord_n (q), \quad \Ord (T^{-1}qT)= \Ord (q) \quad \forall q\in  K_y[[\partial_1, \ldots , \partial_{n-1}]]((\partial_n^{-1}))\cap \Pi_n.
$$

In particular, the operators $T^{-1}\partial_iT$ satisfy the condition $A_1$. 
\end{lem}

\begin{proof}
Let $q'=T^{-1}q T$. Assume, say, $s= \ord_n(q')\neq \ord_n(q)$. Since $Tq'=qT$ and ${HT}_n(q')$, being an operator with constant coefficients, is not a right zero divisor, we have, by section \ref{S:algebra_D_n}, 
\begin{equation}
\label{E:inequality}
\ord_n(Tq')=\ord_n(T)+\ord_n(q')\le \ord_n(q)+\ord_n(T),
\end{equation}
i.e. $s=\ord_n(q')\le \ord_n(q)$. If we take an invertible $q$ in $\hat{E}_n$ (say $q=\partial_n$), 
then $\ord_n(T^{-1}q^l T)=\ord_n((T^{-1}q T)^l)=ls$ for any $l\in \dz$, because $T^{-1}q T$ has constant coefficients. On the other hand,  we have by section \ref{S:algebra_D_n}
$$
sl=\ord_n(T^{-1}q^l T)\le \ord_n(T^{-1})+\ord_n(q^l)+\ord_n(T)=\ord_n(T^{-1})+l\ord_n(q),
$$
which is impossible if $l\ll 0$ and $s\neq \ord_n(q)\neq 0$, since $T^{-1}\in \hat{E}_n$ has finite $\ord_n$-order. 
If $\ord_n (q)=0$ and $s<0$, then by the same arguments,
$$
0=\ord_n(q^l)=\ord_n(T(q')^lT^{-1})\le ls+\ord_n(T^{-1}),
$$ 
which is impossible if $l\gg 0$.  Since any $q$ is a product of an invertible element ($\partial_n^k$) and an element of $\ord_n$-order zero, we are done. 

Now let's show that $\Ord (T^{-1}\partial_iT)=1$ for $1\le i\le n$. Let $q_i'=T^{-1}\partial_iT$. By the same arguments as above (see \eqref{E:inequality}) we have $\Ord (q_i')\le \Ord (\partial_i)=1$ (note that we use only that $T$ has finite $\Ord$-order, since $T\in \Pi_n$). Since $\ord_n(q_n')=1$ we have $q_n'=c_n\partial_n + l.o.t.$, where l.o.t. mean terms with $\ord_n$-order less than one. Note that $c_n=1$ (it is easy to see directly by comparing the coefficients  of operators $\partial_n T$ and $T(c_n\partial_n + l.o.t.)$). So, we have $\Ord (q_n')\ge 1$, hence $\Ord (q_n')=1$.  

 If $\Ord (q_i')\le 0$, then ${HT}_n(q_i')\in K \quad \mod \quad \idm\cap K_y$. If ${HT}_n(q_i') = c_i+a_i$, where $c_i\in K$ and $a_i\in \idm\cap K_y$, then $q_i'=c_i+a_i+w_i$, where $\ord_n(w_i)<0$. Then $w_i+a_i=T^{-1}(\partial_i-c_i)T$, but in this case $\ord_n(Tw_iT^{-1})< 0$, $Ta_iT^{-1}\in \idm\cap K_y$ and $\partial_i-c_i= Tw_iT^{-1} + Ta_iT^{-1}$ -- a contradiction with $\ord_n(\partial_i-c_i)=0$, $\partial_i-c_i\notin \idm$. Thus, $\Ord (q_i')= 1$. 
 
 Clearly, then we should have $\Ord (T^{-1}qT)= \Ord (q)$ for all $q$.
\end{proof}

\subsection{Algebraic properties of quasi-elliptic rings}
\label{S:algebraic_properties}
 
In this section we give another application of the Schur theory -- a description of basic algebraic properties of quasi elliptic subrings. 

\begin{prop}
\label{P:quasi-ell-properties}
Let $B$ be a quasi elliptic commutative subring in $\hat{D}_n^{sym}$. Then
\begin{enumerate}
\item
$B$ and $gr(B)$ are integral, where $\gr$ denotes the associated graded ring with respect to the filtration defined by the function $\Ord$, and the function $-\Ord$ induces a discrete valuation of rank one on $B$ and on its field of fractions $\Quot (B)$ as well as on $gr(B)$ embedded into $\hat{D}_n^{sym}$, cf. remark \ref{R:Schur_oper};
\item
if $K_y=K$, then the $\Gamma$-order is defined on all elements of $B$, in particular, the function $-\ord_{\Gamma}$ is a discrete valuation of rank $n$;
\item
the natural map 
$$
\Phi: gr(\hat{D}_n^{sym}) \rightarrow gr(\hat{D}_n^{sym})/(x_1, \ldots , x_n) gr(\hat{D}_n^{sym})\simeq K_y[\xi_1,\ldots ,\xi_n]
$$
induces an embedding of vector spaces on $gr(B)$;
\item
the spectral module $F$ is torsion free;
\item
if $K_y=K$ then for any $P\in B$ holds: $\Ord (P)=|\ord_{\Gamma}(\sigma(P))|$;
\item
$\trdeg_{K'}(\Quot (B))=n$, where $K'=\Quot (K_y)$, the field $\Quot (B)$ is finitely generated over $K'$ and the localisation $\Quot (B)\cdot F$ is a finitely generated $\Quot (B)$-module. 
\end{enumerate}
\end{prop}

\begin{rem}
\label{R:finitely_gen}
Unlike the case $n=1$ quasi-elliptic rings are not necessarily finitely generated. The most simple example is the subring $K[1,\partial_1^i\partial_2^j, i\ge 0, j>0]\subset \hat{D}_2$. More interesting examples can be found in \cite{KasmanPreviato}. 
\end{rem}

\begin{proof}
1. The first item follows from corollary \ref{C:corol3.1} and remark \ref{R:Schur_oper}.

2. By lemma \ref{L:lemma7} there exists an invertible function $f$ such that the operators $P_i':=f^{-1}P_if$ are almost normalised, where $P_i$ are quasi-elliptic operators from $B$. By lemma \ref{L:lemma2.9} and corollary \ref{C:normalised_roots} the corresponding operators $L_1, \ldots , L_n\in E_n$ are almost normalised as well. Note that the conjugation by function preserves the $\Gamma$-order.  

Then the same arguments as in corollary \ref{C:purity} combined with theorem \ref{T:lemma8}, item 3 show that there is an operator $S\in \hat{E}_n$ such that $SBS^{-1}\subset K[\partial_1, \ldots ,\partial_{n-1}]((L_{n,0}^{-1}))$. Since $\ord_{\Gamma}(S)= (\underline{0})$ and the $\Gamma$-order is defined on all elements of the ring $K[\partial_1, .. ,\partial_{n-1}]((L_{n,0}^{-1}))$, we obtain the claim of this item.

3-5. The proof of items 3-5 is an easy generalisation of \cite[Lemma 2]{Zheglov2018}.

By lemma \ref{L:lemma7} in combination with theorem \ref{T:lemma8} and corollary \ref{C:corol3.1} there exists an invertible operator $S$ of the form $S=S_1S_2$, where $S_1\in \hat{D}_n^n$ is invertible and $S_2\in \hat{E}_n$  satisfies the condition $A_1$, such that $SBS^{-1}\in K_y[\partial_1, \ldots ,\partial_{n-1}]((\partial_n^{-1}))$. Moreover, $S_1=fS_1'$, where $f$ is an invertible function and $\partial_i(S_1')=0$ for $i<n$.

Since $S$ is invertible, for any non-zero element $Q\in B$ we have $\Ord (SQS^{-1})=\Ord (Q)$. But the symbol of $SQS^{-1}$ has constant coefficients, therefore, the image $\sigma'(Q)=\Phi (\sigma (Q))\in K_y[\xi_1,\ldots ,\xi_n]$ is not zero.
Indeed, on the one hand side we have 
$$HT_n(\sigma (S)\cdot \sigma (Q)\cdot \sigma (S)^{-1})=HT_n(\sigma (S))\cdot HT_n(\sigma (Q))\cdot HT_n(\sigma (S)^{-1}).
$$ 
On the other hand, 
$HT_n(\sigma (S)\cdot \sigma (Q)\cdot \sigma (S)^{-1})=C_1\in K_y[\partial_1, \ldots , \partial_{n-1}]$, $\sigma (S)=\sigma (S_1)\cdot \sigma (S_2)$, and $HT_n(\sigma (S))=\sigma (S_1)$. Thus, $HT_n(\sigma (Q))=\sigma (S_1)^{-1}\cdot C_1  \cdot \sigma (S_1)=C_1$, because $\sigma (S_1)=\sigma (f)\sigma (S_1')=c \sigma (S_1')$, where $c\in K_y$, and  $\partial_i(S_1')=0$ for $i<n$ (i.e. $C_1$ commutes with $\sigma (S_1)$). This means that $\sigma '(Q)\neq 0$ in $K_y[\xi_1,\ldots ,\xi_n]$. Thus, $gr(B)$ is embedded in $K_y[\xi_1,\ldots ,\xi_n]$ as a {\it vector space}. 

The module $F$ is a torsion free $B$-module, because for any non-zero $f\in F$ we have $f S\notin (x_1, \ldots , x_n) \hat{E}_n$ and for any non-zero $b\in B$ 
$$
fb=fS^{-1}(SbS^{-1})S\notin  (x_1, \ldots , x_n) \hat{E}_n.
$$

If $K_y=K$, then from the arguments above it follows $\Ord (\sigma (Q))=|\ord_{\Gamma}(\sigma (Q))|$ and therefore $\Ord (Q)=|\ord_{\Gamma}(\sigma (Q))|$ for any $Q\in B$.

6. The proof of this item  is an easy generalisation of  \cite[Lemma 2]{OsipovZheglov}. 

Consider the subspace $\bigtriangleup (N)= \{a\in F | \Ord (a)<  N\}$, where $N \in \dn$. Note that this subspace has a finite dimension over the field $K'$ and this dimension is not greater than $N^n$.

Suppose that the localisation $\Quot (B)\cdot F$ is infinitely generated. This, in particularly, means that there are infinitely many linearly independent over  $K'[P_1, \ldots ,P_n]$ elements in $F$, where $P_1, \ldots ,P_n\in B$ are formally quasi-elliptic operators. Obviously, $P_1, \ldots ,P_n$ are algebraically independent (cf. theorem \ref{T:lemma8}). 

We can estimate the dimension of a subspace generated over $K'$ by
powers of $P_1, \ldots ,P_n$ which lie in the subspace $\bigtriangleup (N)$ for some $N$.  It's clear that this dimension looks like a polynomial of degree $n$ in $N$ for $N\gg 0$, say 
$$
\dim_{K'} (\bigtriangleup (N))= e_n N^n+e_{n-1}N^{n-1}+\ldots +e_0
$$
for some numbers $e_i$. 

Consider $M$ linearly independent over  $K'[P_1, \ldots ,P_n]$ elements $a_1,\ldots , a_M\in F$, where $M$
satisfies the condition $Me_n > 1$. Without loss of generality we can assume that
$\Ord (a_1)\ge \Ord (a_i)$ for all $i$. Since the dimension of a subspace 
generated over $K'$ by powers of $P_1, \ldots ,P_n$ which lie in the subspace $\bigtriangleup (N-M\Ord (a_1))$ is $e_n(N-M\Ord (a_1))^n+\ldots +e_0$ for big $N$, we obtain that the dimension of a 
subspace generated over $K'$ by elements $a_1,\ldots , a_M$ multiplied to powers of $P_1, \ldots ,P_n$ which lie in the subspace $\bigtriangleup (N-M\Ord (a_1))$ is not less than $M(e_n(N-(M-1)\Ord (a_1))^n+\ldots +e_0)$. 
The last subspace is inside $\bigtriangleup (N)$. On the other hand, since $Me_n>1$, we have for $N$ sufficiently large 
$M(e_n(N-(M-1)\Ord (a_1))^n+\ldots +e_0)> N^n> \dim_{K'} \bigtriangleup (N)$, a contradiction. \\

\end{proof}

\section{Sato theory for $\hat{D}_n^{sym}$}
\label{S:Sato_theory}

Another important technical tool in investigating commutative subrings of $\hat{D}_n^{sym}$ is a higher dimensional analogue of the Sato theory in dimension one. In this section we collect technical statements of this theory for the ring  $\hat{D}_n^{sym}$. They improve corresponding results appeared in \cite[\S 5]{BurbanZheglov2017}. A review of the Schur theory for the ring $D_1$ see in \cite{Mul} and in \cite{Quandt} for the  relative case. 

\subsection{Regular operators and units}
\label{S:regular_op}

Consider the ring 
$$V_n:= K_y\{\{\partial_1, \ldots , \partial_{n-1}\}\}((\partial_n^{-1})),$$ 
where $K_y\{\{\partial_1, \ldots , \partial_{n-1}\}\}=K_y[[\partial_1, \ldots , \partial_{n-1}]]\cap \hat{D}_n^n$. 
It has a structure of a right $\hat{E}_n$-module via the isomorphism of vector spaces $V_n\simeq \hat{E}_n/(x_1, \ldots ,x_n) \hat{E}_n$. We will extend all definitions relevant to the ring $\hat{E}_n$ (e.g. the notions of orders, condition $A_1$, etc.) on the ring $V_n$.

 Using the anti-lexicographical order on the group $\dz\oplus\ldots \oplus \dz$ we define the {\it lowest term} $LT(a)$ of any series $a$ from $V_n$ as follows: if $a=\sum_{\underline{i}\le \ord_{\Gamma}(a)}a_{\underline{i}}\underline{\partial}^{\underline{i}}$, then 
$$
LT(a) = \bar{a}_{\underline{\ord_{\Gamma}(a)}}\underline{\partial}^{\underline{\ord_{\Gamma}(a)}}\in K [\partial_1, \ldots , \partial_{n-1}]((\partial_n^{-1})),
$$ 
where $\bar{a}_{\underline{\ord_{\Gamma}(a)}}\in K$ means the residue of the element $a_{\underline{\ord_{\Gamma}(a)}}$ from $K_y$ modulo the ideal $\idm$ (obviously, $\ord_{\Gamma}$  is well defined on all elements of $V_n$). 

\begin{defin}
\label{D:defin2}
The support of a $K_y$-subspace $W\subset V_n$ is the $K_y$-subspace $\Sup (W)$ in the space $V_n$ generated by $\LT (a)$ for all $a\in W$. 
\end{defin}

\begin{defin}
\label{D:1space}
A subspace $W\subset  V_n\cap \Pi_n$ is called 1-space, if there is a basis $\{w_k\}$ in $W$ such that $\bar{w}_k:=w_k|_{y=0}$ form a basis of the space $\Sup (W)$.
\end{defin}

\begin{defin} Let $W \subseteq V_n$ be a $1$-subspace.
\begin{enumerate}
\item For any $k \in \dz$, we put:
$
W_k := \bigl\{w \in W \, \big| \, \Ord (w) \le k \bigr\}.
$
\item $H_W(k):= \dim_{K}\bigl(\bar{W}_k\bigr)$ is the \emph{Hilbert function} of $W$.
\end{enumerate}
\end{defin}

\begin{defin}
\label{D:regular}
An element $P \in \hat{\Pi}_n$ is called \emph{regular} if the $K$--linear map $F \xrightarrow{\pi (-\circ \sigma(P))} V_n$ is injective, where $\circ$ means the action on the module $\hat{\Pi}_n$ and 
$$\pi :\hat{\Pi}_n \rightarrow \hat{\Pi}_n/ (x_1, \ldots , x_n) \hat{\Pi}_n =V_n\cap \Pi_n$$ 
is the projection map. In particular, $P$ is regular if and only if its symbol $\sigma(P)$ is regular.

We'll denote by $f\diamond P:= \pi (f\circ P)$, $f\in F$, $P\in \hat{\Pi}_n$ the map above.  
\end{defin}

\begin{lem}{(cf. \cite[L.5.9]{BurbanZheglov2017})}
\label{L:regularOp}
Let $P \in \hat{\Pi}_n$. Then the following results are true.
\begin{enumerate}
\item The element $P$ is regular if and only if for any $m \in \dn_0$, the elements of the set
$\bigl\{\underline{\partial}^{\underline{k}} \diamond \sigma(P)\,  \big| \, \underline{k} \in \dn_0: |\underline{k}| = m\bigr\} \subset V_n$
are linearly independent.
\item Assume that $P$ is regular. Then $P$ is not a torsion element, i.e.~the equation $Q \circ P = 0$ for $Q\in \hat{D}_n^{sym}$ implies that $Q = 0$.
\end{enumerate}
\end{lem}

\begin{proof} (1) Let $d := \Ord (P) = \Ord \bigl(\sigma(P)\bigr)$. Then for any $\underline{k} \in \dn_0$ we have:
$\Ord\bigl(\underline{\partial}^{\underline{k}} \circ \sigma(P)\bigr) = |\underline{k}| + d$ provided $\underline{\partial}^{\underline{k}} \circ \sigma(P) \ne 0$. Therefore, the linear map $F \xrightarrow{\pi (-\circ \sigma(P))} (V_n\cap \Pi_n)\subset V_n$ splits into a direct sum of its graded components
$F_m \xrightarrow{\pi (-\circ \sigma(P))} (V_n\cap \Pi_n)_{m+d}$, implying the first statement.

\smallskip
\noindent
(2) Let $Q \ne 0$ be such that $Q \circ P = 0$. Then $\sigma(Q) \ne 0$ as well,  whereas  $\sigma(Q) \circ \sigma(P) = 0$. Next, there exists
$\underline{k} \in \dn_0$ such that $F\ni \pi (\underline{\partial}^{\underline{k}} \circ \sigma(Q)) \ne 0$ in $V_n$. On the other hand,
$$
\pi (\pi (\underline{\partial}^{\underline{k}} \circ \sigma(Q)) \circ \sigma(P)) = \pi (\underline{\partial}^{\underline{k}} \circ \sigma(Q) \circ \sigma(P)) = 0,
$$
hence $P$ is not regular, contradiction.
\end{proof}

There is a slight modification of the slice decomposition from definition \ref{D:slices} which we will use in this section and which holds also for elements from $\hat{\Pi}_n$:
\begin{equation}
\label{E:slice_decomposition1}
P = \sum\limits_{\underline{i} \ge \underline{0}} \frac{\underline{x}^{\underline{i}}}{\underline{i}!} \, P_{(\underline{i})}, \quad \mbox{\rm where} \quad \underline{x}^{\underline{i}}=x_1^{i_1}\ldots x_n^{i_n}, \quad  i!=i_1!\ldots i_n!
\end{equation}
$$
P_{(\underline{i})} = \underline{i}! \sum_{\substack{\underline{k} \in \sdn_0^{n-1}\times\sdz \\ |\underline{k}| - |\underline{i}| \le d = \Ord (P)}} \alpha_{\underline{k}, \underline{i}} \, \underline{\partial}^{\underline{k}} \in V_n\cap \Pi_n, \quad \alpha_{\underline{k}, \underline{i}}\in K_y.
$$

\begin{lem}
\label{L:statement}
The following statement is true: for any $P,Q\in \hat{\Pi}_n$ 
$$
P = Q \quad \mbox{if and only if} \quad \underline{\partial}^{\underline{i}} \diamond P = \underline{\partial}^{\underline{i}} \diamond Q
\quad \mbox{for any} \quad \underline{i} \in \dn_0^n.
$$
\end{lem}

\begin{proof}
Note that for any $\underline{i} \in \Sigma$ we have the following identity
$\underline{\partial}^{\underline{i}} \diamond P = P_{(\underline{i})}+ l.o.t.$, where least order terms are elements of the form $\underline{\partial}^{\underline{j}}\cdot P_{(\underline{k})}$ with $|\underline{j}|< |\underline{i}|$ and $|\underline{k}|< |\underline{i}|$. 

In particular, $1\diamond P=P_{(\underline{0})}$ and therefore the equalities 
$$
\underline{\partial}^{\underline{i}} \diamond P = \underline{\partial}^{\underline{i}} \diamond Q
$$
imply, by induction in $|\underline{i}|$, the equalities of all slices of $P, Q$, hence the equality of $P$ and $Q$. 
\end{proof}

\begin{defin}
\label{D:units}
Let $\hat{D}_{n,-}^{sym}:=\{P\in \hat{D}_n^{sym}| \Ord (P)\le 0\}$ and  $\hat{U}_n^{sym}\subset \hat{D}_{n,-}^{sym}$ be the group of units.
\end{defin}

\begin{lem}{(cf. \cite[L.5.11]{BurbanZheglov2017})} 
\label{L:5.11}
The following results are true.
\begin{enumerate}
\item Let $P \in \hat{D}_{n,-}^{sym}$. Then we have: $P \in \hat{U}_n^{sym}$ if and only if $\sigma(P) \in \hat{U}_n^{sym}$.
\item Let $Q \in \hat{U}_n^{sym}$. Then $\Ord (Q) = 0$ and
$\overline{Q_{(\underline{0})}} \in K^\ast$. Moreover, $Q$ is regular.
\end{enumerate}
\end{lem}

\begin{proof}
The proof of first item is obvious. In the second item it is clear that $\Ord (Q) = 0$ and
$\overline{Q_{(\underline{0})}} \in K^\ast$.
By item 1 $\sigma (Q)$ is a unit. 
If $Q$ is not regular, then $\sigma (Q)$ is not regular and therefore there exists $w\in F$ such that $w\diamond \sigma (Q)=0$.  But then 
$$
w=w\cdot \sigma (Q)\cdot (\sigma (Q))^{-1}=0 \quad \mbox{mod $(x_1, \ldots ,x_n) \hat{D}_n$},
$$
a contradiction. 
 
\end{proof}

\subsection{A higher-dimensional analogue of the Sato theorem}
\label{S:Sato_theorem}

\begin{defin}
\label{D:SatoGrassmannian} 
Let $\mu \in \dz$.
\begin{enumerate}
\item We put:
$
\mathsf{Gr}_\mu(V_n) := \Bigl\{\mbox{1-spaces $W \subseteq V_n\cap \Pi_n$} \, \Big| \, H_W(\mu +k) = C_{n+k}^n  \quad \mbox{\rm for any} \quad k \in \dz \Bigr\}
$ (recall that $H_F(k) = C_{n+k}^n$).
\item Let $W \in \mathsf{Gr}_\mu(V_n)$. Then $S \in \hat{\Pi}_n$ is called a  \emph{Sato operator} of $W$ if the following conditions are fulfilled:
\begin{itemize}
\item $S$ is regular and $\Ord (S) = \mu$.
\item We have: $W = F \diamond S$.
\end{itemize}
\end{enumerate}
\end{defin}

\begin{ex}
\label{Ex:Sato_F}
Note that $F\in \mathsf{Gr}_0(V_n)$. If $K_y=K$, then any unit $U\in \hat{U}_n^{sym}$ is a Sato operator for $F$. 
Indeed, by lemma \ref{L:5.11} and by properties of $\Ord$, we have $\Ord (\underline{\partial}^{\underline{k}}\cdot U )= |\underline{k}|$. Moreover, $\underline{\partial}^{\underline{k}}\diamond U \in F$ and all such elements are linearly independent as $U$ is regular (see lemma \ref{L:regularOp}). Therefore, $F_k\diamond U=F_k$ for any $k\in \dn$ by dimension reasons and $F=F\diamond U$.
\end{ex}

\begin{lem}
\label{L:SatoOperators}
Let $\mu \in \dz$, $W \in \mathsf{Gr}_\mu(V_n)$ and $T$ be a Sato operator of $W$. Then the following results are true.
\begin{enumerate}
\item 
$U \cdot T$ is a Sato operator of $W$ for any invertible operator $U \in \hat{U}_n^{sym}$.
\item 
The linear map $F \xrightarrow{-\diamond T} W$ is a bijection.
\end{enumerate}
\end{lem}

\begin{proof}
1. If $U \in \hat{U}_n^{sym}$, then $U \cdot T \ne 0$ and $\Ord (U \cdot T) = \Ord (U) + \Ord (T) = \mu$. Next,
$F \diamond (U \circ T) = (F \diamond U) \diamond T = F \diamond T = W$. Finally, $\sigma(U \circ T) = \sigma(U) \circ \sigma(T)$ as $\sigma(U)$ is invertible and regular.
Hence, the linear map $F \xrightarrow{- \diamond \sigma(U\circ T)} V_n$ is injective, implying that the operator $U \circ T$ is regular. Hence, $U \cdot T$ is indeed
a Sato operator of $W$.

\smallskip
\noindent
2. Since $T$ is regular, all elements from $F_m\diamond \sigma (T)$, $m\in \dn_0$, are linearly independent.  In particular, 
$\Ord (\underline{\partial}^{\underline{k}}\diamond \sigma (T))=|\underline{k}|+\mu$. Therefore, all elements from $F_m\diamond T$ are also linearly independent, hence  
for any $m \in \dn_0$, the  linear map $F_m \xrightarrow{-\diamond T} W_{m+\mu}$
is an isomorphism by the dimension reasons. This implies the second statement.
\end{proof}

The following theorem is an improvement of theorems \cite[Th.3.1]{Zheglov2013}, 
\cite[T. 5.16]{BurbanZheglov2017}.

\begin{thm}
\label{T:SatoAction}
Let $\mu \in \dz$ and $W \in \mathsf{Gr}_\mu(V_n)$. Then the following statements are true.
\begin{enumerate}
\item The vector space $W$ possesses  a Sato operator $S$.

If $T$ is another Sato operator for $W$ then there exists a uniquely determined invertible operator $U \in \hat{U}_n^{sym}$ such that $S = U\circ T$.
\item
If $W\subset F$, then $S\in \hat{D}_n^{sym}$.
\item
If $\mu =0$ and $\Sup (W)=F$, then there exists a unique Sato operator $S=1+S_-$, where   $S_-\in \hat{D}_n^n[[\partial_n^{-1}]]\partial_n^{-1}$, and therefore $|\ord_{\Gamma}(S)|=\Ord (S)=0$. In particular, $S$ is invertible in $\hat{E}_n$, $S^{-1}=1+\tilde{S}_-$, where $\tilde{S}_-=-S_-+S_-^2-\ldots$. 
\item
In particular, if $\mu =0$, $\Sup (W)=F$ and $S$ is a Sato operator, then there exists a unique factorisation $S=U\circ S_0$, where $S_0$ is the operator from previous item and $U\in \hat{U}_n^{sym}$.
\item
The following assertion holds:
$$
\hat{D}_n^{sym}=\{ P\in \hat{\Pi}_n|  F\diamond P\subseteq F\}.
$$
\end{enumerate}

\end{thm}

\begin{proof}
{\bf 1.} Our construction of a Sato operator $S$ is algorithmic and depends on the following choice. Namely, for any $\underline{k} \in \dn_0^n$, we choose
$w_{\underline{k}} \in W_{\mu + |\underline{k}|}$ such that for any $m \in \dn_0$, the set
$\bigl\{\bar{w}_{\underline{k}} \, \big| \, \underline{k} \in \dn_0^n: |\underline{k}| = m \bigr\}$, where $\bar{w}_{\underline{k}}$ means the residue class in the space $W_{m+\mu}/W_{m+\mu-1}$,  forms a basis of the vector space $W_{m+\mu}/W_{m+\mu-1}$  (at this place, we essentially use the assumption on the Hilbert function of $W$). Then the following statements are true:
\begin{itemize}
\item $\Ord \bigl(w_{\underline{k}}\bigr) = \mu + |\underline{k}|$ for any $\underline{k} \in \dn_0^n$;
\item the  set $\bigl\{{w}_{\underline{k}} \, \big| \, \underline{k} \in \dn_0^n: |\underline{k}| \le  m \bigr\}$ is  a basis of the vector space $W_{m+\mu}$.
\end{itemize}
We need to construct the operator $S$ in the form $S :=
 \sum\limits_{\underline{k} \ge \underline{0}} \dfrac{\underline{x}^{\underline{k}}}{\underline{k}!} S_{(\underline{k})} \in \hat{\Pi}_n
$
such that $\Ord (S_{(\underline{k})})\le \mu + |\underline{k}| $ for all $\underline{k} \in \dn_0^n$. 

Put $S_{(\underline{0})}= {w}_{\underline{0}}$. Now we can find all slices $S_{(\underline{k})}$ recursively:
if we know all slices with $|\underline{k}|\le m-1$, then the slices with $|\underline{k}|=m$ can be found from the linear equations:
\begin{equation}
\label{E:lin_equations}
w_{\underline{k}}=\underline{\partial}^{\underline{k}}\diamond S= S_{(\underline{k})} + \mbox{known terms},
\end{equation}
where $|\underline{k}|=m$ and known terms include linear combinations of terms of the form $S_{(\underline{l})} \cdot \underline{\partial}^{\underline{t}}$ with $|\underline{l}|< m$ and 
$|\underline{l}|+ |\underline{t}| \le \mu + m$. Obviously, this system has a unique solution, which determines the slice $S_{(\underline{k})}$. By construction of $S$ we have:
\begin{itemize}
\item 
 $\Ord (S) = \mu$ (since $\Ord (S_{(\underline{0})})=\mu$ and $\Ord S_{(\underline{k})}\le \mu +|\underline{k}|$);
 \item $\underline{\partial}^{\underline{k}} \diamond S = w_{\underline{k}}$ for all $\underline{k} \in \dn_0^n$, hence $W = F \diamond S$.
 \item $\underline{\partial}^{\underline{k}} \diamond \sigma(S) = \sigma\bigl(w_{\underline{k}})$ for all $\underline{k} \in \dn_0^n$. Since  $\bigl\{\bar{w}_{\underline{k}} \, \big| \, \underline{k} \in \dn_0^n: |\underline{k}| = m \bigr\}$ form a basis of the vector space $W_{m+\mu}/W_{m+\mu-1}$, according to Lemma \ref{L:regularOp}, the operator $S$ is regular.
\end{itemize}
Summing up, $S$ is a Sato operator of the vector space $W$.

Let $S$ and $T$ be two Sato operators of $W$. As both operators are regular, for any $\underline{k} \in \dn_0^n$ there exist an element $w_{\underline{k}}'\in F$ of order $\Ord (w_{\underline{k}}')=|\underline{k}|$ such that 
$$
\underline{\partial}^{\underline{k}} \diamond S = w_{\underline{k}}' \diamond T
$$
Since the elements  $\underline{\partial}^{\underline{k}} \diamond \sigma (S)$, $|\underline{k}|=m$ are linearly independent, so are the elements $w_{\underline{k}}' \diamond \sigma (T)$, hence $w_{\underline{k}}'$. Then by dimension reasons, we have
$$
W'=\langle w_{\underline{k}}', \underline{k}\in \dn_0^n\rangle =F.
$$
Let $U\in \hat{\Pi}_n$ be the operator constructed as above for the space $W'$, i.e. $\underline{\partial}^{\underline{k}} \diamond U =w_{\underline{k}}'$. Then $\Ord (U)=0$, moreover, $U\in \hat{D}_n^{sym}$ by construction above. 
By the construction of $U$, we have for all $\underline{i}\in \dn_0^n$:
$$
\underline{\partial}^{\underline{i}} \diamond S= \bigl(\underline{\partial}^{\underline{i}}\diamond U\bigr)  \diamond T = \underline{\partial}^{\underline{i}} \diamond (U \circ T)
$$
so $S = U\circ T$ by Lemma \ref{L:statement}. In a similar way, we can find $V \in \hat{D}_{n,-}^{sym}$ such that $T = V \circ S$. Therefore, we have: $(1- U\cdot V)\circ S = 0$. Since
$S$ is regular, Lemma \ref{L:regularOp} implies that $U\cdot V = 1$. In a similar way, $V\cdot U = 1$, hence $U, V \in \hat{U}_{n}^{sym}$ as claimed.

The uniqueness of the unit $U$ also follows from Lemma \ref{L:regularOp}.

{\bf 2.} This item follows immediately from the construction of item 1.

{\bf 3.} We can construct the Sato operator practically by the same arguments as in item 1, with minimal appropriate modifications. Namely, we will look for the operator $S$ in the form $S :=
 \sum\limits_{\underline{k} \ge \underline{0}} \dfrac{\underline{x}^{\underline{k}}}{\underline{k}!} S_{(\underline{k})} \in \hat{\Pi}_n
$
such that 
\begin{itemize}
\item
$\Ord (S_{(\underline{k})})\le |\underline{k}| $ for all $\underline{k} \in \dn_0^n$, 
\item
$S_{(\underline{k})}\in K_y[[\partial_1, \ldots , \partial_{n-1}]][[\partial_n^{-1}]]\cdot \partial_n^{-1}\cap \Pi_n$ for all $\underline{k} \neq \underline{0}$ 
\item
$S_{(\underline{0})}\in 1+K_y[[\partial_1, \ldots , \partial_{n-1}]][[\partial_n^{-1}]]\cdot \partial_n^{-1}\cap \Pi_n$. 
\end{itemize}

First let's prove that we can choose a basis $w_{\underline{k}}$, $\underline{k} \in \dn_0^n$ in the space $W$ such that 
\begin{enumerate}
\item\label{E:a1}
$w_{\underline{k}}=\underline{\partial}^{\underline{k}}+ w_{\underline{k}}^-$, where $w_{\underline{k}}^-\in K_y[[\partial_1, \ldots , \partial_{n-1}]][[\partial_n^{-1}]]\cdot \partial_n^{-1}\cap \Pi_n$,
\item\label{E:a2}
$\Ord (w_{\underline{k}})= |\underline{k}|$.
\end{enumerate}
We'll construct such a basis recursively in the subspaces $W_m$, $m\ge 0$. 
The proof is by induction on $m=\Ord (w_{\underline{k}})$. 

Let $m=0$. 
Let $w_{\underline{0}}$ be a generator of the space $W_0$. Since  $W,\Sup (W)\in \mathsf{Gr}_0 (V_n)$, $\dim_K (\Sup (W_0))=H_W(0)=1$ and $\dim_K (\Sup (W_{-1}))=H_W(-1)=0$, we have $0=\Ord (w_{\underline{0}})=\Ord (LT(w_{\underline{0}}))=|\ord_{\Gamma}(w_{\underline{0}})|$. In particular, we can choose a monic generator $w_{\underline{0}}$. Then  $\LT (w_{\underline{0}})=1$ (as it belongs to $F$ and $w_{\underline{0}}$ is monic). Therefore, $w_{\underline{0}}^- :=w_{\underline{0}}-1 \in K_y[[\partial_1, \ldots , \partial_{n-1}]][[\partial_n^{-1}]]\cdot \partial_n^{-1}\cap \Pi_n$.

Suppose $m$ is arbitrary. By induction, there exists a basis $w_{\underline{k}}$, $\underline{k} \in \dn_0^n$, $|\underline{k}|\le m-1$ in the space $W_{m-1}$, satisfying properties \ref{E:a1} and \ref{E:a2}. Let's complete this basis to a basis $w_{\underline{k}}$, $|\underline{k}|\le m$. According to the condition on dimensions of $W_m$  we can assume all its elements are monic; besides we can assume that $\LT (w_{\underline{k}})\notin F_{m-1}$ for all $\underline{k} \in \dn_0^n$ with $|\underline{k}|=m$. Besides, it holds $m=\Ord (w_{\underline{k}})= |\ord_{\Gamma}(w_{\underline{k}})|$ for all elements $w_{\underline{k}}$ with $|\underline{k}|=m$.  So, by renumbering elements of the basis if needed,  we can assume that $\LT (w_{\underline{k}})=\underline{\partial}^{\underline{k}}$. But then, obviously, there exists a basis in $W_m$ satisfying conditions \ref{E:a1} and \ref{E:a2}.

Now put $S_{(\underline{0})}= {w}_{\underline{0}}$. All other slices $S_{(\underline{k})}$ we can find recursively:
if we know all slices with $|\underline{k}|\le m-1$, then the slices with $|\underline{k}|=m$ can be found from the linear equations:
\begin{equation}
\label{E:lin_equations1}
\underline{\partial}^{\underline{k}}\diamond S\in W_m.
\end{equation}
As in item 1 we have
$$
\underline{\partial}^{\underline{k}}\diamond S= S_{(\underline{k})} + \mbox{known terms},
$$
where $|\underline{k}|=m$ and known terms include linear combinations of terms of the form $S_{(\underline{l})} \cdot \underline{\partial}^{\underline{t}}$ with $|\underline{l}|< m$ and 
$|\underline{l}|+ |\underline{t}| \le m$. By the induction hypothesis, $\Ord (\mbox{known terms})\le m$, and therefore there exists an element $w\in W_m$ such that $\underline{\partial}^{\underline{k}}\diamond S -w\in K_y[[\partial_1, \ldots , \partial_{n-1}]][[\partial_n^{-1}]]\cdot \partial_n^{-1}\cap \Pi_n$ (because all monomials of non-negative order with non-negative power of all derivatives belong to the space $F_m$, generated by the least terms of the basis elements $w_{\underline{k}}$). Now put
$$
S_{(\underline{k})}:= \underline{\partial}^{\underline{k}}\diamond S -w.
$$
Clearly, all conditions on $S_{(\underline{k})}$ are satisfied.

Obviously, this system has a unique solution, which determines the slice $S_{(\underline{k})}$.

{\bf 4.} This follows from items 1 and 3.

{\bf 5.} Let $P=
 \sum\limits_{\underline{k} \ge \underline{0}} \dfrac{\underline{x}^{\underline{k}}}{\underline{k}!} P_{(\underline{k})} 
$
be the slice decomposition. Let $m$ be the first number such that there exists $\underline{k}\in \dn_0^n$ with  $P_{(\underline{k})}\notin F$,  $m=|\underline{k}|$. But then 
$$
\underline{\partial}^{\underline{k}}\diamond P= P_{(\underline{k})} + P',
$$
where  $P'\in F$. Therefore, $P_{(\underline{k})}$ must belong to $F$ (as 
$\underline{\partial}^{\underline{k}}\diamond P\in F$), a contradiction.
So, all slices must belong to $F$ and therefore $P\in \hat{D}_n^{sym}$. 
\end{proof}

\begin{cor}
\label{C:units}
An operator $P\in \hat{D}_{n,-}^{sym}$ is a unit if and only if $\Ord (P)=0$ and $\bar{P}$ is regular, where $\bar{P}=P|_{\underline{y}=\underline{0}}$. 
\end{cor}

\begin{proof}
If $P$ is a unit, then $\Ord (P)=0$ and $P$ is regular by Lemma \ref{L:5.11}, item 2. Obviously, $\bar{P}$ is regular by the same reason.

Conversely, if $\bar{P}$ is regular and $\Ord (P)=0$, then $F\diamond \bar{P}=F$ and therefore $F\diamond P=F$, cf. example \ref{Ex:Sato_F}. Then by theorem \ref{T:SatoAction}, item 1 $P=U\circ 1$, where $U$ is a unit. 
\end{proof}

\begin{ex}
\label{R:Birkhoff}
As an application of our general Sato theory we can give a short proof (though not effective) of the generalized Birkhoff decomposition theorem from \cite{Mulase_inv}, cf. \cite[Th. 11.19]{Zheglov_book}.

Recall the notation from \cite{Mulase_inv}. Consider the completion of the ring $K[x, t_1, t_2, \ldots ]$ with respect to the $K$-valuation $\nu$ defined as $\nu (t_n):=n$, $\nu (x)=1$. Denote this completion by $K[[x,t]]$.

Set $\cd := K[[x,t]] [\partial ]$, $\cee :=  K[[x,t]]((\partial^{-1}))$. Denote by $\Upsilon$ the subgroup of zeroth order monic invertible operators from $\cee$. Consider a completion of the ring $\cee$ with respect to the valuation $v$ defined as $v(t_i):=i$, $v(x)=0$: $\hat{\cee}:=\{ \sum_{i=-\infty}^{\infty} a_i\partial^i| \quad v(a_i)\stackrel{i\rightarrow \infty}{\longrightarrow} \infty\}$. Denote by $\hat{\cd}$ the corresponding completion of the ring $\cd$. Denote by $Pr:K[[x,t]]\rightarrow K[[x,t]]/(x,t_1,t_2,\ldots )\simeq K$ the natural projection. 

Define the subgroups 
$$
\hat{\cd}^*=\{ P\in \hat{\cd}| \quad \mbox{$Pr(P)=1$ and $\exists$ $P^{-1}\in \hat{\cd}$}\},
$$ 
$$
\hat{\cee}^*=\{ P\in \hat{\cd}| \quad \mbox{$Pr(P)\in \Upsilon$ and $\exists$ $P^{-1}\in \hat{\cee}$}\}. 
$$ 
\begin{thm}[generalized Birkhoff decomposition by M. Mulase]
\label{T:Birkhoff}
For any $\Psi\in \hat{\cee}^*$ there exists a unique operators $S\in \Upsilon$, $Y\in\hat{\cd}^*$ such that $\Psi =YS$. 
\end{thm}

\begin{proof} Consider the space $W:=F\diamond \Psi$.  Then the proof of the Sato theorem \ref{T:SatoAction}, item 3 can be applied to get a uniquely defined operator $S\in \Upsilon$ such that $W=F\diamond S$. Therefore,  $Y:=\Psi S^{-1}\in \hat{\cd}$, and, obviously, $Y\in \hat{\cd}^*$. 

The decomposition $\Psi =YS$ is unique: if $Y_1S_1=Y_2S_2$, then $Y_2^{-1}Y_1=S_2S_1^{-1}$, but $\hat{\cd}^*\cap \Upsilon =\{1\}$, therefore $Y_1=Y_2$ and $S_1=S_2$. 
\end{proof}

\end{ex}

\subsection{Schur pairs, two constructions}
\label{S:Schur_pairs}

The Schur theory for $\hat{D}_n$ leads to a notion of Schur pair associated to a commutative ring of operators. In this section we describe this notion for quasi-elliptic subrings from $\hat{D}_n^{sym}$.  A review of this notion for  subrings from $D_1$ see in \cite{Mul} and in \cite[\S 5]{Quandt} for generalisations in the case of relative situation. The Sato theory is used to establish one-to-one correspondence between Schur pairs and quasi-elliptic rings.

\begin{defin}
\label{D:sch}
We say that a pair of subspaces $(A,W)$, where $A,W \subset  V_n$ and $A$ is a $K$-algebra with unity such that $W\cdot A\subset W$, is a {\it 1-Schur pair in $V_n$}  if $A$ and $W$ are 1-spaces (see definition \ref{D:1space}) and $W \in \mathsf{Gr}_\mu(V_n)$. 

We say that 1-Schur pair is a quasi elliptic 1-Schur pair if $A$ is a quasi-elliptic ring. 

We say that two 1-Schur pairs $(A,W)$ and $(A',W')$ are equivalent if $A'=T^{-1}AT$ and $W'=W\diamond T:= \pi (W\circ T)$, where $T$ is a 1-admissible operator, cf. definition \ref{D:admissible}, and we consider $W$ as a submodule of the bimodule $\hat{\Pi}_n$, cf. lemma \ref{L:module_structure} and definition \ref{D:defin2}.
\end{defin}

The Schur pairs appear naturally from quasi-elliptic subrings from $\hat{D}_n$ due to the following construction.

{\bf Construction 1.} Given a quasi-elliptic ring $B\subset \hat{D}_n$ we choose a set of monic formally quasi-elliptic elements $P_1, \ldots , P_n$. Then by lemma \ref{L:lemma2.9} there exist monic formally quasi-elliptic operators $L_1,\ldots , L_n$ with $\Ord (L_i)=1$ for all $i$. Then by lemma \ref{L:lemma7} and corollary \ref{C:normalised_roots} there exists an operator $S_0$  such that $S_0^{-1}L_iS_0$ and $S_0^{-1}P_iS_0$  are normalized (of the same order), and by theorem \ref{T:lemma8} there exists a monic invertible operator $S$ satisfying the condition $A_1$ with $\Ord (S)=|\ord_{\Gamma}(S)|=0$ such that $SS_0^{-1}L_iS_0S^{-1}=\partial_i\in V_n$.  

The spaces $A$ and $W$ are obtained as follows: 
$$
A:=SS_0^{-1}BS_0S^{-1}\subset V_n\cap \Pi_n \quad W=F\diamond (S_0S^{-1})=F\diamond S^{-1}\subset V_n\cap \Pi_n
$$ 
(the last equality follows from the special form of $S_0$, so that $F\diamond S_0=F$). Note that $\Sup (W) =F$, $W\in \mathsf{Gr}_0(V_n)$, and $\dim_K (\Sup (W_k))=H_F(k)$, because $LT(\underline{\partial}^{\underline{j}}\diamond S^{-1})=\underline{\partial}^{\underline{j}}$ for any $\underline{j}\in \dn_0^n$.
 
By the same facts from the Schur theory it follows that $A$ does not depend on the choice of $S_0,S$ (but  $W$ depends: it changes by a multiplication with a 1-admissible operator with constant coefficients), but may depend on the choice of $P_1, \ldots , P_n$ . Choosing other elements $P_1',\ldots , P_n'$ we obtain another subring $A'\subset V_n$, which differs from $A$ by conjugation with a 1-admissible operator. This explains the definition of equivalence relation between 1-Schur pairs. Note that in both cases 1-Schur pairs $(A,W)$ and $(A',W')$ are equivalent.

{\bf Construction 2.} Vice versa, given a  1-Schur pair $(A, W)$ with $\Sup (W) =F$, we can apply theorem \ref{T:SatoAction}, item 3 and find an invertible  Sato operator $S$. Then $B:=SAS^{-1}\in \hat{E}_n\cap \Pi_n$ consists of operators $P$ such that $F\diamond P\subseteq F$, hence $B\subset \hat{D}_n$ by theorem \ref{T:SatoAction}, item 5. If $(A, W)$ is a quasi-elliptic 1-Schur pair, then 
since $\Ord (S)=|\ord_{\Gamma}(S)|=\Ord (S^{-1})=|\ord_{\Gamma}(S^{-1})|=0$ and formal quasi-elliptic operators from $A$ satisfy condition $A_1$ (by definition), it follows from lemma \ref{L:lemma2.7} and corollary \ref{C:corol2.1} that $B$ is also quasi-elliptic. 

For a generic Schur pair, the following theorem is an improvement of \cite[Th. 5.18]{BurbanZheglov2017}

\begin{thm}\label{T:SatoSchur}
Let $(A, W)$ be a 1-Schur pair and $S$ be a Sato operator of $W$. Then the following statements are true.
\begin{enumerate}
\item For any element $f \in A$, there exists a uniquely determined operator $\mathtt{L}_S(f) \in \hat{D}_n^{sym}$ such that
$
S \circ f = \mathtt{L}_S(f) \circ S.
$
Moreover, $\Ord \bigl(\mathtt{L}_S(f)\bigr) = \Ord (f)$ for any $f\in A$.
\item Next, for any $f_1, f_2 \in A$ and $\lambda_1, \lambda_2 \in K$ we have:
\begin{equation*}
\mathtt{L}_S(f_1 \cdot f_2) = \mathtt{L}_S(f_1) \cdot \mathtt{L}_S(f_2) \quad \mbox{\rm and} \quad \mathtt{L}_S(\lambda_1 f_1 + \lambda_2 f_2) = \lambda_1 \mathtt{L}_S(f_1) + \lambda_2 \mathtt{L}_S(f_2)
\end{equation*}
In other words, the map $A \rightarrow \hat{D}_n^{sym}, f \mapsto \mathtt{L}_S(f)$ is a homomorphism of $K$--algebras, which is moreover
injective.
\end{enumerate}
\end{thm}

\begin{proof}
For any multiindex $\underline{k}\in\dn_0^n$ we define $w_{\underline{k}}:=\underline{\partial}^{\underline{k}}\diamond S\cdot f\in W$. Since $W=F\diamond S$, there exist (uniquely defined) elements $w_{\underline{k}}'\in F$ such that $w_{\underline{k}} =w_{\underline{k}}'\diamond S$. 

Put $W'=\langle w_{\underline{k}}', \underline{k}\in\dn_0^n\rangle$. Then by the proof of theorem \ref{T:SatoAction}, item 1,2 there exists an operator $\mathtt{L}_S(f)\in \hat{D}_n^{sym}$ such that $w_{\underline{k}}'=\underline{\partial}^{\underline{k}}\diamond \mathtt{L}_S(f)$. Then by lemma \ref{L:statement} we have 
$
S \circ f = \mathtt{L}_S(f) \circ S.
$
The uniqueness of $\mathtt{L}_S(f)$ follows from the regularity of $S$ and Lemma  \ref{L:regularOp}.

\smallskip
\noindent
(2) Let $f_1, f_2 \in A$. By construction, we have:
$$
S \circ (f_1 \cdot f_2) = \bigl(\mathtt{L}_S(f_1) \circ S\bigr) \circ f_2 = \bigl(\mathtt{L}_S(f_1) \cdot \mathtt{L}_S(f_2)\bigr)\circ S.
$$
Since the operator $L_S(f)$ is uniquely determined by $f \in A$, we get:  $\mathtt{L}_S(f_1) \cdot \mathtt{L}_S(f_2) = \mathtt{L}_S(f_1 \cdot f_2)$. The proof of the second statement is analogous, hence
$A \rightarrow \hat{D}_n^{sym} \, f\mapsto L_S(f)$ is indeed a homomorphism of $K$--algebras. For $f \ne 0$ we have: $\sigma(S \circ f) = \sigma(S) \circ \sigma(f) \ne 0$, hence $\mathtt{L}_S(f) \ne 0$, too.

\end{proof}

There is an analogue of proposition \ref{P:quasi-ell-properties}, item 6 for quasi-elliptic 1-Schur pairs:

\begin{lem}
\label{L:quasi-ell-properties}
Let $(A,W)$ be a  quasi-elliptic 1-Schur pair. Then $\trdeg_{K'}(\Quot (A))=n$, where $K'=\Quot (K_y)$, the field $\Quot (A)$ is finitely generated over $K'$ and the localisation $\Quot (A)\cdot W$ is a finitely generated $\Quot (A)$-module.
\end{lem}

\begin{proof}
The proof is literally the same as the proof of proposition \ref{P:quasi-ell-properties}, item 6.
\end{proof}

The most interesting examples of quasi-elliptic rings in $\hat{D}_n^{sym}$ and of quasi-elliptic 1-Schur pairs are those of rank $1$:

\begin{defin}
\label{D:an-alg-rank}
Let $B\subset \hat{D}_n$ be a commutative subring. Then we define the {\it analytical rank} as 
$$
An.rank(B):=\rk (F\cdot  \Quot (B))=\dim_{\Quot (B)}(F\cdot  \Quot (B)).
$$

Analogously, if $(A,W)$ is a Schur pair, we define its analytical rank as 
$$
An.rank((A,W)):={\rk}_A (W)=\dim_{\Quot (A)}(W\cdot  \Quot (A)).
$$

For brevity, we'll use the word {\it rank} instead of analytical rank in this paper, cf. \cite[Def. 7.1]{Zheglov_belovezha}. 
\end{defin}

\begin{defin}
\label{D:normal_Schur_pair}
A quasi-elliptic 1-Schur pair is called {\it normalized} if $1\in W$ and $A$ contains formal quasi-elliptic operators which are {\it monomials}. 
\end{defin}

\begin{rem}
\label{R:rank_one_pairs}
1) Note that construction 1 produces almost normalised Schur pairs in the following sense: if we multiply $W$ with an appropriate operator $S$ with constant coefficients, then the pair $(S^{-1}AS,W\diamond S)=(A,WS)$ is an equivalent normalized pair. Besides, by the same arguments from the construction, using lemma \ref{L:lemma7}, corollary \ref{C:normalised_roots} and theorem \ref{T:lemma8}, we can find for any quasi-elliptic 1-Schur pair an equivalent normalized quasi-elliptic 1-Schur pair. 

2) There is one beautiful property of a normalized quasi-elliptic 1-Schur pair of rank one: since $1\in W$ and $A\cdot W\subset W$, we have $W\in \Quot (A)$. We will use this property  to prove some non-trivial results on the behaviour of a 1-quasi-elliptic ring under a linear change of variables.
\end{rem}

\subsection{Admissible linear change of variables}
\label{S:admissible_change}

Before we'll formulate the next lemma, we need to introduce new notation. Consider a set of elements $Y\subset V_n$ of fixed $\Ord$-order $p$. Note that the $\Gamma$-order induces on $Y$ an order function with values in $\dz$ and with the maximal value $(0,\ldots ,0, p)$ (since the $\Gamma$-order takes values in $\dz_+\oplus \ldots \oplus \dz_+\oplus \dz$). Namely, $(0, \ldots , 0,1, p-1)<(0,\ldots ,0, p)$ and there are no orders $\underline{j}=(j_1, \ldots ,j_n)$ with $|\underline{j}|=p$ such that $(0, \ldots , 0, 1, p-1)<\underline{j}<(0,\ldots ,0, p)$. We'll use the notation $(0, \ldots , 0,1, p-1)= (0,\ldots ,0, p)-1$. Analogously, $(0,\ldots ,0, p)-i$ will denote the $\Gamma$-order such that there are exactly $i-1$ orders greater than this, e.g.  $(0,\ldots ,0, p)-2=(0, \ldots , 1,0, p-1)$,  $(0,\ldots ,0, p)-3=(0, \ldots , 1,0,0, p-1)$, and so on.

\begin{lem}
\label{L:linear_changes}
%Assume $K$ is algebraically closed and uncountable. 

Let $(A,W)$ be a normalized quasi-elliptic 1-Schur pair. Assume that $v_1, \ldots ,v_k\in A$ are elements with $\Ord (v_i)=p>0$ for all $i=1, \ldots , k\le n$, $\sigma (v_i)=\sigma (\bar{v}_i)$,  for all $i=1, \ldots , k$ and $\sigma(\bar{v}_1), \ldots , \sigma(\bar{v}_k)$ are algebraically  independent (over $K$). Assume additionally that $v_1\partial_n^{-\Ord (v_k)}, \ldots , v_k\partial_n^{-\Ord (v_k)} \in K_y[[\partial_1', \ldots , \partial_{n-1}', \partial_n^{-1}]]$, where $\partial_i'=\partial_i\partial_n^{-1}$. \footnote{If $K_y=K$, this condition holds automatically.} Denote by $\hat{K}$  some completion of $K$ (cf. \cite[Ch. VI]{Bu}).

Then there exists an invertible linear change of variables 
$$\varphi :\partial_i \mapsto \sum_{j=1}^{n-1}c_{ij}\partial_j +\lambda_i \partial_n, \quad i=1, \ldots , n-1, \quad c_{ij},\lambda_i\in K, \quad \partial_n \mapsto \partial_n
$$ 
such that the vector space $\hat{K}_y\langle \varphi (v_1), \ldots , \varphi (v_k)\rangle$ (over $\hat{K}_y$) has a basis $w_1 , \ldots , w_k\in \hat{K}_y[[\partial_1', \ldots , \partial_{n-1}', \partial_n^{-1}]]$ with $\Ord (w_i)=p$, $\sigma (w_i)=\sigma (\bar{w}_i)$, $i=1,\ldots ,k$, $\ord_{\Gamma}(\sigma (w_1))=\ord_{\Gamma}(w_1)=(0, \ldots ,0, p)$, $\ord_{\Gamma}(\sigma (w_i))=\ord_{\Gamma}(w_i)=\ord_{\Gamma}(\sigma (w_1))- (i-1)$, $k\ge i>1$, and $\sigma (w_1), \ldots ,\sigma (w_k)$ are monic. In particular, $w_1, \ldots , w_k$ are formally quasi-elliptic. 

Moreover, there is an open subset $0\in U\subset \hat{K}^{n-1}$ such that this is true for any $(\lambda_1, \ldots ,\lambda_{n-1})\in U$. 
\end{lem}

\begin{proof}
First, note that we can reduce the proof to the case $p=0$ by multiplying all elements on $\partial_n^{-p}$. 

Let $P_1, \ldots , P_n$ be monic monomial formally quasi-elliptic elements in $A$. Consider the  ring 
\begin{multline*}
A'=K_y[P_1\partial_n^{-\Ord (P_1)}, \ldots , P_{n-1}\partial_n^{-\Ord (P_n)}, \partial_n^{-1}, v_1\partial_n^{-\Ord (v_1)}, \ldots , v_k\partial_n^{-\Ord (v_k)}]= \\
K_y[\partial_1\partial_n^{-1}, \ldots , \partial_{n-1}\partial_n^{-1}, \partial_n^{-1}, v_1\partial_n^{-\Ord (v_1)}, \ldots , v_k\partial_n^{-\Ord (v_k)}].
\end{multline*} 
It is easy to see that $A'\subset K_y[[\partial_1', \ldots , \partial_{n-1}', \partial_n^{-1}]]$, where $\partial_i'=\partial_i\partial_n^{-1}$, the ideal $$I=(\partial_1', \ldots , \partial_{n-1}', \partial_n^{-1}, y_1, \ldots ,y_n)\cdot K_y[[\partial_1', \ldots , \partial_{n-1}', \partial_n^{-1}]]\cap A'$$ 
is maximal in $A'$ and that $\hat{A'}_{I}\simeq K[[y_1, \ldots ,y_n, \partial_1', \ldots , \partial_{n-1}', \partial_n^{-1}]]$, i.e. $A_I'$ is a regular local ring by \cite[Prop. 11.24]{AM}. By lemma \ref{L:quasi-ell-properties}, the elements $v_1\partial_n^{-\Ord (v_1)}, \ldots , v_k\partial_n^{-\Ord (v_k)}$ are algebraic over the field $\Quot (K_y)(\partial_1', \ldots , \partial_{n-1}', \partial_n^{-1})$, in particular, the images of these elements in $\hat{A'}_{I}\hat{\otimes}_K\hat{K}\simeq \hat{K} [[y_1, \ldots ,y_n, \partial_1', \ldots , \partial_{n-1}', \partial_n^{-1}]]$ are  analytic functions (determined by their Taylor series in some open poly-cylinder $P_0(r)$ with the center at zero). 

Put $v_i':=v_i\partial_n^{-\Ord (v_i)}$, $i=1,\ldots ,k$. Our claim is now reduced to the claim about linear changes $\partial_i' \mapsto \partial_i'+\lambda_i$ for analytic functions $v_i'$ (note that $\ord_{\Gamma}$ is still defined on elements of $\hat{K} [[y_1, \ldots ,y_n, \partial_1', \ldots , \partial_{n-1}', \partial_n^{-1}]]$). 

The proof is by induction on $k$. Let $k=1$. Let $P_0(r)$ denotes the open poly-cylinder in $\widehat{K}^{n+m}$ with the center at zero. Then for some small $r$ we have for any $\lambda =(\lambda_1, \ldots ,\lambda_{n-1})\in K^{n-1}$, $|\lambda |<r$ and $(y_1, \ldots ,y_n, \partial_1'=t_1, \ldots ,\partial_{n-1}'=t_{n-1}, \partial_n^{-1})\in P_0(r-|\lambda|)$ (see e.g. \cite[Ch. 2]{Serr})
$$
v_1'(t_1+\lambda_1, \ldots , t_{n-1}+\lambda_{n-1})=\frac{1}{\underline{j}!}\sum_{\underline{j}=(j_1, \ldots , j_{n-1})\ge (\underline{0})} \underline{\partial}^{\underline{j}}(v_1')(\lambda_1, \ldots , \lambda_{n-1})t^{\underline{j}}.
$$
For any fixed $\lambda_1, \ldots ,\lambda_{n-1}\in K$, $|\lambda |<r$ we'll consider the function $v_1'(t_1+\lambda_1, \ldots , t_{n-1}+\lambda_{n-1})$ as an element of the ring $\widehat{K}[[y_1, \ldots , y_n, \partial_n^{-1}]] [[t_1, \ldots , t_{n-1}]]$, and we will use the same notation for the order functions $\Ord$, $\ord_{\Gamma}$ defined on   elements of this ring (setting $\Ord (t_i)=0$, $\ord_{\Gamma}(t_i)=(\underline{0})-(n-i)$).

 Consider the function $\tilde{v}_1:= \bar{v}_1'|_{\partial_n^{-1}:=0}$. Since $\Ord (v_1')=0$ and $\sigma (v_1')=\sigma (\bar{v}_1')$, we have  $\tilde{v}_1\neq 0$. 
Clearly, there is an open subset $0\in U\subset \hat{K}^{n-1}$ of points $(\lambda_1, \ldots ,\lambda_{n-1})\in K^{n-1}$, $|\lambda |<r$ such that $\tilde{v}_1(\lambda_1, \ldots ,\lambda_{n-1})\neq 0$. Then 
$$
\ord_{\Gamma}(v_1'(t_1+\lambda_1, \ldots , t_{n-1}+\lambda_{n-1}))=
\ord_{\Gamma}(\bar{v}_1'(t_1+\lambda_1, \ldots , t_{n-1}+\lambda_{n-1}))=(\underline{0}).
$$
Without loss of generality we can assume  $\tilde{v}_1(\lambda_1, \ldots ,\lambda_{n-1})=1$. 
Note that the series $f(y):=v_1'|_{\partial_1'=\ldots =\partial_{n-1}'=\partial_n^{-1}=0}-1$ converges in the same poly-cylinder and therefore the series $w_1:=v_1'-f(y)v_1'$ converges too; besides, $\sigma (w_1)$ is monic and $\sigma (w_1)=\sigma (\bar{w}_1)$, $\ord_{\Gamma}(\sigma (w_1))=\ord_{\Gamma} (w_1)=(\underline{0})$.

For generic $k$ consider the  $\hat{K}_y$-vector space $H_k=\hat{K}_y \langle v_1'(t), \ldots , v_{k}'(t)\rangle $. By induction, there exists a linear change of variables $\varphi$ such that there is a basis $w_1(t), \ldots , w_{k-1}(t)$ in the space $\hat{K}_y\langle \varphi (v_1'(t)), \ldots , \varphi (v_{k-1}'(t))\rangle$ with the desired properties. Then we can complete this basis to a basis $w_1(t), \ldots , w_{k-1}(t), v_k(t)$ such that $\ord_{\Gamma}(\sigma (v_k(t)))\le (\underline{0})-(k-1)$ (and all other properties of $\sigma (v_k(t))$ remain valid, in particular it is monic). If $\ord_{\Gamma}(\sigma (v_k(t)))= (\underline{0})-(k-1)$, then clearly we can choose $v_k$ is such a way that $\ord_{\Gamma}(\sigma (v_k(t)))=\ord_{\Gamma}(v_k(t))$, and we are done, so we assume that $\ord_{\Gamma}(\sigma (v_k(t)))< (\underline{0})-(k-1)$.

If each entry of $\ord_{\Gamma}(\sigma (v_k(t)))$ is not more than one, then 
$$\sigma (v_k(t))= t_j+\mbox{terms of lower $\Gamma$-order}, $$
where $j<n-k$. Now apply the linear change $\varphi :\partial_j\mapsto \partial_j +\epsilon \partial_{n-k-1}$, $\partial_i \mapsto \partial_i$, $i\neq j$.  It is not difficult to see that $\varphi (w_i)$, $\varphi (v_k)$ are well defined for any $\epsilon$, belong to $K_y[[\partial_1', \ldots , \partial_{n-1}', \partial_n^{-1}]]$ and are analytic in the same poly-disc (clearly they are still algebraic; besides, this linear change is a result of conjugating with the operator $S:=\exp (\sum_{i=1}^{n-1}(\sum_{j=1}^{n-1}(c_{i,j}-\delta_{i,j})x_j)*\partial_i)) \in \hat{D}_n^n$ for appropriate constants $c_{i,j}$, cf. corollary \ref{C:linear_changes} below). Obviously, they satisfy all our conditions just for  $\epsilon =1$.

%apply any linear change as above such that $\varphi (v_k(t))$ contains a monomial of order $|\ord_{\Gamma}(v_k(t))|$ containing $t_{n-k-1}$. Without loss of generality we can assume that this monomial has coefficient in $K$ and that $\varphi (w_i)$, $i=1, \ldots , k-1$ form a basis with the same properties as the basis $w_i$. We'll denote these new functions again as $w_1, \ldots ,w_{k-1}, v_k$.

Assume there is an entry of $\ord_{\Gamma}(\sigma (v_k(t)))$ which is more than one.  Now  we represent the series $w_i(t+h)$, $v_k(t+h)$, where $h=(y_1=0,\ldots , y_n=0, \lambda_1, \ldots , \lambda_{n-1}, \partial_n^{-1}=0)\in P_0(r)$, $t=(y_1, \ldots, y_n, t_1, \ldots , t_{n-1}, \partial_n^{-1})\in P_0(r-|h|)$, in {\it powers of $t$} as vector-rows with entries 
\begin{multline*}
\underline{\partial}^{\underline{j}=(j_1, \ldots ,j_{n-1})}(v_i')(h):=\\
\sum_{(i_1, \ldots , i_{n+1})\ge (\underline{0})}\frac{1}{\underline{i}!}\partial_{y_1}^{i_1}\ldots \partial_{y_n}^{i_n}\partial_{\partial_n^{-1}}^{i_{n+1}}\frac{1}{\underline{j}!}\partial_{t_1}^{j_1}\ldots \partial_{t_{n-1}}^{j_{n-1}}(v_i')(h)y_1^{i_1}\ldots y_n^{i_n}\partial_n^{-i_{n+1}}\\
\in \widehat{K}[[y_1, \ldots , y_n, \partial_n^{-1}]] [[\lambda_1, \ldots , \lambda_{n-1}]],
\end{multline*}
each entry is labelled by the multi-index $(j_1, \ldots , j_{n-1})\ge (\underline{0})$, and these indices are ordered w.r.t. the  order mentioned before this lemma. Now we can reformulate the problem of finding the needed basis in $H_k=\hat{K}_y \langle w_1(t+h), \ldots , w_{k-1}(t+h), v_{k}(t+h)\rangle $: there is such a basis $w_i$ iff the $k\times k$-minor of the matrix $(\underline{\partial}^{\underline{j}}(v_i')(h))$, $\underline{j}< \underline{0}+k$, $i=1, \ldots ,k$ is invertible, i.e. its determinant is non-zero for some $h$.  

To show that such $h$ exists we need to show that the determinant of this matrix is not identically zero as a series in the ring $\widehat{K}[[y_1, \ldots , y_n, \partial_n^{-1}]] [[\lambda_1, \ldots , \lambda_{n-1}]]$ (after, may be, some linear change of coordinates). For convenience, we'll present each entry of this minor as a series of the form $f(\lambda_1, \ldots ,\lambda_{n-1})+(\cm^{\Ord f+1})$, where $f(t_1,\ldots ,t_{n-1})$ is a homogeneous polynomial with respect to the $\Ord$-order, and $ (\cm^{\Ord f+1})$ denotes the sum of elements from the ideal $\cm$ generated by $y_1, \ldots ,y_n, \partial_n^{-1}, \lambda_1, \ldots ,\lambda_{n-1}$. 
%Making a linear change of the form $\partial_i\mapsto \partial_i+c_i\partial_k$, $i=1,\ldots ,k-1$ for appropriate constants $c_i\in \hat{K}$ if necessary (in particular, if $k\le n-1$), 
Clearly, we can assume that $\underline{\partial}^{\underline{j}}(w_i)(h)=(\cm)$ for $i=1, \ldots , k-1$, $(\underline{0})-(k-2)\le \underline{j}< (\underline{0}) - (i-1)$. Making a linear change $\partial_i \mapsto \partial_i+\varepsilon_{i,k} \partial_k + \ldots + \varepsilon_{i,n} \partial_n$, $i=1, \ldots , k-1$ for appropriate constants $\varepsilon_{i,j} \in K$ if necessary we can also assume that $\underline{\partial}^{\underline{j}}(w_i)(h)=(\cm)$ for $i=k, \ldots , n$. 

Note that 
$$\widehat{K}[[y_1, \ldots , y_n, \partial_n^{-1}]] [[w_1(t), \ldots , w_{k-1}(t), t_k, \ldots , t_n]]= \widehat{K}[[y_1, \ldots , y_n, \partial_n^{-1}]] [[t_1, \ldots , t_n]].$$ 
Then, note that $\sigma (v_k(t))$ written as a series in the ring \\
$\widehat{K}[[y_1, \ldots , y_n, \partial_n^{-1}]] [[w_1, \ldots , w_{k-1}, t_k, \ldots , t_n]]$ depends on one of $t_k, \ldots , t_n$ (it is not difficult to see that $v_k(t)$ is  analytical also as a series in new variables). Indeed, if it does not depend on $t_k, \ldots , t_n$, then it is algebraically independent with $t_k, \ldots , t_n$. It is not difficult to see that $\sigma (v_1(t)), \ldots , \sigma (v_k(t))$ are algebraically independent over $K$ iff $\varphi (\sigma (v_1(t))), \ldots , \varphi (\sigma (v_k(t)))$ are algebraically independent for any linear change $\varphi$ from the formulation, and also iff $\sigma (w_1(t)), \ldots, \sigma (w_{k-1}(t)), \sigma (v_k(t))$ are algebraically independent. So,  $\sigma (v_k(t))$  is also algebraically independent with $\sigma (w_1), \ldots , \sigma (w_{k-1})$ (as it is algebraically independent with $\sigma (v_1), \ldots , \sigma (v_{k-1})$  by assumption). Thus, applying the inverse linear change of variables to the subspace generated by these elements, we get $n+1$ algebraically independent elements in the ring $A$ - a contradiction with lemma \ref{L:quasi-ell-properties}. 

Making a linear change of variables if necessary we can assume that $\sigma (v_k(t))$ depends on $t_k$. Then, making a linear change $\partial_i \mapsto \partial_i +\lambda_i \partial_n$, $i=1, \ldots ,k$ for generic $\lambda_i$ from some open subset $0\in U\subset \hat{K}^k$ (as in the first induction step) we'll get  a new series $v_k(\lambda )$ which contains a monomial of the form $ct_k$, $0\neq c\in \hat{K}$. At the same time, in the space $\hat{K}\langle w_1(\lambda ), \ldots , w_{k-1}(\lambda )\rangle$ we can still choose a basis with the same properties as before. Then taking an appropriate linear combination of this basis and new $v_k(\lambda )$, we can obtain a series of the form $v_k(\lambda )= \lambda_k+ (\cm)$ with $\ord_{\Gamma}(v_k)=\ord_{Gamma}(\sigma (v_k))=\ord_{\Gamma}(\lambda_k)$. 

 With this notation the minor has the form 
$$
\left (
\begin{array}{ccccccc}
1+ (\cm ) & (\cm ) & \ldots &  (\cm )& (\cm )\\
(\cm ) & 1 +(\cm ) & \ldots &   (\cm )& (\cm )\\
\vdots & \ldots  & \ddots &  (\cm )& \vdots \\
(\cm ) & \ldots & \ldots & 1 +(\cm ) & (\cm )\\
(\cm )& \ldots & \ldots & (\cm )  &  1 +(\cm ) \\
\end{array}
\right )
$$
which means that its determinant is not equal to zero, and we are done. 

\end{proof}

\begin{defin}
\label{D:admissible_changes}
Let $(A,W)$ be a 1-Schur pair. A linear change of variables $\varphi$ from lemma \ref{L:linear_changes} is called {\it admissible} if the spaces $\varphi (A), \varphi (W)$ are well defined (i.e. all elements $\varphi (a), \varphi (w)$, $a\in A$, $w\in W$ are well defined). 
\end{defin}

\begin{rem}
\label{R:admissible_changes}
1) If $K$ is a complete field and $(A,W)$ is a normalized 1-Schur pair associated with a quasi-elliptic ring $B\subset \hat{D}_n$ of rank one, and if there is a basis $\{ w_i\}$ in $W$ such that for all $i$ $w_i\in  K_y[[\partial_1', \ldots , \partial_{n-1}', \partial_n^{-1}]]$ (the notation is from lemma \ref{L:linear_changes})\footnote{This condition is trivial if $K_y=K$}, then there exist admissible changes. Indeed, by remark \ref{R:admissible_changes} $A\subset W\subset \Quot (A)$. Then from the proof of lemma \ref{L:linear_changes} we get that elements from $W, A$ are analytic functions (after multiplication with an appropriate power of $\partial_n$), and therefore there exist admissible changes, because the intersection of countably many open neighbourhoods of zero in $K^{n-1}$ is not empty for a complete field $K$. 

If $A$ is finitely generated over $K$ and $W$ is a finitely generated $A$-module, the same arguments show there exists an open subset $0\in U\subset K^{n-1}$ such that the linear change $\varphi$ from lemma \ref{L:linear_changes} is admissible for any $(\lambda_1, \ldots ,\lambda_{n-1})\in U$. 

2) Clearly, any invertible linear change of coordinates can be decomposed into a composition of a linear change from lemma \ref{L:linear_changes} and a linear change of the form $\partial_i \mapsto \partial_i$, $i=1, \ldots , n-1$,  $\partial_n \mapsto \alpha \partial_n +\sum_{i=1}^{n-1}c_i \partial_i$, $\alpha\neq 0$. Obviously, the last linear change is admissible, i.e. both spaces are well defined after its application. 

3) Note that if $B\subset \hat{D}_n$ is a quasi-elliptic ring and $\varphi$ is {\it any} linear change then $\varphi (B)$ is a well defined ring in $\hat{D}_n^{sym}$.

4) Let $S$ be the Sato operator from construction 1 constructed by given quasi-elliptic ring $B\subset \hat{D}_n$, let $(A,W)$ be the corresponding 1-Schur pair, let $\varphi$ be an admissible linear change. Then $\varphi (S)$ is also well defined operator from $\hat{\Pi}_n$. For, all slices of $S$ are finite linear combinations of elements from $W$, cf. the proof of theorem \ref{T:SatoAction}.

5) Let $(A,W)$ be a 1-Schur pair. Then for any admissible $\varphi$ we have $\varphi (a_1a_2)= \varphi (a_1) \varphi (a_2)$, $\varphi (aw)= \varphi (a) \varphi (w)$ for any $a_i, a\in A$, $w\in W$  -- this follows easily from the Taylor series expansion.
\end{rem}

\begin{prop}
\label{P:admissible_changes}
Let $K$ be a complete field. Let $B \subset \hat{D}_n$ be a quasi-elliptic ring, let $(A,W)$ be a normalized 1-Schur pair associated with $B$ by construction 1, let $\varphi$ be an admissible linear change of variables. 

Then there exists an operator $U\in \hat{U}_n^{sym}$ such that  $B':= U^{-1}\varphi (B)U\subset \hat{D}_n$. Moreover, $B'$ is the ring associated with the 1-Schur pair $(\varphi (A), \varphi (W))$ by the construction 2, and if this Schur pair is quasi-elliptic, then $B'$ is quasi-elliptic. 
\end{prop}

\begin{proof} Let $(A,W)$ be a normalized Schur pair associated with the ring $B$, $A=S^{-1}BS$, $W=F\diamond S$, $S\in \hat{E}_n$, cf. remark \ref{R:rank_one_pairs}. Then by remark \ref{R:admissible_changes}, for any $f\in A$ and corresponding operator $\mathtt{L}_S(f)$ the operators $\varphi (S)$, $\varphi (f)$, $\varphi (\mathtt{L}_S(f))$ are well defined and moreover, $\varphi (Sf)=\varphi (S)\varphi (f)$, $\varphi (\mathtt{L}_S(f)S)= \varphi (\mathtt{L}_S(f)) \varphi (S)$. Indeed, the products of operators on the right hand sides are well defined by their definition, the left hand sides are well defined as all slices of $Sf$ and $\mathtt{L}_S(f)S$ are finite sums of products of slices of $S$ and $f$ (for the first product) and of $S$ and polynomials  (for the second product). The identities then follows from the Taylor series expansions of left and right hand sides (alternatively one can compare $\varphi (\underline{\partial}^{\underline{j}})\diamond l.h.s$ and $\varphi (\underline{\partial}^{\underline{j}})\diamond r.h.s$ by induction on $|\underline{j}|$). Finally we have 
$$
\varphi (Sf)=\varphi (S)\varphi (f)=\varphi (\mathtt{L}_S(f)S)= \varphi (\mathtt{L}_S(f)) \varphi (S),
$$
i.e. $\varphi (S)\varphi (A)=\varphi (B) \varphi (S)$. Besides, $\varphi (W)= F \diamond \varphi (S)$ as $\varphi (F)=F$, i.e. $\varphi (S)$ is a Sato operator for $\varphi (W)$. 
By theorem \ref{T:SatoAction}, item 4 there is a unique factorisation $\varphi (S)= U\circ S_0$, where $U\in \hat{U}_n^{sym}$, hence $S_0\varphi (A)S_0^{-1}=U^{-1}\varphi (B)U$. From construction 2 we know (cf. also theorem \ref{T:SatoAction}, item 6) that $B'=S_0\varphi (A)S_0^{-1}\in \hat{D}_n$ and moreover it is quasi-elliptic if $\varphi (A)$ is. 
\end{proof}

\section{Schur theory for $\hat{D}_n^{sym}$}
\label{S:general_Schur_theory}

There is an analogue of the Schur theory also for the ring $\hat{D}_n^{sym}$. It is completely new and differs from the Schur theory for $\hat{D}_n$. 

\subsection{Centralizers of operators with constant coefficients}
\label{S:centralizers}

Let's introduce the following notation. For any $P\in \hat{D}_n^{sym}$ we put 
$$
P_{[q]}:= \sum_{\underline{k}\in \sdn_0^n, |\underline{k}|=q} 
\frac{x^{\underline{k}}}{\underline{k}!}P_{(\underline{k})}
$$ 
(a partial slice decomposition), where we use the slice decomposition \eqref{E:slice_decomposition1}. Thus, $P=\sum_{q\ge 0}P_{[q]}$. Put $\bar{P}:= P|_{y=0}$. 

Let $k\in\dn$ be a number. Fix a primitive $k$-th root of unity $\xi_k$ in an appropriate extension $\tilde{K}$ of the field $K$. For any $i\in \dz$ and $1\le q\le n$ we define operators 
$$
A_{k; i,q}:= \exp ((\xi_k^i -1)x_q*\partial_q)\in \hat{D}_n^{sym}\otimes_K \tilde{K} 
$$
(we use here the same notations as in theorem \ref{T:D_n_properties}). Recall also the operators of integration introduced in the proof of theorem \ref{T:D_n_properties}:
$$
\int_i:=(1-\exp ((-x_i)*\partial_i))\cdot \partial_i^{-1}=\sum_{k=0}^{\infty}\frac{x^{k+1}}{(k+1)!}(-\partial)^k.
$$
Note that all these operators are homogeneous and $\Ord (A_{k; i,q})=0$, $\Ord (\int_i)= -1$. Besides, $\partial_i \int_i=1$, but $\int_i\partial_i=1-\delta_i$.

\begin{prop}
\label{P:roots_of_unity}
The following properties hold: 
\begin{enumerate}
\item
For any $i, j, q, p, l$ we have 
$$
A_{k; i,q} A_{k; j,q}= A_{k; i+j,q}, \quad A_{k; i,q}A_{k; j,l}=A_{k; j,l}A_{k; i,q}, 
$$
$$
\partial_q^pA_{k; i,q}=\xi_k^{pi}A_{k; i,q}\partial_q^p, \quad A_{k; i,q}x_q^p= \xi_k^{pi}x_q^pA_{k; i,q}, \quad \int_q^pA_{k;i,q}=\xi_k^{-pi}A_{k;i,q}\int_q^p,
$$
$$
[\partial_q^p, A_{k; i,l}]=0, \quad [x_q^p, A_{k; i,l}]=0, \quad [\int_q^p, A_{k; i,l}]=0 \mbox{ for $q\neq l$} 
$$ 
In particular, $A_{k; i,q} A_{k; k-i,q}=1$. 
\item
For a given $Q\in \hat{D}_n^{sym}$ assume that $[\partial_q^k, Q]=0$ for some $q\in \{1, \ldots ,n\}$. 

Then $Q=c_0+c_1A_{k; 1,q}+\ldots +c_{k-1}A_{k; k-1,q}\in \hat{D}_n^{sym}\otimes_K\tilde{K}$, where $c_i\in \hat{D}_n^{sym}\otimes_K\tilde{K}$ are given by the following formula:
$$
c_i=\sum_{m=0}^{\infty} c_{i,m} \partial_q^{m} +c_{i,-1}\int_q+\ldots +c_{i,-k+1}\int_q^{k-1},
$$
where $c_{i,j}$ don't depend on $x_q, \partial_q$ and $\Ord (c_{i,j})\le \Ord (Q)-j$ for all $i,j$ (so that $\Ord (c_i)\le \Ord (Q)$). 

\item
In particular, if $[\partial_q^{k_j}, Q]=0$ for $q=i_1, \ldots , i_l \in \{1, \ldots ,n\}$, $j=1, \ldots ,l$, then 
$$
Q=\sum_{j_{1}=0}^{k_{1}-1}\ldots \sum_{j_{l}=0}^{k_{l}-1}c_{j_1,\ldots ,j_l}A_{k_1; j_1,i_1}\ldots A_{k_l; j_l,i_l},
$$
where $c_{j_1,\ldots ,j_l}\in \hat{D}_n^{sym}\otimes_K\tilde{K}$ are given by analogous formulas 
and $\Ord (c_{j_1,\ldots ,j_l})\le \Ord (Q)$. Besides, the coefficients $c_{j_1,\ldots ,j_l}$ are uniquely defined. 

If $K_y=K$ and $\{i_1, \ldots , i_l\}=\{1, \ldots ,n\}$, then all coefficients  $c_{j_1,\ldots ,j_l}$ are {\it polynomials} in $\partial_q$, $\int_q$, $q=1,\ldots , n$ with constant coefficients (from $\tilde{K}$) and the degree of these polynomials with respect to $\partial_q$ is not greater than $\Ord (Q)$ and the degree of these polynomials with respect to $\int_q$ is not greater than $k-1$. 
\end{enumerate}

\end{prop}

\begin{proof} 1. From general properties of operators (cf. theorem \ref{T:D_n_properties}) we have 
$$A_{k; i,q}(f(x_1, \ldots , x_n))=f(x_1, \ldots , \xi_k^ix_q,\ldots , x_n)$$ 
for any $f\in \tilde{K}[[x_1, \ldots , x_n]]$. All identities in item 1 immediately follow from this observation.

2. The identity $[\partial_q^k, Q]=0$ can be rewritten as 
\begin{equation}
\label{E:comm}
\sum_{i=1}^k C_k^i \partial_q^i(Q)\partial_q^{k-i}=0. 
\end{equation}
Note that any solution $Q\in \hat{D}_n^{sym}$ of this equation gives a solution \\
$Q'\in \tilde{K}[[x_1,\ldots , x_n, \tilde{\partial}_1, \ldots , \tilde{\partial}_n ]]$ of the equation  
$$
\sum_{i=1}^k C_k^i \partial_q^i(Q')\tilde{\partial}_q^{k-i}=0,
$$
where $\tilde{\partial}_q$ means a new formal variable (commutative with all $x_i$). Namely, we just replace $\partial_i$ by $\tilde{\partial}_i$ in the series $Q$. 

On the other hand, the last equation can be written in the form 
$$
\prod_{i=1}^k(\partial_q+ (1-\xi_k^i)\tilde{\partial}_q)(Q')=0.
$$
Any solution of the last equation {\it in the commutative ring} \\
$\tilde{R}:=\tilde{K}[[x_1, \ldots , x_n, \tilde{\partial}_1, \ldots , \widehat{\tilde{\partial}_q}, \ldots , \tilde{\partial}_n]]((\tilde{\partial}_q))$ has the form
$$
Q'= c_0 +c_1 \exp ((\xi_k-1)x_q\tilde{\partial}_q) +\ldots +c_{k-1}\exp ((\xi_k^{k-1}-1)x_q\tilde{\partial}_q), 
$$
where $c_i\in \tilde{R}$ don't depend on $x_q$ (as it follows from elementary differential algebra).
If we choose a canonical representation form of elements in $\tilde{R}$ such that each monomial has the form $\underline{x}^{\underline{j}}\underline{\tilde{\partial}}^{\underline{k}}$ (with $x_i$ on the left and 
$\tilde{\partial}_j$ on the right), then the right hand side of the last formula can be rewritten as 
$$
Q'':= \sum_{\underline{j}=(j_1,\ldots , \hat{j}_q, \ldots , j_n) \ge 0} \underline{x}^{\underline{j}}(c_{0, (\underline{j})} + \exp ((\xi_k-1)x_q\tilde{\partial}_q) c_{1, (\underline{j})} + \ldots + \exp ((\xi_k^{k-1}-1)x_q\tilde{\partial}_q) c_{k-1, (\underline{j})})=
$$ 
\begin{equation}
\label{E:summand'}
\sum_{\underline{j}=(j_1,\ldots , \hat{j}_q, \ldots , j_n) \ge 0} 
\underline{x}^{\underline{j}}(c_{0, (\underline{j})} + 
\tilde{A}_{k; 1,q} c_{1, (\underline{j})} + \ldots + 
\tilde{A}_{k; k-1,q} c_{k-1, (\underline{j})}),
\end{equation}
where $c_{i, (\underline{j})}$ denotes the corresponding slice of the coefficient $c_i$. Then $Q'=Q''$ also as elements written in this representation. Note that, since $Q'$ contains only non-negative powers of $\partial_q$, all brackets in \eqref{E:summand'} are series in which $\partial_q$ appears only with non-negative powers (and if we replace $\tilde{\partial}_i$, $i=1, \ldots ,n$ by $\partial_i$ in all terms of $Q''$, we get again the operator $Q$).

\begin{lem}
\label{L:various_decompositions}
For any $i=0, \ldots , k-1$ we have $\Ord (c_{i, (\underline{j})})\le \Ord (Q) +|\underline{j}|$ in formula \eqref{E:summand'}, where the order $\Ord$ is defined on $\tilde{R}$ in the same way as on $\hat{\cm}_n$.
\end{lem}

\begin{proof}
Since the elements $c_{i, (\underline{j})}$ are polynomials in $\tilde{\partial}_q^{-1}$, any bracket in \eqref{E:summand'} can be written in the form 
$$
(\tilde{c}_{0, (\underline{j})} + 
\tilde{A}_{k; 1,q} \tilde{c}_{1, (\underline{j})} + \ldots + 
\tilde{A}_{k; k-1,q} \tilde{c}_{k-1, (\underline{j})})\tilde{\partial}_q^{-m_{\underline{j}}},
$$
where $\tilde{c}_{i, (\underline{j})}\in \tilde{K}[[x_1, \ldots , x_n, \tilde{\partial}_1, \ldots , \tilde{\partial}_n]]$, and $m_{\underline{j}}\ge 0$, i.e. the series in brackets is divisible by  $\tilde{\partial}_q^{m_{\underline{j}}}$. We'll additionally assume that $m_{\underline{j}}$ is minimal, i.e. that $GCD (\tilde{c}_{0, (\underline{j})}, \ldots , \tilde{c}_{k-1, (\underline{j})})$ is not divisible by $\tilde{\partial}_q$. 

Obviously, the slice decomposition is unique also in the space $\tilde{R}$, and therefore for any fixed $\underline{j}$ in \eqref{E:summand'} we have the unique decomposition 
\begin{multline*}
c_{0, (\underline{j})} + \tilde{A}_{k; 1,q} c_{1, (\underline{j})} + \ldots + 
\tilde{A}_{k; k-1,q} c_{k-1, (\underline{j})} = 
\sum_{|\underline{l}|\ge -m_{\underline{j}}} (c_{0,\underline{l}} + c_{1, \underline{l}} \tilde{A}_{k; 1,q} + \ldots + 
c_{k-1, \underline{l}} \tilde{A}_{k; k-1,q})\underline{\tilde{\partial}}^{\underline{l}},
\end{multline*}
where $c_{i, \underline{l}} \in \tilde{K}_y$. Since $\Ord (Q) <\infty$ and $\Ord (A_{k; i,q})=0$, we should have 
$$
\Ord (c_{0, (\underline{j})} + \tilde{A}_{k; 1,q}  c_{1, (\underline{j})} + \ldots + 
\tilde{A}_{k; k-1,q} c_{k-1, (\underline{j})})\le \Ord (Q) +|\underline{j}|
$$
and therefore by lemma \ref{L:linear_independence} below $c_{i, \underline{l}}=0$ for all  $\underline{l}$ with $|\underline{l}|\ge \Ord (Q) +|\underline{j}|$ and all $i=0,\ldots ,k-1$. But this exactly means that $\Ord (c_{i, (\underline{j})})\le \Ord (Q) +|\underline{j}|$ for any $i$.
\end{proof}

\begin{lem}
\label{L:linear_independence}
The sum 
$$
A:=c_{0,\underline{l}} + c_{1, \underline{l}} \tilde{A}_{k; 1,q}+ \ldots + 
c_{k-1, \underline{l}} \tilde{A}_{k; k-1,q}, \quad c_{i, \underline{l}}\in \tilde{K}_y
$$
is equal to zero iff  $c_{i, \underline{l}}=0$, $i=0,\ldots ,k-1$. If it is not equal to zero, then it is of order zero.
\end{lem}

\begin{proof}
Clearly, $\Ord (A)\le 0$. Note that the equality $A=0$ is equivalent to the infinite system of linear equations
$$
c_{0,\underline{l}}  + c_{1, \underline{l}} + \ldots + 
c_{k-1, \underline{l}}=0, \quad  c_{1, \underline{l}} (\xi_k-1)^j + \ldots + 
c_{k-1, \underline{l}} (\xi_k^{k-1}-1)^j =0, \quad j\in \dn .
$$
By the well known property of the Vandermonde matrix this system has the unique solution $c_{i, \underline{l}}=0$, $i=0,\ldots ,k-1$. The second claim is obvious.
\end{proof}

From lemma \ref{L:various_decompositions} it follows that the series $\tilde{c}_{i, (\underline{j})}$ will belong to $\hat{D}_n^{sym}$ after replacing $\tilde{\partial}_i$ by $\partial_i$ in all terms. Now note that in $\hat{D}_n^{sym}$ we have 
\begin{multline*}
(\tilde{c}_{0, (\underline{j})} + \tilde{A}_{k; 1,q} \tilde{c}_{1, (\underline{j})} + \ldots + 
\tilde{A}_{k; k-1,q} \tilde{c}_{k-1, (\underline{j})})\partial_q^{-m_{\underline{j}}}|_{\tilde{\partial}_i\mapsto \partial_i}=\\
(\tilde{c}_{0, (\underline{j})} + \tilde{A}_{k; 1,q} \tilde{c}_{1, (\underline{j})} + \ldots + 
\tilde{A}_{k; k-1,q} \tilde{c}_{k-1, (\underline{j})})|_{\tilde{\partial}_i\mapsto \partial_i}\int_q^{m_{\underline{j}}},
\end{multline*}
i.e. $Q$ can be written in the form:
\begin{equation}
\label{E:summand}
Q= \sum_{\underline{j}=(j_1,\ldots , \hat{j}_q, \ldots , j_n) \ge 0} 
\underline{x}^{\underline{j}}(\tilde{c}_{0, (\underline{j})} + \tilde{A}_{k; 1,q} \tilde{c}_{1, (\underline{j})} + \ldots + 
\tilde{A}_{k; k-1,q} \tilde{c}_{k-1, (\underline{j})})|_{\tilde{\partial}_i\mapsto \partial_i}\int_q^{m_{\underline{j}}}.
\end{equation}

Besides, all summands in the sum \eqref{E:summand} are well defined elements of $\hat{D}_n^{sym}$ of order $\le \Ord (Q)$ and their sum is also well defined in $\hat{D}_n^{sym}$. 

\begin{lem}
\label{L:restriction_on_int}
In the formula \eqref{E:summand} we have $m_{\underline{j}}\le k-1$. 
\end{lem}

\begin{proof}
First note that all numbers $m_{\underline{j}}$ are bounded from above, because the coefficients $c_i$ of $Q'$ are polynomials in $\tilde{\partial}_q^{-1}$. Now assume the converse: there exists $m:=m_{\underline{j}}\ge k$. Assume that $m$ is a maximal such number. Let $Q_l$ be a homogeneous component of $Q$ containing  a summand with $\int_q^{m}$. Note that each summand of $Q_l$ has the form $c\partial_q^i$ or $c\int_q^j$, where $c$ is a linear combination of $A_{k;s,q}$, $s=0,\ldots , k-1$ with coefficients not depending on $x_q, \partial_q$, and $j\le m$, i.e.
\begin{multline}
\label{E:Q_l}
Q_l=\sum_{i=0}^{\infty} (a_{i,0}+a_{i,1}A_{k;1,q}+\ldots +a_{i,k-1}A_{k;k-1,q})\partial_q^i +\\
(a_{-1,0}+\ldots +a_{-1,k-1}A_{k;k-1,q})\int_q+\ldots + (a_{-m,0}+a_{-m,1}A_{k;1,q}+\ldots +a_{-m,k-1}A_{k;k-1,q})\int_q^m.
\end{multline}

Since $Q_l$ is homogeneous, we have $\Ord (a_{i,0}+a_{i,1}A_{k;1,q}+\ldots +a_{i,k-1}A_{k;k-1,q})=l-i$. Since $(a_{-m,0}+a_{-m,1}A_{k;1,q}+\ldots +a_{-m,k-1}A_{k;k-1,q})\neq 0$, there is $p\ge 0$ such that 
$$
(a_{-m,0}+\frac{a_{-m,1}(\xi_k-1)^px_q^p\partial_q^p}{p!}+\ldots +\frac{a_{-m,k-1}(\xi_k^{k-1}-1)^px_q^p\partial_q^p}{p!})\neq 0.
$$ 
We'll assume that $p$ is minimal such number. From lemma \ref{L:linear_independence} it follows that $p\le k-1$. 

Since $[\partial_q^k,Q]=0$, it follows 
$$
0=[\partial_q^k,Q_l]=\partial_q^k(Q_l)+k\partial_q^{k-1}(Q_l)\partial_q+\ldots +k\partial_q(Q_l)\partial_q^{k-1}.
$$ 
Now let's expand all series in this sum and collect all summands containing $x_q^{m-k+p}\partial_q^p$. Note that there are no such summands except the expression 
$$
m\ldots (m-k+1)(a_{-m,0}+\frac{a_{-m,1}(\xi_k-1)^px_q^p\partial_q^p}{p!}+\ldots +\frac{a_{-m,k-1}(\xi_k^{k-1}-1)^px_q^p\partial_q^p}{p!})x_q^{m-k+p}\partial_q^p
$$ 
(coming from $(a_{-m,0}+a_{-m,1}A_{k;1,q}+\ldots +a_{-m,k-1}A_{k;k-1,q})\partial_q^k(\int_q^m)$), and since this expression is not zero, we get a contradiction. Indeed, 
for $i\ge 0$ and any $0<l\le k$ we have
$$
\partial_q^l((a_{i,0}+a_{i,1}A_{k;1,q}+\ldots +a_{i,k-1}A_{k;k-1,q})\partial_q^i)\partial_q^{k-l}=
(a_{i,1}'A_{k;1,q}+\ldots +a_{i,k-1}'A_{k;k-1,q})\partial_q^{i+k} 
$$ 
(here $a_{i,r}'=a_{i,j}\xi_k^{jr}$), and all monomials containing $x_q^{m-k+p}$ from this series contain also $\partial_q^{m+p+i}$. For other expressions from \eqref{E:Q_l}  we have
\begin{multline*}
\partial_q^l((a_{i,0}+a_{i,1}A_{k;1,q}+\ldots +a_{i,k-1}A_{k;k-1,q})\int_q^j)\partial_q^{k-l}=\\
\partial_q^l(\int_q^j(a_{i,0}+a_{i,1}'A_{k;1,q}+\ldots +a_{i,k-1}'A_{k;k-1,q}))\partial_q^{k-l}=\\
\sum_{\mu + \nu =l}c_{\mu ,\nu}\partial_q^{\nu}(\int_q^i)\partial_q^{\mu}(a_{i,0}+a_{i,1}'A_{k;1,q}+\ldots +a_{i,k-1}'A_{k;k-1,q})\partial_q^{k-l}=\\
\sum_{\mu + \nu =l}c_{\mu ,\nu}\partial_q^{\nu}(\int_q^i)(a_{i,1}'(\xi_k-1)^{\mu}A_{k;1,q}+\ldots +a_{i,k-1}'(\xi_k^{k-1}-1)^{\mu}A_{k;k-1,q})\partial_q^{k-\nu},
\end{multline*}
where $c_{\mu, \nu}\in \dn$ and $a_{i,r}'=a_{i,j}\xi_k^{jr}$. Now consider any series 
$$
\partial_q^{\nu}(\int_q^i)(a_{i,1}'(\xi_k-1)^{\mu}A_{k;1,q}+\ldots +a_{i,k-1}'(\xi_k^{k-1}-1)^{\mu}A_{k;k-1,q})\partial_q^{k-\nu}
$$ 
from this formula and note that all monomials containing $x_q^{m-k+p}$ from this series contain also $\partial_q^{m-k+p-i+\nu +k-\nu}=\partial_q^{m+p-i}$. If $i<m$, then $m+p-i>p$, and therefore there are no terms containing $x_q^{m-k+p}\partial_q^p$.

If $i=m$, but $\nu <k$, then  
$$
(a_{i,1}'(\xi_k-1)^{\mu +s}x_q^s\partial_q^s+\ldots +a_{i,k-1}'(\xi_k^{k-1}-1)^{\mu +s}x_q^s\partial_q^s)=0
$$
for $s=0, \ldots , \nu -k+p$ by the condition on $p$ (its minimality), and therefore there are no terms containing $x_q^{m-k+p}\partial_q^p$ again. 
\end{proof}

 Now, expanding all brackets in \eqref{E:summand} and using the identities from item 1, we can rewrite $Q$ in the form stated in item 2.
 
3. This item immediately follows from item 2 by induction.  Let's show the uniqueness of the coefficients. 

It suffices to prove that the equality 
$$
0=\sum_{j_{1}=0}^{k_{1}-1}\ldots \sum_{j_{l}=0}^{k_{l}-1}c_{j_1,\ldots ,j_l}A_{k_1; j_1,i_1}\ldots A_{k_l; j_l,i_l}
$$
implies the equalities $c_{j_1,\ldots ,j_l}=0$. By lemma \ref{L:statement} it is equivalent to the system of  equalities $\underline{\partial}^{\underline{q}}\diamond A=0$, where $A$ means our sum above and $\underline{q}\in \dn_0^n$. Note that this system is equivalent to the system $A_{(\underline{q})}=0$, $\underline{q}\in \dn_0^n$. 

Let $\underline{i}$ be the subset in $\{1, \ldots ,n\}$ complement to the set $\{i_1, \ldots ,i_l \}$. Then note that the last system is equivalent to the system 
\begin{equation}
\label{E:nach_uslovija}
\sum_{j_{1}=0}^{k_{1}-1}\ldots \sum_{j_{l}=0}^{k_{l}-1}(c_{j_{1},\ldots ,j_l})_{(\underline{i})} \xi_{k_{1}}^{q_{1}j_{1}}\ldots \xi_{k_l}^{q_lj_l}=0, \quad \underline{q}\in \dn_0^n.
\end{equation}

Now note that system \eqref{E:nach_uslovija} is a system of linear equations with indeterminates $(c_{j_{1},\ldots ,j_l})_{(\underline{i})}$ and the matrix $(\xi_{k_{1}}^{q_{1}j_{1}}\ldots \xi_{k_l}^{q_lj_l})$, which is just the matrix of the tensor product of the Vandermonde matrices $(\xi_{k_{s}}^{q_{s}j_{s}})$, $q_s, j_s=0, \ldots , k_s$. Hence it is invertible and the system has a unique (zero) solution. 

\end{proof}

\begin{defin}
\label{D:p-regular}
An element $P\in \hat{D}_n^{sym}$ is called $p$-regular, where $p\in \dn_0$, if the elements of the set $\{\underline{\partial}^{\underline{k}}\diamond \sigma (P)| \quad \underline{k}\in \dn_0^n: |\underline{k}|\le p\}\subset F$ are linearly independent. 
\end{defin}

We start with the Schur theory for the case $n=1$.

\subsection{Schur theory, $n=1$}
\label{S:Schur,n=1}

\begin{prop}
\label{P:Schur,n=1}
Let $P\in \hat{D}_1^{sym}$ be a regular operator with $\Ord (P)=\Ord (\bar{P})=k>0$ and such that $\bar{P}$ is regular.

Then there exists an invertible operator $S\in \hat{D}_1^{sym}$ with $\Ord (S)=0$ such that \\
$P=S^{-1}\partial_1^kS$. 

\end{prop}

\begin{proof}
We'll find the operator $S$ recursively, looking for the partial slice decompositions $S_{[q]}$, 
$S=\sum_{q\ge 0}S_{[q]}$. 

For any $p\ge 0$ we need to solve the system
$$
(\partial_1^kS)_{[p]}=(SP)_{[p]}.
$$
Note that $(\partial_1^kS)_{[p]}=\partial_1^k(S_{[p+k]}) +\mbox{expression depending on $S_{[q]}$, $q<p+k$}$, 
and $(SP)_{[p]} = ((S_{[0]}+\ldots +S_{[p]})P)_{[p]}$. Put $S_{[0]}:=1$, $S_{[p]}:=0$ for $1\le p\le k-1$. Then for $p=0$ we have
$$
(\partial_1^kS)_{[0]}=\partial_1^k(S_{[k]})+\partial_1^k=P_{[0]},
$$
and therefore the slice $S_{(k)}$ is uniquely determined; besides, $\Ord (S_{[k]})\le 0$. By the same reason all partial slice decompositions $S_{[p]}$, $p>k$ will be uniquely determined by this system and by induction $\Ord (S_{[p]})\le 0$ as well. Since $\Ord (S_{[0]})=0$, we get $\Ord (S)=0$ too. Note that $(\partial_1^k\bar{S})_{[p]}=(\bar{S}\bar{P})_{[p]}$ automatically. 

By corollary \ref{C:units} $S$ is invertible iff $S$, $\bar{S}$ are regular. First note that $S$, $\bar{S}$ are $k$-regular: 
for any $0\le i\le k$ we have 
$$
\partial_1^i\diamond S=\partial_1^i\diamond (S_{[0]}+\ldots + S_{[k]})=\left\{
\begin{array}{c}
\partial_1^i \quad \mbox{if $i<k$}\\
P_{[0]}  \quad \mbox{if $i=k$}
\end{array}
\right . ,
\quad 
$$
$$
\partial_1^i\diamond \bar{S}=\partial_1^i\diamond (\bar{S}_{[0]}+\ldots + \bar{S}_{[k]})=\left\{
\begin{array}{c}
\partial_1^i \quad \mbox{if $i<k$}\\
\bar{P}_{[0]}  \quad \mbox{if $i=k$}
\end{array}
\right . 
$$
Since $P$, $\bar{P}$ are regular, we have $\sigma (P)_{[0]}\neq 0$, $\Ord (\sigma (P)_{[0]})= \Ord (P_{[0]})=k$ and analogously $\overline{\sigma (P)}_{[0]}\neq 0$, $\Ord (\overline{\sigma (P)}_{[0]})= \Ord (\bar{P}_{[0]})=k$, therefore the elements \\
$1, \partial_1, \ldots , \partial_1^{k-1}, \overline{\sigma (P_{[0]})}$ are linearly independent. 

Now we can prove that $S$ is $p$-regular for $p>k$ by induction. For any $p>k$ we have
$$
\partial_1^p\diamond S= \partial_1^{p-k}\diamond (\partial_1^k S)=\partial_1^{p-k}\diamond (SP)=
\partial_1^{p-k}\diamond (S_{[0]}+\ldots +S_{[p-k]})P,
$$
$$
\partial_1^p\diamond \bar{S}= \partial_1^{p-k}\diamond (\partial_1^k \bar{S})=\partial_1^{p-k}\diamond (\bar{S}\bar{P})=
\partial_1^{p-k}\diamond (\bar{S}_{[0]}+\ldots +\bar{S}_{[p-k]})\bar{P}.
$$
Since $P$, $\bar{P}$ are regular  and $(S_{[0]}+\ldots +S_{[p-k]})$, $(\bar{S}_{[0]}+\ldots +\bar{S}_{[p-k]})$ are $(p-k)$-regular by induction, the symbols of the last elements in the equalities are not equal to zero (cf. lemma \ref{L:regularOp}), i.e. $S$, $\bar{S}$ are $p$-regular and therefore they are regular by induction.
 
\end{proof}

\begin{ex}
\label{Ex:Busov}
Let's illustrate propositions \ref{P:roots_of_unity} and  \ref{P:Schur,n=1} on one well known Wallenber's example of two commuting ordinary differential operators of orders 2 and 3 with rational coefficients:
$$
L=\partial^2-\frac{2}{(x+1)^2}, \quad P= 4\partial^3 -\frac{12}{(x+1)^2}\partial+\frac{12}{(x+1)^3}.
$$
Then direct calculations with the help of these propositions show the existence of an operator $S$ such that $SLS^{-1}=\partial^2$, and 
$$
SPS^{-1}=(4 \partial^3 +2 \int )-(4(1-\partial )-2 \int )A_{2;1,1}.
$$
\end{ex}

\subsection{Schur theory, general case}
\label{S:Schur,n}

\begin{prop}
\label{P:Schur,n}
Let $P\in \hat{D}_n^{sym}$ be a regular operator with $\Ord (P)=\Ord (\bar{P})=k>0$ and such that $\bar{P}$ is regular.

Then for any $i=1, \ldots ,n$ there exists an invertible operator $S_i\in \hat{D}_n^{sym}$ with $\Ord (S_i)=0$, $S_{i,0}=1$ such that $P=S_i^{-1}\partial_i^kS_i$. 
\end{prop}

\begin{proof}
The proof is similar to proposition \ref{P:Schur,n=1}. First, note that $\sigma (P)_{[0]}=\sigma (P_{[0]})\neq 0$ (since $P$ is regular) and therefore $\Ord (P_{[0]})=\Ord (P)$. Analogously, $\sigma (\bar{P})_{[0]}=\sigma (\bar{P}_{[0]})\neq 0$ and $\Ord (\bar{P}_{[0]})=\Ord (\bar{P})$.

We'll find the operator $S_i$ recursively, looking for the partial slice decompositions $S_{i,[q]}$, 
$S_i=\sum_{q\ge 0}S_{i,[q]}$. 

For any $p\ge 0$ we need to solve the system
\begin{equation}
\label{E:system0}
(\partial_i^{k}S_i)_{[p]}=(S_iP)_{[p]}.
\end{equation}
Note that 
$(\partial_i^{k}S_i)_{[p]}=\partial_i^{k}(S_{i,[p+k]}) +\mbox{expression depending on $S_{i,[q]}$, $q<p+k$}$, 
and 
$$(S_iP)_{[p]} = ((S_{i,[0]}+\ldots +S_{i,[p]})P)_{[p]}.$$
In particular, $\partial_i^{k}(S_{i,[p+k]})$ is uniquely determined via expressions depending on $S_{i,[q]}$, $q<p+k$ and $P$, i.e. all slices $S_{i,(\underline{j})}$, where $\underline{j}=(j_1,\ldots ,j_n)$ with $j_i\ge k$ and $|\underline{j}|=k$, are uniquely determined. 

Put $S_{i,[0]}:=1$, $S_{i,[p]}:=0$ for $1\le p\le k-1$. Then for $p=0$ we have the system
$$
(\partial_i^{k}S_i)_{[0]}=\partial_i^{k}(S_{i,[k]})+\partial_i^{k}=P_{[0]}.
$$
From this equation the slice $S_{i,(\underline{j})}$ with $\underline{j} = (0, \ldots ,0, k)$ is uniquely  determined, and we define all other slices  $S_{i,(\underline{j})}$ with $|\underline{j}| =k$, $\underline{j} \neq (0, \ldots ,0, k)$  in such a way that the elements $\underline{\partial}^{\underline{j}} +\sigma (S_{i,(\underline{j})})$ are linearly independent with $\sigma (P)_{[0]}$ (and the elements $\underline{\partial}^{\underline{j}} +\sigma (\bar{S}_{i,(\underline{j})})$ are linearly independent with $\sigma (\bar{P})_{[0]}$ correspondingly). Then we get $\Ord (S_{i, [k]})\le 0$ (as $\Ord (\partial_i^{k}) = k$ and $\Ord (P_{[0]})= k$). 

Note that the operators $(S_{i,[0]}+\ldots + S_{i,[k]})$, $(\bar{S}_{i,[0]}+\ldots + \bar{S}_{i,[k]})$ are $k$-regular: 
for any $\underline{j}\in \dn_0^n$ with $|\underline{j}|\le k$ we have 
$$
\underline{\partial}^{\underline{j}}\diamond (S_{i,[0]}+\ldots + S_{i,[k]})=\left\{
\begin{array}{c}
\underline{\partial}^{\underline{j}} \quad \mbox{if $|\underline{j}|< k$}\\
\underline{\partial}^{\underline{j}}+ S_{i,(\underline{j})} \quad \mbox{if $|\underline{j}|= k$, $\underline{j}\neq (0, \ldots ,0, k) $}\\
P_{[0]}  \quad \mbox{if $\underline{j}= (0, \ldots ,0, k) $}
\end{array}
\right . ,
$$
$$
\underline{\partial}^{\underline{j}}\diamond (\bar{S}_{i,[0]}+\ldots + \bar{S}_{i,[k]})=\left\{
\begin{array}{c}
\underline{\partial}^{\underline{j}} \quad \mbox{if $|\underline{j}|< k$}\\
\underline{\partial}^{\underline{j}}+ \bar{S}_{i,(\underline{j})} \quad \mbox{if $|\underline{j}|= k$, $\underline{j}\neq (0, \ldots ,0, k) $}\\
\bar{P}_{[0]}  \quad \mbox{if $\underline{j}= (0, \ldots ,0, k) $}
\end{array}
\right .
$$
By construction of the slices  $S_{i,(\underline{j})}$, $\bar{S}_{i,(\underline{j})}$ all symbols of these elements are linearly independent. 

By the same reason all partial slice decompositions $S_{i,[p+k]}$, $p>0$ can be  determined by the equation
$$
(\partial_i^{k}S_i)_{[p]}=(\partial_i^{k}(S_{i,[0]}+\ldots +S_{i,[p+k]}))_{[p]}=(S_iP)_{[p]}=((S_{i,[0]}+\ldots +S_{i,[p]})P)_{[p]},
$$
 where all slices  $S_{i,(\underline{j})}$, where $|\underline{j}| =p+k$, $\underline{j}=(j_1, \ldots , j_n)$ with $j_i\ge k$,  are uniquely determined by the equation, and to define the slices  $S_{i,(\underline{j})}$, where $|\underline{j}| =p+k$, $\underline{j}=(j_1, \ldots , j_n)$ with $j_i<k$, let's 
consider first elements $v_{k, \underline{j}}:= S_{i,(\underline{j})} +\underline{\partial}^{\underline{j}} \diamond (S_{i,[0]}+\ldots +S_{i,[p+k-1]})$, where $\underline{j} =(j_1, \ldots , j_n)$ with $j_i\ge k$, $|\underline{j}|=p+k$. Observe that for all $q< p$ we have 
\begin{multline*}
(\partial_i^{k}(S_{i,[0]}+\ldots +S_{i,[p+k]}))_{[q]}=(\partial_i^{k}(S_{i,[0]}+\ldots +S_{i,[q+k]}))_{[q]}= \\
((S_{i,[0]}+\ldots +S_{i,[q]})P)_{[q]}=((S_{i,[0]}+\ldots +S_{i,[p]})P)_{[q]}
\end{multline*}
by induction, and therefore
\begin{multline*}
v_{k, \underline{j}} = \underline{\partial}^{(j_1, \ldots ,j_i- k, \ldots , j_n)}\diamond (\partial_i^{k} (S_{i,[0]}+\ldots +S_{i,[p+k-1]}+\frac{x^{\underline{j}}}{\underline{j}!}S_{i,(\underline{j})}))=\\
\underline{\partial}^{(j_1, \ldots ,j_i- k, \ldots , j_n)}\diamond (\sum_{q=0}^p (\partial_i^{k} (S_{i,[0]}+\ldots +S_{i,[p+k-1]}+\frac{x^{\underline{j}}}{\underline{j}!}S_{i,(\underline{j})}))_{[q]})=\\
\underline{\partial}^{(j_1, \ldots ,j_i- k, \ldots , j_n)}\diamond (\sum_{q=0}^p ((S_{i,[0]}+\ldots +S_{i,[p]})P)_{[q]})=\underline{\partial}^{(j_1, \ldots ,j_i- k, \ldots , j_n)}\diamond ((S_{i,[0]}+\ldots +S_{i,[p]})P),
\end{multline*}
where from the elements $v_{k, \underline{j}}$ are linearly independent (moreover, their symbols are linearly independent), since $P$ is regular  and $(S_{i,[0]}+\ldots +S_{i,[p]})$ is $p$-regular by induction. Analogously defined elements $v_{\underline{j}}:=\underline{\partial}^{\underline{j}} \diamond (S_{i,[0]}+\ldots +S_{i,[p+k-1]})$ with $|\underline{j}|<p+k$ and their symbols are also linearly independent by induction, and symbols of the elements $v_{\underline{j}}$, $v_{k, \underline{j}}$ are linearly independent because elements from two different groups (with $|\underline{j}|<p$ and with $|\underline{j}|=p$) have different $\Ord$- orders. 

Now we define the slices  $S_{i,(\underline{j})}$, where $|\underline{j}| =p+k$, $\underline{j}=(j_1, \ldots , j_n)$ with $j_i<k$, in such a way that symbols of the elements $v_{p, \underline{j}}:= S_{i,(\underline{j})} +\underline{\partial}^{\underline{j}} \diamond (S_{i,[0]}+\ldots +S_{i,[p+k-1]})$ 
are linearly independent with the symbols of elements $v_{k, \underline{j}}$ and  $v_{\underline{j}}$, $|\underline{j}| < p+k$ (and analogous properties hold for $\bar{v}_{p, \underline{j}}$). 

As a result, the operators $(S_{i,[0]}+\ldots +S_{i,[p+k]})$, $(\bar{S}_{i,[0]}+\ldots +\bar{S}_{i,[p+k]})$ are $(p+k)$-regular and satisfy the equations
$$
(\partial_i^{k}(S_{i,[0]}+\ldots +S_{i,[p+k]}))_{[p]}=((S_{i,[0]}+\ldots +S_{i,[p]})P)_{[p]}= ((S_{i,[0]}+\ldots +S_{i,[p+k]})P)_{[p]},
$$
$$
(\partial_i^{k}(\bar{S}_{i,[0]}+\ldots +\bar{S}_{i,[p+k]}))_{[p]}=((\bar{S}_{i,[0]}+\ldots +\bar{S}_{i,[p]})\bar{P})_{[p]}= ((\bar{S}_{i,[0]}+\ldots +\bar{S}_{i,[p+k]})\bar{P})_{[p]}.
$$
By induction and construction $\Ord (S_{i,[p+k]})\le 0$ as well. 

Now consider the operator $S_i:= \sum_{q\ge 0}S_{i,[q]}$. Since $\Ord (S_{i,[0]})=0$, we get $\Ord (S_i)=0$ too.  Moreover, $S$, $\bar{S}$ are regular (as they are $p$-regular for all $p>0$) and satisfy equation \eqref{E:system0} and analogous equation for $\bar{S}$ (for all partial slice decompositions).  By corollary \ref{C:units} $S_i$ is invertible, and we are done.
\end{proof}

\begin{thm}
\label{T:Schur,n}
Let $P_1, \ldots , P_n\in \hat{D}_n^{sym}$ be commuting operators with $\Ord (P_i)=\Ord (\bar{P}_i)=k$ for all $i=1,\ldots ,n$.  Assume that the spectral module $F$ of the ring $K[\sigma (\bar{P}_1), \ldots , \sigma (\bar{P}_n)]$ is finitely generated and Cohen-Macaulay. 

Then there exists an invertible operator $S\in \hat{D}_n^{sym}$ with $\Ord (S)=0$ such that
$$
S^{-1}\partial_i^{k_i}S=P_i, \quad i=1, \ldots ,n. 
$$
\end{thm}

\begin{rem}
The condition on the spectral module is satisfied for many examples arising in quantum integrable systems, e.g. for rings of commuting partial differential operators, cf. \cite[Th. 4]{KurkeZheglov}, \cite[Cor. 1]{Zheglov2018}, \cite[Th. 4.1]{Zheglov2013}. Further examples will appear in a subsequent paper.
\end{rem}

\begin{proof} The proof is similar to proposition \ref{P:Schur,n}.

{\bf Construction.}  Namely, for any $p\ge 0$ we need to solve the system
\begin{equation}
\label{E:system}
(\partial_n^{k_n}S)_{[p]}=(SP_{n})_{[p]}, \quad , \ldots , \quad (\partial_1^{k_1}S)_{[p]}=(SP_1)_{[p]}
\end{equation}
We have again  
$$
(\partial_i^{k_i}S)_{[p]}=\partial_i^{k_i}(S_{[p+k_n]}) +\mbox{expression depending on $S_{[q]}$, $q<p+k_i$},
$$ 
and $(SP_i)_{[p]} = ((S_{[0]}+\ldots +S_{[p]})P_i)_{[p]}$. In particular, if this system is compatible, the expressions $\partial_i^{k_i}(S_{[p+k_i]})$ are uniquely determined via expressions depending on $S_{[q]}$, $q<p+k_i$ and $P_i$. 

Put $S_{[0]}:=1$, $S_{(\underline{j})}:=0$ for $\underline{j}=(j_1, \ldots , j_n)\neq \underline{0}$ with $j_i<k_i$, $i=1, \ldots ,n$. Then for $p=0$ we have
$$
(\partial_n^{k_n}S)_{[0]}=\partial_n^{k_n}(S_{[k_n]})+\partial_n^{k_n}=P_{n,[0]}, \quad , \ldots , \quad
(\partial_1^{k_1}S)_{[0]}=\partial_1^{k_1}(S_{[k_1]})+\partial_1^{k_1}=P_{1,[0]}
$$
From this system the slices $S_{(\underline{j})}$ with $\underline{j} = (0, \ldots ,0, k_i, 0, \ldots ,0)$, $i=1, \ldots ,n$ are uniquely  determined. 

%Without loss of generality assume $k_1\le k_2\le \ldots \le k_n$. 

Note that we have defined the partial slice decompositions $S_{[q]}$ with $q\le k_1$. We will define slices $S_{(\underline{j})}$  with $q:=|\underline{j}|>k_1$ by induction on $q$. Assume we have defined all slices $S_{(\underline{j})}$ with $|\underline{j}|< q$. We first determine the slices $S_{(\underline{j})}$ with $j_1\ge k_1$ from the equation
$$
(\partial_1^{k_1}S)_{[q-k_1]}=(SP_1)_{[q-k_1]}
$$
of the system \eqref{E:system}. We claim that other equations of this system remain compatible, i.e. for all $\underline{j}$ with $j_1\ge k_1$ as above we have 
$$
(\partial_i^{k_i}S)_{(\underline{j})}=(SP_i)_{(\underline{j})}, \quad i=2, \ldots , n.
$$
Indeed, we have on the one side
$$
(\partial_i^{k_i}S)_{(\underline{j})}= \underline{\partial}^{\underline{j}}\diamond (\partial_i^{k_i}S)=
(\underline{\partial}^{\underline{j}}\partial_i^{k_i}S)_{[0]}=(\partial_i^{k_i}\underline{\partial}^{\underline{j}-(k_1,0,\ldots ,0)}SP_1)_{[0]}=(\underline{\partial}^{\underline{j}-(k_1,0,\ldots ,0)}\partial_i^{k_i}S)_{[0]}\diamond P_1,
$$
and on the other side
\begin{multline*}
(SP_i)_{(\underline{j})}=\underline{\partial}^{\underline{j}}\diamond (SP_i)=(\underline{\partial}^{\underline{j}} SP_i)_{[0]} = (\underline{\partial}^{\underline{j}-(k_1,0,\ldots ,0)}SP_1P_i)_{[0]} = \\
(\underline{\partial}^{\underline{j}-(k_1,0,\ldots ,0)}SP_iP_1)_{[0]} = 
(\underline{\partial}^{\underline{j}-(k_1,0,\ldots ,0)}SP_i)_{[0]}\diamond P_1.
\end{multline*}

By induction, we have 
\begin{multline*}
(\underline{\partial}^{\underline{j}-(k_1,0,\ldots ,0)}\partial_i^{k_i}S)_{[0]} = \underline{\partial}^{\underline{j}-(k_1,0,\ldots ,0)}\diamond (\partial_i^{k_i}S)= (\partial_i^{k_i}S)_{(j_1-k_1,j_2, \ldots , j_n)}= \\
(SP_i)_{(j_1-k_1,j_2, \ldots , j_n)}= (\underline{\partial}^{\underline{j}-(k_1,0,\ldots ,0)}SP_i)_{[0]},
\end{multline*}
thus $(\partial_i^{k_i}S)_{(\underline{j})}=(SP_i)_{(\underline{j})}$. 

Therefore, continuing this line of reasoning, we  can determine recursively the slices $S_{(\underline{j})}$ with $j_i\ge k_i$, $j_l<k_l$ for $l<i$ from the equation
$$
(\partial_i^{k_i}S)_{[q-k_i]}=(SP_i)_{[q-k_i]}
$$
for each $i=2, \ldots , n$. 

At the end of the induction procedure we'll get the operator $S$. Note that $\Ord (S)=0$ by construction.
Automatically we get also the analogue of the equation \eqref{E:system} for $\bar{S}$, $\bar{P}$.

{\bf Regularity.} By corollary \ref{C:units} $S$ is invertible iff $\bar{S}$ is regular. Note that for any $\underline{j}=(l_1+q_1k_1, l_2+q_2k_2, \ldots , l_n+q_nk_n)$, where $q_i\in \dz_+$, $l_i<k_i$ for $i=1, \ldots , n$, we have 
$$
\underline{\partial}^{\underline{j}}\diamond \sigma (S)= \sigma (\underline{\partial}^{\underline{j}}\diamond \bar{S}) = \sigma (\underline{\partial}^{\underline{l}}\diamond \bar{S} \bar{P}_1^{q_1}\ldots \bar{P}_n^{q_n})=
\sigma (\underline{\partial}^{\underline{l}}\diamond \bar{P}_1^{q_1}\ldots \bar{P}_n^{q_n})= 
\underline{\partial}^{\underline{l}}\diamond \sigma (\bar{P}_1^{q_1}\ldots \bar{P}_n^{q_n}).
$$
Now note that the elements $\underline{\partial}^{\underline{l}}$ belong to the set of generators of the free \\
$K[\sigma (\bar{P}_1), \ldots , \sigma (\bar{P}_n)]$-module $F$ ($F$ is free since it is Cohen-Macaulay, see \cite{BH}). Indeed, for any fixed system of generators $e_1, \ldots , e_m$ the $\Ord$-order of any linear combination $\alpha_1 e_1 + \ldots +\alpha_m e_m$, where $\alpha_i$ are polynomials without zero terms, is greater than $k$. Thus, each monomial $\underline{\partial}^{\underline{l}}$ must be a linear combination of $e_i$'s with {\it constant} coefficients. Thus, making a linear change, we can replace some of generators by these monomials. On the other hand, all such monomials are obviously linear independent over $K$. 

Therefore, all elements $\underline{\partial}^{\underline{j}}\diamond \sigma (\bar{S})$ are linearly independent, and $\bar{S}$ is regular.

\end{proof}

\begin{cor}
\label{C:linear_changes}
Any invertible linear change of variables $\partial_i \mapsto \sum_{j=1}^{n}c_{i,j}\partial_j$, $c_{i,j}\in K$, $i=1, \ldots ,n$ can be obtained by conjugation with an operator $S\in \hat{U}_n^{sym}$. More precisely, 
\begin{equation}
\label{E:special_form}
S=c_0 \exp (\sum_{i=1}^{n}(\sum_{j=1}^{n}(c_{i,j}-\delta_{i,j})x_j)*\partial_i)) \in \hat{D}_n^{sym},
\end{equation}
where  $c_0, c_{i,j}\in K$, $c_0\neq 0$, and $\delta_{i,j}$ is the Kronecker delta. In particular, $\Ord (S)=\Ord (S^{-1})=0$, and 
$$
S^{-1}\partial_iS=\sum_{j=1}^{n}c_{i,j}\partial_j,  \mbox{\quad} x_i\mapsto \sum_{j=1}^{n}c_{i,j}'x_j,
$$
where $(c_{i,j}')$ denotes the inverse matrix of $(c_{i,j})$.
\end{cor}

\begin{proof}
The proof immediately follows from theorem \ref{T:Schur,n}. If $(c_{i,j}')$ denotes the inverse matrix of $(c_{i,j})$, then it sends $x_i \mapsto \sum_{j=1}^{n-1}c_{i,j}'x_j$, $1\le i\le n-1$. 

 By theorem \ref{T:D_n_properties}, item 4) the operator $S$ gives the automorphism $x_i \mapsto \sum_{j=1}^{n-1}c_{i,j}x_j$, $1\le i\le n-1$
of the ring $\hat{R}$. In particular, it is invertible with the inverse 
$S^{-1}=\exp (\sum_{i=1}^{n}(\sum_{j=1}^{n}(c_{i,j}'-\delta_{i,j})x_j)*\partial_i))$. Note that $\Ord (S)=\Ord (S^{-1})=0$. Direct calculations show that 
$S^{-1}\partial_iS= \sum_{j=1}^{n}c_{i,j}\partial_j$ for $1\le i\le n$.
\end{proof}

\section{Classification theorems}
\label{S:classification}

Now we are ready to prove a first classification theorem - a classification of certain equivalent classes of quasi-elliptic rings in terms of Schur pairs.

\begin{defin}{(cf. \cite[Def.3.4]{Zheglov2013})}
\label{D:equivalent_rings}
The commutative quasi-elliptic rings $B_1$, $B_2\subset \hat{D}_n$ are said to be {\it equivalent in $\hat{D}_n$} if there is an invertible operator $S\in \hat{D}_n^n$  such that $B_1=SB_2S^{-1}$. 
\end{defin}

\begin{prop}
\label{P:admissible}
Any 1-admissible operator $T$ can be written in the form 
$$
T=S_2S_1,
$$ 
where $S_2\in \hat{E}_n$ is monic,  satisfies $A_1$, $\Ord (S_2)=\Ord (S_2^{-1})=\ord_n(S_2^{-1})=\ord_n(S_2)=0$, $\ord_{\Gamma}(S_2)=\ord_{\Gamma}(S_2^{-1})=(0,\ldots ,0)$, and $S_1\in \hat{D}_n^{n}$ is invertible and $\Ord (S_1)=0$. 
\end{prop}

\begin{proof}
Put $q_i'=T^{-1}\partial_iT$. From lemma \ref{L:conjugation_preserves_order} it follows that for $i=1, \ldots ,n-1$ we have $q_i'=\sum_{j=1}^{n-1}\sum_{k\ge 0} c_{i;j,k}\partial_j^k+ l.o.t.$, where l.o.t. mean terms with $\ord_n$-order less than zero and $c_{i;j,k}\in K_y$, and $q_n'=\partial_n + l.o.t.$, where l.o.t. mean terms with $\ord_n$-order less than one. Moreover, $\Ord (\sum_{j=1}^{n-1}\sum_{k\ge 0} c_{i;j,k}\partial_j^k)=1$.

 Note  that ${HT}_n(T\cdot {HT}_n(q_i')\cdot T^{-1})=\partial_i$, i.e. in particular the map 
$$
\partial_i \mapsto \partial_i':= \sum_{j=1}^{n-1}\sum_{k\ge 1} c_{i;j,k}\partial_j^k, \quad 1\le i\le n-1
$$
is invertible. Since $\Ord (\partial_i')=1$, and ${HT}_n(T\partial_i'T^{-1})=\partial_i\notin \idm$, we must have $\overline{\partial_i'}\neq 0$, and the map 
$$
\partial_i \mapsto \overline{\partial_i'}:= \sum_{j=1}^{n-1} c_{i;j}\partial_j, \quad 1\le i\le n-1 \quad c_{i;j}\in K
$$
is invertible too.
In particular, the spectral module $F_{n-1}=K[\partial_1, \ldots ,\partial_{n-1}]$ of the ring $K[\overline{\partial_1'}, \ldots , \overline{\partial_{n-1}'}]$ is Cohen-Macaulay. Then by theorem \ref{T:Schur,n} there exists an invertible operator $S_0\in \hat{D}_{n-1}^{sym}\subset \hat{D}_n^n$ with $\Ord (S_0)=0$ such that $S_0\partial_i' S_0^{-1}=\partial_i$ for all $i=1, \ldots ,n-1$.

Put $T_1=TS_0$. Then the operators $T_1^{-1}\partial_iT_1$ are monic formally quasi-elliptic. Then by lemma \ref{L:lemma7}  there exists an operator 
$S_1=c_0 \exp (\sum_{i=1}^{n} e_{i}x_i) \exp (\sum_{i=1}^{n-1}d_{i,n}x_{n}\partial_i)$ (see the proof of lemma \ref{L:lemma7}) such that the operators $S_1^{-1}T_1^{-1}\partial_iT_1S_1$ are normalized. 
Note that again $\Ord (S_1)=\Ord (S_1^{-1})=0$. 
And by theorem \ref{T:lemma8} there exists an operator $S_2\in \hat{E}_n$ such that $\Ord (S_2)=|\ord_{\Gamma}(S_2)|=\Ord (S_2^{-1})=|\ord_{\Gamma}(S_2^{-1})|=0$ , $S_2, S_2^{-1}$ satisfies $A_1$ and $S_2^{-1}S_1^{-1}T_1^{-1}\partial_iT_1S_1S_2=\partial_i$ for $1\le i\le n$. 

Therefore, $T':=T_1S_1S_2\in K_y[[\partial_1, \ldots , \partial_{n-1}]]((\partial_n^{-1}))\cap \Pi_n$. Since all operators in this product are of order zero and $\ord_n(T_1S_1S_2)=\ord_n(T)+\ord_n(S_0)+\ord_n(S_1)+\ord_n(S_2)=0$, it follows that $\Ord (T')=0$, where from $\Ord ((T')^{-1})=\ord_n((T')^{-1})=0$.  Thus, $T=T'S_2^{-1}S_1^{-1}S_0^{-1}$ and therefore $\Ord (T)=\Ord(T^{-1})=0$, $\ord_n(T)=\ord_n(T^{-1})=0$. The rest of the proof is straightforward.
\end{proof}

\begin{thm}{(cf. \cite[Th.3.2]{Zheglov2013})}
\label{T:schurpair} 
There is a one to one correspondence between the classes of equivalent quasi-elliptic 1-Schur pairs $(A,W)$ in $V_n$ of rank $r$ with $\Sup(W)=F$ and the classes of equivalent in $\hat{D}_n$ quasi-elliptic rings   $B\subset \hat{D}_n^{sym}$ of rank $r$.    
\end{thm}

\begin{rem}
\label{R:equivalent_rings}
In \cite[Th.3.2]{Zheglov2013} (see also \cite[Sec. 2]{KurkeZheglov}) a correspondence between slightly   narrower  classes was established: namely, the 1-admissible operators were assumed to satisfy an extra condition.  In fact, the extended equivalence classes are not much more, as it can be seen from the proof. 
\end{rem}

\begin{proof}  In one way we use the construction 1 from section \ref{S:Schur_pairs}: given a 1-quasi elliptic ring $B$ we construct a pair $(A,W)$ with the required properties. 
 If we choose any equivalent 1-quasi elliptic ring $B'$, then obviously the pair constructed by $B'$ will be equivalent to $(A,W)$. 

In another way we need to use the analogue of the Sato theory for subspaces in $V_n$ from section \ref{S:Sato_theory}. Namely, given a quasi-elliptic 1-Schur pair $(A,W)$, we construct the corresponding ring $B$ by construction 2 from section \ref{S:Schur_pairs} with the help of a Sato operator $S$. 

 If $(A',W')$ is an equivalent 1-Schur pair, then $A'=T^{-1}AT$, $W'=WT$ for some 1-admissible operator $T$, which can be written (see proposition \ref{P:admissible}) in the form $T=S_2S_1$. Then it is easy to see that 
$$W'=W\diamond (S_2S_1)= (F \diamond S_1)\diamond (S_1^{-1}SS_2S_1)=F\diamond (S_1^{-1}SS_2S_1),$$
i.e. the corresponding Sato operator for the space $W'$ from theorem \ref{T:SatoAction} is \\
$S'=S_1^{-1}S S_2S_1$. So, the corresponding ring $B'=S'A'(S')^{-1}=S_1^{-1}BS_1$, i.e. it is equivalent to $B$. 

At last, note that the composition of two constructions $B \rightsquigarrow (A,W)$, $(A,W) \rightsquigarrow B$ leads to equivalent objects. 
\end{proof} 

The theorem above can be improved for quasi-elliptic rings of rank one: as we have noted in remark \ref{R:admissible_changes} the set of admissible linear changes is not empty for such rings. 

\begin{defin}
\label{D:weak_equivalence} 
Two quasi-elliptic rings $B_1, B_2\subset \hat{D}_n$ are said to be {\it weakly equivalent} if there exits a unit $U\in \hat{U}_n^{sym}$ such that $B_1=UB_2U^{-1}$. 

Two quasi-elliptic 1-Schur pairs $(A,W)$ and $(A',W')$ are said to be {\it weakly equivalent} if $A'=\varphi_m\circ Ad(T_m)\circ \ldots \varphi_1 \circ Ad (T_1)(A)$ and $W'=\varphi_m\circ Ad(T_m)\circ \ldots \varphi_1 \circ Ad (T_1) (W)$ for some $m\in \dn$, where $T_i$ are 1-admissible operators and $\varphi_i$ are admissible linear changes of variables\footnote{Of course, $\varphi_i$ is an admissible change for the pair $(Ad(T_{i-1})\circ \ldots \varphi_1 \circ Ad (T_1)(A), Ad(T_{i-1})\circ \ldots \varphi_1 \circ Ad (T_1)(A)(W)) $}, $Ad(T)(A):=T^{-1}AT$, $Ad(T)(W):=W\diamond T$. 
\end{defin}

\begin{thm}
\label{T:schurpair1} 
Let $K$ be a complete field. Assume $K_y=K$.  

There is a one to one correspondence between the classes of weakly equivalent quasi-elliptic 1-Schur pairs $(A,W)$ in $V_n$ of rank one  with $\Sup(W)=F$ and the classes of weakly equivalent  quasi-elliptic rings  of commuting operators $B\subset \hat{D}_n^{sym}$ of rank one.    
\end{thm}

\begin{proof}
The proof is essentially the same as the proof of theorem \ref{T:schurpair}. Given two weakly equivalent rings $B_1, B_2$ we need to show that the corresponding quasi-elliptic 1-Schur pairs are weakly equivalent. 

Assume $B_1=UB_2U^{-1}$, denote by $(A_i, W_i)$  normalized 1-Schur pairs of rings $B_i$ constructed by the construction 1 of section \ref{S:Schur_pairs}. 

First, taking appropriate powers of formally quasi-elliptic elements in $A_1$ and making an appropriate  linear change of variables $\varphi'$ from lemma \ref{L:linear_changes} we can replace the 1-Schur pair $(A_1, W_1)$ with the quasi-elliptic 1-Schur pair $(A_1', W_1'):=(\varphi'(A_1), \varphi' (W_1))$ such that $\varphi'(A_1)$ contains formally quasi-elliptic operators of the same $\Ord$-order (we can find an admissible such $\varphi'$ by remark \ref{R:admissible_changes} since the Schur pairs are of rank one). 
By proposition \ref{P:admissible_changes} the quasi-elliptic ring $B_1'$ corresponding to this Schur pair by construction 2 from section \ref{S:Schur_pairs} is weakly equivalent to $B_1$. i.e. $B_1'$ is weakly equivalent to $B_2$. Obviously, it suffices to prove that $(A_1', W_1')$ is weakly equivalent to $(A_2, W_2)$, and we'll assume below that $(A_1, W_1)$ satisfies the properties of the pair $(A_1', W_1')$.

Let $B_i= S_iA_iS_i^{-1}$, $W_i=F\diamond S_i$. Denote by
$\varphi_i :W_i \rightarrow F$, $w_i\mapsto w_i\diamond S_i^{-1}$ the isomorphisms of $A_i$-modules $W_i$ with $B_i$-modules $F$ ($a_i\mapsto S_ia_iS_i^{-1}$). Then $\varphi_i^{-1}$ are given as  $\varphi_i :F \rightarrow W_i$, $w\mapsto w\diamond S_i$. At last, denote by $\psi :  F \rightarrow  F$, $w \mapsto w\diamond U^{-1}$ the isomorphism of $B_i$-modules $F$ ($b_1\mapsto Ub_1U^{-1}$). Then the composition $\tilde{\varphi}:= \varphi_2^{-1}\circ \psi \circ \varphi_1$ determines an isomorphism of $A_i$-modules $W_i$. By construction, it preserves the $\Ord$-order filtration on $W_i$ and on $A_i$. Let $a_1, \ldots ,a_n\in A_1$ be formally quasi-elliptic elements of the same $\Ord$-order, with $\ord_{\Gamma}(a_1)=(0, \ldots , 0, p)$, $\ord_{\Gamma}(a_i)=(\underline{0})-(i-1)$. Let $\tilde{a}_1$ be the monic $p$-th root of $a_1$: $\tilde{a}_1^p=a_1$. 
Applying lemma \ref{L:linear_changes} and remark \ref{R:admissible_changes}, we can find an admissible linear change of variables $\varphi$ such that $\ord_{\Gamma}(\varphi (\tilde{\varphi} (a_i)))=\ord_{\Gamma}(a_i)$, $i=1, \ldots ,n$. Applying theorem \ref{T:lemma8} and lemma \ref{L:lemma7}, we can find a 1-admissible operator $T$ such that $T\varphi (\tilde{\varphi} (a_i))T^{-1}=a_i$, $i=2, \ldots ,n$ and $T\varphi (\tilde{\varphi} (\tilde{a}_1))T^{-1}=\tilde{a}_1$. Then $T \varphi (\tilde{\varphi} (A_1))T^{-1}=A_1$. For, arguing as in lemma \ref{L:linear_changes}, for any $a\in A_1$ we have $a\tilde{a}_1^{-\Ord (a)}\in K[[a_2\tilde{a}_1^{-\Ord (a_2)}, \ldots , a_n\tilde{a}_1^{-\Ord (a_n)}, \tilde{a}_1^{-1}]]$.  Thus, $\tilde{\varphi} = \varphi^{-1} \circ Ad(T)$, i.e. $\varphi (A_2)=T^{-1}A_1T$ and $(A_2, W_2)$ is weakly equivalent to $(A_1, W_1)$. 

Conversely, two weakly equivalent quasi-elliptic 1-Schur pairs of rank one give, via the construction 2 and proposition \ref{P:admissible_changes}, two weakly equivalent quasi-elliptic rings in $\hat{D}_n^{sym}$. 
\end{proof}

\begin{rem}
If $K_y\neq K$, the theorem holds if an additional assumption is added: namely, in each equivalence class of Schur pairs corresponding to quasi-elliptic rings via construction 1, there should exist a representative pair satisfying the extra condition of remark \ref{R:admissible_changes}, item 1). 
\end{rem}

\begin{ex}
\label{Ex:Last}
Let's illustrate theorem \ref{T:schurpair1} as well as some other theorems of these paper on the explicit example calculated in \cite[Sec. 6]{BurbanZheglov2017}. 

Recall that there was a family of 1-Schur pairs $(A, W_{\beta})$ defined, where $A, W_{\beta}\in \dc [z_1,z_2]$, $A=\dc [z_1^2, z_1^3,  z_2^2, z_2^3]$, where $z_i=\partial_i$, and 
$$
W_\beta = \dc \cdot w + (\xi_2 + \xi_2^2 z_2 + \beta z_1) z_1^2 \dc [z_1] + (\xi_1 + \xi_1^2 z_1 + \beta z_2) z_2^2 \dc [z_2] + z_1^2 z_2^2 \dc [z_1, z_2],
$$
where
$
w = 1 + \xi_1 z_1 + \xi_2 z_2 + (\xi_1 \xi_2 - \beta) \cdot \left(z_1 z_2 + \left(\dfrac{z_1^2}{\xi_2^2} + \dfrac{z_2^2}{\xi_1^2}\right)\right).
$
In \cite[Th. 6.5]{BurbanZheglov2017} corresponding Sato operators $S_{\beta}$ were found, $S_{\beta}=S_0+\beta T$, where 
$$
S_0:= \partial_1 \partial_2 + \frac{1}{\xi_2 - x_2} \partial_1 + \frac{1}{\xi_1 - x_1} \partial_2 + \frac{1}{(\xi_1 - x_1)(\xi_2 - x_2)}
$$
and
\begin{equation*}
\begin{split}
T =  &   \frac{1}{(\xi_1 - x_1)(\xi_2 - x_2)}\left(\frac{1}{\xi_2}  \left(\delta_2 \partial_1 + (\xi_1-x_1) \delta_2\partial_1^2\right)  +
         \frac{1}{\xi_1}\left(\delta_1 \partial_2 + (\xi_2-x_2) \delta_1 \partial_2^2\right)\right)   + \\
           &  \frac{1}{(\xi_1 \xi_2 - \beta)(\xi_1 - x_1) (\xi_2 - x_2)} \delta_1 \delta_2\left(1 + \beta \left(\frac{\partial_1}{\xi_2} + \frac{\partial_2}{\xi_1}\right)\right),
\end{split}
\end{equation*}

It is not difficult to see that $A$ is not quasi-elliptic, but for a generic linear change of variables (obviously, all linear changes are admissible in our case) it becomes quasi-elliptic. Besides, for a generic linear change $\varphi$ we have $\Sup (\varphi (W_{\beta})z_2^{-2})=F$. By definition, all such quasi-elliptic Schur pairs $(\varphi (A), \varphi (W_{\beta})z_2^{-2})$ are weakly equivalent.  By theorem \ref{T:SatoAction}, item 3 there exists another Sato operator $S_n$ (invertible in $\hat{E}_n$) for the space $\varphi (W_{\beta})z_2^{-2}$. By theorem \ref{T:SatoAction}, item 4 we have $\varphi (S_{\beta})\partial_2^{-2}=U\circ S_n$ for some unity $U\in \hat{U}_n^{sym}$. 

By theorem \ref{T:Schur,n} the linear change $\varphi$ can be defined by conjugation with a Schur operator $S_{\varphi}\in \hat{D}_2^{sym}$ of order zero. Then the corresponding rings of commuting operators are just conjugated rings $S_{\varphi}\mathtt{L}(A)S_{\varphi}^{-1}$. By theorem \ref{T:schurpair1} they are weakly equivalent as well. 

Besides, such pairs $(\varphi (A), \varphi (W_{\beta})z_2^{-2})$ fit definition of a Schur pair (and even a pre-Schur pair of rank one) from papers \cite{Zheglov2013}, \cite[Def. 14]{Zheglov2018}, and therefore determine an algebraic-geometric pre-spectral data, see e.g. \cite[Th. 4]{Zheglov2018} (in these constructions the  Sato operator $S_n$ above was essentially used). By \cite[Th. 3]{Zheglov2018} the last data can be extended to a (reduced) geometric datum $(X, C, p, \cf_{\beta})$ of rank one, where $X$ is a projective surface (a natural compactification of $\Spec (A)$), $C$ is an integral ample Cartier divisor (by \cite[Th. 2]{Zheglov2018}), $\cf_{\beta}=\Proj (\widetilde{W_{\beta}})$ is a torsion free spectral sheaf with fixed Hilbert polynomial (see \cite[Sec. 6]{BurbanZheglov2017} or \cite[Def. 13, 14]{Zheglov2018}), and $p$ is a regular point on $C$. 

From \cite[Sec. 6]{BurbanZheglov2017} we know that the normalisation of the surface $\Spec (A)$ is just the affine plane $\da^2$. By \cite[Th.2.1.]{KOZ2014} we have $C^2=1$ and $C$ is a rational curve, and therefore by \cite[Th. 3.2]{KurkeZheglov} the normalisation of $X$ is isomorphic to the projective plane $\dpp^2$. From the explicit description of the spaces $W_{\beta}$ and from definition of the spectral sheaf $\cf_{\beta}$ it is not difficult to see that $\cf_{\beta}|_C\simeq n_*(\co_{\sdp^1})$, where $n:\dpp^1\rightarrow C$ is the normalisation morphism, iff $\beta =0$. In this case the operator $S_0$ is just a differential operator and $\varphi (S_{\beta})\partial_2^{-2}= cS_n$, $c\in \dc$. In view of \cite[Th. 7.9]{Zheglov_belovezha} this confirms the conjecture 7.11 about characterisation of commuting PDOs between all quasi-elliptic rings from \cite{Zheglov_belovezha}.

It is not difficult to see that the corresponding to such a geometric data Schur pairs from \cite[Th. 1]{Zheglov2018} will be weakly equivalent for different choices of $p$. Thus, the isomorphism class of weakly equivalent Schur pairs (or an isomorphism class $[S_{\varphi}\mathtt{L}(A)S_{\varphi}^{-1}]$ of weakly equivalent quasi-elliptic rings) determines an isomorphism class of pre-spectral datum $(X,C,\cf )$. A  detailed description of this correspondence (in a more general situation) will appear in \cite{Zheglov2020}. 
\end{ex}

%\addtocontents{toc}{References}

\vspace{0.5cm}

\noindent A. Zheglov,  Lomonosov Moscow State  University, faculty
of mechanics and mathematics, department of differential geometry
and applications, Leninskie gory, GSP, Moscow, \nopagebreak 119899,
Russia
\\ \noindent e-mail
 $azheglov@math.msu.su$, $alexander.zheglov@math.msu.ru$, $abzv24@mail.ru$

\end{document}